\documentclass[a4paper,11pt,reqno]{amsart}
\usepackage[utf8]{inputenc}
\usepackage[T1]{fontenc}
\usepackage{lmodern}
\usepackage[english]{babel}
\usepackage{amsmath,a4wide}
\usepackage{xfrac}
\usepackage{esint}
\usepackage{stmaryrd,mathrsfs,bm,amsthm,mathtools,yfonts,amssymb,color,braket,booktabs,graphicx,graphics,amsfonts}
\usepackage{latexsym,microtype,indentfirst,hyperref}
\usepackage{xcolor}
\usepackage{courier}
\usepackage[colorinlistoftodos]{todonotes}

\begingroup
\newtheorem{theorem}{Theorem}[section]
\newtheorem{lemma}[theorem]{Lemma}
\newtheorem{proposition}[theorem]{Proposition}
\newtheorem{corollary}[theorem]{Corollary}
\endgroup

\theoremstyle{definition}
\newtheorem{definition}[theorem]{Definition}
\newtheorem{remark}[theorem]{Remark}
\newtheorem{ipotesi}[theorem]{Assumption}

\numberwithin{equation}{section}
\numberwithin{subsection}{section}

\newcommand{\Na}{\mathbb{N}} 
\newcommand{\Z}{\mathbb{Z}} 
\newcommand{\R}{\mathbb{R}} 
\newcommand{\C}{\mathbb{C}} 
\newcommand{\bbD}{\mathbb{D}} 
\newcommand{\Sf}{\mathbb{S}} 

\newcommand{\M}{\mathcal{M}} 
\newcommand{\T}{\mathcal{T}} 
\newcommand{\N}{\mathcal{N}} 

\newcommand{\K}{\mathcal{K}} 
\newcommand{\p}{\mathbf{p}} 
\newcommand{\U}{\mathbf{U}} 
\newcommand{\Rm}{\mathrm{Rm}} 
\newcommand{\Ric}{\mathrm{Ric}} 
\newcommand{\Ri}{\mathcal{R}} 

\newcommand{\A}{\mathcal{A}} 
\newcommand{\G}{\mathcal{G}} 
\newcommand{\bfeta}{\boldsymbol{\eta}} 
\newcommand{\Dir}{\mathrm{Dir}} 
\newcommand{\Jac}{\mathrm{Jac}} 
\newcommand{\reg}{\mathrm{reg}} 
\newcommand{\sing}{\mathrm{sing}} 

\newcommand{\mass}{\mathbb{M}} 

\newcommand{\Ha}{\mathcal{H}} 
\newcommand{\dHa}{{\rm d}\mathcal{H}}

\newcommand{\spt}{\mathrm{spt}} 
\newcommand{\dist}{\mathrm{dist}} 
\newcommand{\tr}{\mathrm{tr}} 
\newcommand{\di}{\mathrm{div}} 
\newcommand{\diam}{\mathrm{diam}} 
\newcommand{\B}{\mathbf{B}} 
\newcommand{\ddist}{\mathbf{dist}} 
\newcommand{\Lip}{\mathrm{Lip}} 

\def\XXint#1#2#3{{\setbox0=\hbox{$#1{#2#3}{\int}$ }
\vcenter{\hbox{$#2#3$ }}\kern-.6\wd0}}

\newcommand{\weak}{\rightharpoonup} 

\newcommand{\abs}[1]{\lvert#1\rvert} 
\newcommand{\Abs}[1]{\left\lvert#1\right\rvert} 

\DeclareMathOperator*{\esssup}{ess\,sup} 

\title{Multiple valued Jacobi fields}

\author{Salvatore Stuvard}

\newcommand{\Addresses}{{
  \bigskip
  \footnotesize

  S. S., \textsc{Universit\"at Z\"urich, Winterthurerstrasse 190, CH-8057 Z\"urich, Switzerland}
  \par\nopagebreak
  
\bigskip 
 
  \textit{E-mail address}, S. S.: \href{salvatore.stuvard@math.uzh.ch}{salvatore.stuvard@math.uzh.ch}
   
}}

\begin{document}

\begin{abstract}
We develop a multivalued theory for the stability operator of (a constant multiple of) a minimally immersed submanifold $\Sigma$ of a Riemannian manifold $\M$. We define the multiple valued counterpart of the classical Jacobi fields as the minimizers of the second variation functional defined on a Sobolev space of multiple valued sections of the normal bundle of $\Sigma$ in $\M$, and we study existence and regularity of such minimizers. Finally, we prove that any $Q$-valued Jacobi field can be written as the superposition of $Q$ classical Jacobi fields everywhere except for a relatively closed singular set having codimension at least two in the domain.

\vspace{4pt}

\vspace{4pt}

\end{abstract}

\maketitle

\tableofcontents

\section{Introduction} \label{sec:intro}

Given an $m$-dimensional area minimizing integer rectifiable current $T$ in $\R^{m+n}$ and any point $x \in \spt(T) \setminus \spt(\partial T)$, it is a by now well known consequence of the monotonicity of the function $r \mapsto \frac{\| T \|(B_{r}(x))}{\omega_{m} r^{m}}$ (cf. \cite[Section 5]{Allard72}) that for any sequence of radii $\{ r_{j} \}_{j=1}^{\infty}$ with $r_{j} \downarrow 0$ there exists a subsequence $r_{j'}$ such that the corresponding blow-ups  $T_{x,r_{j'}} := (\eta_{x,r_{j'}})_{\sharp} T$ (where $\eta_{x,r}(y) := \frac{y - x}{r}$) converge to a (locally) area minimizing $m$-dimensional current $C$ which is invariant with respect to homotheties centered at the origin: such a limit current is called a \emph{tangent cone} to $T$ at $x$. If $x$ is a regular point, and thus $\spt(T)$ is a classical $m$-dimensional minimal submanifold in a neighborhood of $x$, then the cone $C$ is certainly unique, and in fact $C = Q \llbracket \pi \rrbracket$, where $\pi = T_{x}(\spt(T))$ is the tangent space to $\spt(T)$ at $x$ and $Q = \Theta(\|T\|,x)$ is the $m$-dimensional density of the measure $\| T \|$ at $x$. On the other hand, singularities do occur for area minimizing currents of arbitrary codimension as soon as the dimension of the current is $m \geq 2$: indeed, by the regularity theory developed by F. Almgren in his monumental Big Regularity Paper \cite{Almgren00} and recently revisited by C. De Lellis and E. Spadaro in \cite{DLS14, DLS13b, DLS13c}, we know that area minimizing $m$-currents in $\R^{m+n}$ may exhibit a singular set of Hausdorff dimension at most $m-2$, and that this result is sharp when $n \geq 2$ (\cite{Fed65}). Now, if $x$ happens to be singular, then not only we have no information about the limit cone, but in fact it is still an open question whether in general such a limit cone is unique (that is, independent of the approximating sequence) or not. The problem of uniqueness of tangent cones at the singular points of area minimizing currents of general dimension and codimension stands still today as one of the most celebrated of the unsolved problems in Geometric Measure Theory (cf. \cite[Problem 5.2]{openGMT}), and only a few partial answers corresponding to a limited number of particular cases are available in the literature. In \cite{BW83}, B. White showed such uniqueness for two-dimensional area minimizing currents in any codimension, building on a characterization of two-dimensional area minimizing cones proved earlier on by F. Morgan in \cite{Morgan82}. In general dimension, W. Allard and F. Almgren \cite{AA81} were able to prove that uniqueness holds under some additional requirements on the limit cone. Specifically, they have the following theorem, which is valid in the larger class of stationary integral varifolds.
\begin{theorem}[{\cite{AA81}}] \label{All-Alm}
Let $T$ be an $m$-dimensional area minimizing integer rectifiable current in $\R^{m+n}$, and let $x \in \spt(T)$ be an isolated singular point. Assume that there exists a tangent cone $C$ to $T$ at $x$ satisfying the following hypotheses:
\begin{itemize}
\item[$(H1)$] $C$ is the cone over an $(m-1)$-dimensional minimal submanifold $\Sigma$ of $\Sf^{m+n-1}$, and thus $C$ has an isolated singularity at $0$ and $\Theta(\|C\|,x) = 1$ for every $x \in \spt(C) \setminus \{0\}$;

\item[$(H2)$] all normal \emph{Jacobi fields} $N$ of $\Sigma$ in $\Sf^{m+n-1}$ are integrable, that is for every normal Jacobi field $N$ there exists a one-parameter family of minimal submanifolds of $\Sf^{m+n-1}$ having velocity $N$ at $\Sigma$.
\end{itemize}

Then, $C$ is the \emph{unique} tangent cone to $T$ at $x$. Furthermore, the blow-up sequence $T_{x,r}$ converges to $C$ as $r \downarrow 0$ with rate $r^{\mu}$ for some $\mu > 0$.
\end{theorem} 

The hypotheses $(H1)$ and $(H2)$ are however quite restrictive. Allard and Almgren were able to show that $(H2)$ holds in case $\Sigma$ is the product of two lower dimensional standard spheres (of appropriate radii to ensure minimality), since in this case all normal Jacobi fields of $\Sigma$ in $\Sf^{m+n-1}$ arise from isometric motions of $\Sf^{m+n-1}$. It seems however rather unlikely that the condition can hold for any general $\Sigma$ admitting normal Jacobi fields other than those generated by rigid motions of the sphere. In \cite{Sim83}, L. Simon was able to prove Theorem \ref{All-Alm} dropping the hypothesis $(H2)$, with a quite different approach with respect to \cite{AA81} and purely PDE-based techniques. Not much has been done, instead, in the direction of removing the hypothesis $(H1)$: to our knowledge, indeed, the only result concerning the case when a tangent cone $C$ has more than one isolated singularity at the origin is contained in L. Simon's work \cite{Sim94}, where the author proves uniqueness of tangent cones to any \emph{codimension one} area minimizing $m$-current $T$ whenever one limit cone $C$ is of the form $C = C_{0} \times \R$, with $C_{0}$ a \emph{strictly stable}, \emph{strictly minimizing} $(m-1)$-dimensional cone in $\R^{m}$ with an isolated singularity at the origin, and under additional assumptions on the Jacobi fields of $C$ and on the spectrum of the Jacobi normal operator of $C_{0}$.   

However, all the results discussed above do not cover the cases when a tangent cone has higher multiplicity: it is remarkable that uniqueness is still open even under the strong assumption that \emph{all} tangent cones to an area minimizing $m$-current $T$ ($m > 2$) at an interior singular point $x$ are of the form $C = Q \llbracket \pi \rrbracket$, where $\llbracket \pi \rrbracket$ is the rectifiable current associated with an oriented $m$-dimensional linear subspace of $\R^{m+n}$ and $Q > 1$ (cf. \cite[Section I.11(2), p. 9]{Almgren00}). 

The purpose of this work is to present a \emph{multivalued} theory of the Jacobi normal operator: we believe that such a theory may facilitate the understanding of the qualitative behaviour of the area functional near a minimal submanifold \emph{with multiplicity}, and eventually lead to a generalization of Theorem \ref{All-Alm} (and neighbouring results) to relevant cases when the condition that $\Theta(\|C\|,x) = 1$ for every $x \in \spt(C) \setminus \{0\}$ fails to hold.

In our investigation, we will make use of tools and techniques coming from the theory of \emph{multiple valued functions} minimizing the Dirichlet energy, developed by Almgren in \cite{Almgren00} and revisited by De Lellis and Spadaro in \cite{DLS11a}. A quick tutorial on the theory of multiple valued functions is contained in $\S$ \ref{ssec:Q_func}, in order to ease the reading of the remaining part of the paper. As a byproduct, the theory of multiple valued Jacobi fields will show that the regularity theory for $\Dir$-minimizing $Q$-valued functions is robust enough to allow one to produce analogous regularity results for minimizers of functionals defined on Sobolev spaces of $Q$-valued functions other than the Dirichlet energy (see also \cite{DLFS11} for a discussion about general integral functionals defined on spaces of multiple valued functions and their semi-continuity properties, and \cite{Hir16,HSV} for a regularity theory for multiple valued energy minimizing maps with values into a Riemannian manifold).

\subsection{Main results}  

Let us first recall what is classically meant by Jacobi operator and Jacobi fields. Let $\Sigma$ be an $m$-dimensional compact oriented submanifold (with or without boundary) of an $(m+k)$-dimensional Riemannian manifold $\M \subset \R^{d}$, and assume that $\Sigma$ is stationary with respect to the $m$-dimensional area functional. Then, a one-parameter family of normal variations of $\Sigma$ in $\M$ can be defined by setting $\Sigma_{t} := F_{t}(\Sigma)$, where $F_{t}$ is the flow generated by a smooth cross-section $N$ of the normal bundle $\N \Sigma$ of $\Sigma$ in $\M$ which has compact support in $\Sigma$. It is known that the second variation formula corresponding to such a family of variations can be expressed in terms of an elliptic differential operator $\mathcal{L}$ defined on the space $\Gamma(\N\Sigma)$ of the cross-sections of the normal bundle. This operator, usually called the \emph{Jacobi normal operator}, is given by $\mathcal{L} = - \Delta_{\Sigma}^{\perp} - \mathscr{A} - \mathscr{R}$, where $\Delta_{\Sigma}^{\perp}$ is the Laplacian on $\N\Sigma$, and $\mathscr{A}$ and $\mathscr{R}$ are linear transformations of $\N\Sigma$ defined in terms of the second fundamental form of the immersion $\iota \colon \Sigma \to \M$ and of a partial Ricci tensor of the ambient manifold $\M$, respectively. The notions of \emph{Morse index}, \emph{stability} and \emph{Jacobi fields}, central in the analysis of the properties of the class of minimal submanifolds of a given Riemannian manifold, are all defined by means of the Jacobi normal operator and its spectral properties (see Section \ref{sec:2nd_var} for the precise definitions and for a discussion about the most relevant literature related to the topic). In particular, Jacobi fields are defined as those sections $N \in \Gamma(\N\Sigma)$ lying in the kernel of the operator $\mathcal{L}$, and thus solving the system of partial differential equations $\mathcal{L}(N) = 0$. 

In this work, we consider instead \emph{multivalued} normal variations in the following sense. Let $\Sigma$ and $\M$ be as above, and consider, for a fixed integer $Q > 1$, a Lipschitz multiple valued vector field $N \colon \Sigma \to \A_{Q}(\R^d)$ vanishing at $\partial \Sigma$ and having the form $N = \sum_{\ell=1}^{Q} \llbracket N^{\ell} \rrbracket$, where $N^{\ell}(x)$ is tangent to $\M$ and orthogonal to $\Sigma$ at every point $x \in \Sigma$ and for every $\ell = 1,\dots,Q$. The ``flow'' of such a multiple valued vector field generates a one-parameter family $\Sigma_{t}$ of $m$-dimensional integer rectifiable currents supported on $\M$ such that $\Sigma_{0} = Q \llbracket \Sigma \rrbracket$ and $\partial \Sigma_{t} = Q \llbracket \partial \Sigma \rrbracket$ for every $t$. The second variation 
\[
\delta^{2}\llbracket \Sigma \rrbracket(N) := \frac{d^{2}}{dt^{2}} \mass(\Sigma_{t})\big|_{t=0},
\]
$\mass(\cdot)$ denoting the mass of a current, is a well-defined functional on the space $\Gamma_{Q}^{1,2}(\N\Sigma)$ of $Q$-valued $W^{1,2}$ sections of the normal bundle $\N\Sigma$ of $\Sigma$ in $\M$. We will denote such Jacobi functional by $\Jac$. Explicitly, the $\Jac$ functional is given by
\begin{equation} \label{Jac0}
\Jac(N,\Sigma) := \int_{\Sigma} \sum_{\ell=1}^{Q} \left( |\nabla^{\perp}N^{\ell}|^{2} - |A \cdot N^{\ell}|^{2} - \mathcal{R}(N^{\ell},N^{\ell}) \right) \, \dHa^{m},
\end{equation}
where $\nabla^{\perp}$ is the projection of the Levi-Civita connection of $\M$ onto $\N\Sigma$, $|A \cdot N^{\ell}|$ is the Hilbert-Schmidt norm of the projection of the second fundamental form of the embedding $\Sigma \hookrightarrow \M$ onto $N^{\ell}$ and $\mathcal{R}(N^\ell, N^\ell)$ is a partial Ricci tensor of the ambient manifold $\M$ in the direction of $N^\ell$ (see Section \ref{sec:2nd_var} for the precise definition of the notation used in \eqref{Jac0}).

Unlike the classical case, it is not possible to characterize the stationary maps of the $\Jac$ functional as the solutions of a certain Euler-Lagrange equation, and no PDE techniques seem available to study their regularity. Therefore, we develop a completely variational theory of multiple valued Jacobi fields. Hence, we give the following definition.

\begin{definition}\label{Jac_min}
Let $\Omega \subset \Sigma \hookrightarrow \M$ be a Lipschitz open set. A map $N \in \Gamma_{Q}^{1,2}(\N \Omega)$ is said to be a \emph{Jac-minimizer}, or a \emph{Jacobi $Q$-field} in $\Omega$, if it minimizes the Jacobi functional among all $Q$-valued $W^{1,2}$ sections of the normal bundle of $\Omega$ in $\M$ having the same trace at the boundary, that is
\begin{equation}
\Jac(N,\Omega) \leq \Jac(u,\Omega) \hspace{0.5cm} \mbox{ for all } u \in \Gamma_{Q}^{1,2}(\N \Omega) \mbox{ such that } u|_{\partial \Omega} = N|_{\partial \Omega}.
\end{equation}
\end{definition}

We are now ready to state the main theorems of this paper. They develop the theory of Jacobi $Q$-fields along three main directions, concerning \emph{existence}, \emph{regularity} and \emph{estimate of the singular set}.

\begin{theorem}[Conditional existence] \label{cond_ex}
Let $\Omega$ be an open and connected subset of $\Sigma \hookrightarrow \M$ with $C^{2}$ boundary. Assume that the following \emph{strict stability condition} is satisfied: the only $Q$-valued Jacobi field $N$ in $\Omega$ such that $N|_{\partial\Omega} = Q\llbracket 0 \rrbracket$ is the null field $N_{0} \equiv Q \llbracket 0 \rrbracket$. Then, for any $g \in \Gamma_{Q}^{1,2}(\N\Omega)$ such that $g|_{\partial\Omega} \in W^{1,2}(\partial\Omega,\A_{Q}(\R^d))$ there is a Jacobi $Q$-field $\overline{N}$ such that $\overline{N}|_{\partial \Omega} = g|_{\partial\Omega}$.
\end{theorem}

\begin{remark}
Note that the above result strongly resembles the classical Fredholm alternative condition for solving linear elliptic boundary value problems: the solvability of the minimum problem for the $\Jac$ functional in $\Gamma_{Q}^{1,2}(\N\Omega)$ for any given boundary datum $g$ as in the statement is guaranteed whenever $\Omega$ does not admit any non-trivial Jacobi $Q$-field vanishing at the boundary.  
\end{remark}

\begin{theorem}[Regularity] \label{qual_reg_Holder}
Let $\Omega \subset \Sigma$ be an open subset, with $\Sigma \hookrightarrow \M$ as above. There exists a universal constant $\alpha = \alpha(m,Q) \in \left( 0,1 \right)$ such that if $N \in \Gamma_{Q}^{1,2}(\N\Omega)$ is $\Jac$-minimizing then $N \in C^{0,\alpha}_{loc}(\Omega, \A_{Q}(\R^d))$. 
\end{theorem}

The statement of the next theorem requires the definition of regular and singular points of a Jacobi $Q$-field.

\begin{definition}[Regular and singular set] \label{sing_set}
Let $N \in \Gamma_{Q}^{1,2}(\N \Omega)$ be $\Jac$-minimizing. A point $p \in \Omega$ is \emph{regular} for $N$ (and we write $p \in \reg(N)$) if there exists a neighborhood $B$ of $p$ in $\Omega$ and $Q$ classical Jacobi fields $N^{\ell} \colon B \to \R^d$ such that 
\[
N(x) = \sum_{\ell=1}^{Q} \llbracket N^{\ell}(x) \rrbracket \hspace{1cm} \forall \, x \in B
\]
and either $N^{\ell} \equiv N^{\ell'}$ or $N^{\ell}(x) \neq N^{\ell'}(x)$ for all $x \in B$, for any $\ell, \ell' \in \{1,\dots,Q\}$. The \emph{singular set} of $N$ is defined by
\[
\sing(N) := \Omega \setminus \reg(N).
\]  
\end{definition}

\begin{theorem}[Estimate of the singular set] \label{sing:thm}
Let $N$ be a $Q$-valued Jacobi field in $\Omega \subset \Sigma^m$. Then, the singular set $\sing(N)$ is relatively closed in $\Omega$. Furthermore, if $m=2$, then $\sing(N)$ is at most countable; if $m \geq 3$, then the Hausdorff dimension $\dim_{\Ha}\sing(N)$ does not exceed $m-2$.
\end{theorem}

\begin{remark}
Following the approach of \cite{DLMSV16}, we expect to be able to improve Theorem \ref{sing:thm} to show, for $m \geq 3$, that $\sing(N)$ is countably $(m-2)$-rectifiable.
\end{remark}

Theorems \ref{cond_ex}, \ref{qual_reg_Holder} and \ref{sing:thm} have a counterpart in Almgren's theory of $\Dir$-minimizing multiple valued functions (cf. Theorems \ref{ex_reg} and \ref{est_sing} below). The existence result for Jacobi $Q$-field is naturally more difficult than its $\Dir$-minimizing counterpart, because in general the space of $Q$-valued $W^{1,2}$ sections of $\N\Sigma$ with bounded Jacobi energy is not weakly compact. Therefore, the proof of Theorem \ref{cond_ex} requires a suitable extension result (cf. Corollary \ref{Extension}) for multiple valued Sobolev functions defined on the boundary of an open subset of $\Sigma$ to a tubular neighborhood, which eventually allows one to exploit the strict stability condition in order to gain the desired compactness. In turn, such an extension theorem is obtained as a corollary of a multivalued version of the celebrated Luckhaus' Lemma, cf. Proposition \ref{Luckhaus}. The proof of Theorem \ref{qual_reg_Holder} is obtained from the H\"older regularity of $\Dir$-minimizing $Q$-valued functions by means of a perturbation argument. Finally, the estimate of the Hausdorff dimension of the singular set of a $\Jac$-minimizer, Theorem \ref{sing:thm}, relies on its $\Dir$-minimizing counterpart once we have shown that the tangent maps of a Jacobi $Q$-field at a collapsed singular point are non-trivial homogeneous $\Dir$-minimizing functions, see Theorem \ref{blow-up:thm}. In turn, the proof of the Blow-up Theorem \ref{blow-up:thm} is based on a delicate asymptotic analysis of an Almgren's type frequency function, which is shown to be almost monotone and bounded at every collapsed point. This is done by providing fairly general first variation integral identities satisfied by the $\Jac$-minimizers.

Let us also remark that Theorem \ref{blow-up:thm} does not guarantee that tangent maps to a Jacobi $Q$-field at a collapsed singularity are unique. Similarly to what happens for tangent cones to area minimizing currents (and for several other problems in Geometric Analysis), different blow-up sequences may converge to different limit profiles. Whether this phenomenon can actually occur or not is an open problem. On the other hand, if the dimension of the base manifold is $m = \dim \Sigma = 2$, then we are able to show that the limit profile must be a \emph{unique} non-trivial Dirichlet minimizer. Indeed, we have the following theorem.
\begin{theorem}[Uniqueness of the tangent map at collapsed singularities] \label{uniqueness_tg_map}
Let $m = \dim\Sigma =2$, and let $N$ be a $Q$-valued Jacobi field in $\Omega \subset \Sigma^2$. Let $p$ be a collapsed singular point, that is, assume that $N(p) = Q \llbracket \mathrm{v} \rrbracket$ for some $\mathrm{v} \in T_{p}^{\perp}\Sigma \subset T_{p}\M$ but there exists no neighborhood $U$ of $p$ such that $\left.N\right|_{U} \equiv Q \llbracket \zeta \rrbracket$ for some single-valued section $\zeta$. Then, there exists a \emph{unique} tangent map $\mathscr{N}_p$ to $N$ at $p$. $\mathscr{N}_p$ is a non-trivial homogeneous $\Dir$-minimizer $\mathscr{N}_p \colon T_p\Sigma \to \A_Q(T_p^\perp\Sigma)$.
\end{theorem}

The key to prove Theorem \ref{uniqueness_tg_map} is to show that, in dimension $m=2$, the rate of convergence of the frequency function at a collapsed singularity to its limit is a small power of the radius. In turn, this is achieved by exploiting one more time the variation formulae satisfied by $N$.

This note is organized as follows. In Section \ref{sec:prel} we fix the terminology and notation that will be used throughout the paper and we summarize the main results of the theory of multiple valued functions. Section \ref{sec:2nd_var} contains the derivation of the second variation formula generated by a $Q$-valued section of $\N\Sigma$ which leads to the definition of the $\Jac$ functional. In section \ref{sec:Jac_min} we investigate the first elementary properties of the $\Jac$ functional, we show that it is lower semi-continuous with respect to $W^{1,2}$ weak convergence (cf. Proposition \ref{lsc}) and we study the strict stability condition mentioned in the statement of Theorem \ref{cond_ex} (cf. Lemma \ref{str_stab_eq}). Section \ref{sec:ex} contains the proof of Theorem \ref{cond_ex}. The proof of Theorem \ref{qual_reg_Holder} (and actually of a quantitative version of it including an estimate of the $\alpha$-H\"older seminorm, cf. Theorem \ref{reg_Holder}) is contained in Section \ref{section:reg}. In Section \ref{sec:frequency} we prove the properties of the frequency function which are needed to carry on the blow-up scheme, which is instead the content of Section \ref{sec:blow-up}. Theorem \ref{sing:thm} is finally proved in Section \ref{sec:sing}. Last, Section \ref{chap:improved_dim2} contains the uniqueness of tangent maps in dimension $2$.

\subsection*{Acknowledgements} The author is warmly thankful to Camillo De Lellis for suggesting him to study this problem, and for his precious guidance and support; and to Guido De Philippis, Francesco Ghiraldin, and Luca Spolaor for several useful discussions.

The research of S.S. has been supported by the ERC grant agreement RAM (Regularity for Area Minimizing currents), ERC 306247. 

\section{Notation and preliminaries} \label{sec:prel}

\subsection{The geometric setting}

We start immediately specifying the geometric environment and fixing the notation that will be used throughout the paper.
\begin{ipotesi} \label{assumptions}
We will consider:
\begin{itemize}
\item[(M)] a closed (i.e. compact with empty boundary) Riemannian manifold $\M$ of dimension $m+k$ and class $C^{3,\beta}$ for some $\beta \in \left( 0,1 \right)$;

\item[(S)] a compact oriented \emph{minimal} submanifold $\Sigma$ of the ambient manifold $\M$ of dimension $\dim(\Sigma) = m$ and class $C^{3,\beta}$.
\end{itemize}

Without loss of generality, we will also regard $\M$ as an isometrically embedded submanifold of some Euclidean space $\R^d$. We will let $n := d - m$ and $K := d - (m+k)$ be the codimensions of $\Sigma$ and $\M$ in $\R^d$ respectively.
\end{ipotesi}

Let $\Sigma^{m} \hookrightarrow \M^{m+k} \subset \R^{d}$ be as in Assumption \ref{assumptions}. The Euclidean scalar product in $\R^{d}$ is denoted $\langle \cdot, \cdot \rangle$. The metric on $\M$ and $\Sigma$ is induced by the flat metric in $\R^{d}$: therefore, the same symbol will also denote the scalar product between tangent vectors to $\M$ or to $\Sigma$.

The tangent space to $\M$ at a point $z$ will be denoted $T_{z}\M$. The maps $\mathbf{p}^{\M}_{z} \colon \R^d \to T_{z}\M$ and $\mathbf{p}^{\M \perp}_{z} \colon \R^d \to T_{z}^{\perp}\M$ denote orthogonal projections of $\R^d$ onto the tangent space to $\M$ at $z$ and its orthogonal complement in $\R^d$ respectively. If $x \in \Sigma$, the tangent space $T_{x}\M$ can be decomposed into the direct sum
\[
T_{x}\M = T_{x}\Sigma \oplus T_{x}^{\perp}\Sigma,
\]  
where $T_{x}^{\perp}\Sigma$ is the orthogonal complement of $T_{x}\Sigma$ in $T_{x}\M$. At each point $x \in \Sigma$, we define orthogonal projections $\p_{x} \colon T_{x}\M \to T_{x}\Sigma$ and $\p_{x}^{\perp} \colon T_{x}\M \to T_{x}^{\perp}\Sigma$.

This decomposition at the level of the tangent spaces induces an orthogonal decomposition at the level of the tangent bundle, namely
\[
\T \M = \T \Sigma \oplus \N \Sigma,
\]
where $\N \Sigma$ denotes the normal bundle of $\Sigma$ in $\M$.

If $f \colon \Sigma \to \R^q$ is a $C^{1}$ map and $\xi$ is a vector field tangent to $\Sigma$, the symbol $D_{\xi}f$ will denote the directional derivative of $f$ along $\xi$, that is 
\[
D_{\xi}f(x) := \left. \frac{d}{dt} (f \circ \gamma)\right|_{t=0}
\]
whenever $\gamma = \gamma(t)$ is a $C^{1}$ curve on $\Sigma$ with $\gamma(0) = x$ and $\dot{\gamma}(0) = \xi(x)$. The differential of $f$ at $x \in \Sigma$ will be denoted $Df(x)$: we recall that this is the linear operator $Df(x) \colon T_{x}\Sigma \to \R^{q}$ such that $Df(x) \cdot \xi(x) = D_{\xi}f(x)$ for any tangent vector field $\xi$. The notation $Df|_{x}$ will sometimes be used in place of $Df(x)$. Moreover, the derivative along $\xi$ of a scalar function $f \colon \Sigma \to \R$ will be sometimes simply denoted by $\xi(f)$.

The symbol $\nabla$, instead, will identify the Levi-Civita connection on $\M$. If $\xi$ and $X$ are tangent vector fields to $\Sigma$, then for every $x \in \Sigma$ we have
\[
\nabla_{\xi}X(x) = \p_{x} \cdot \nabla_{\xi}X(x) + \p_{x}^{\perp} \cdot \nabla_{\xi}X(x) =: \nabla^{\Sigma}_{\xi}X(x) + A_{x}\left(\xi(x),X(x)\right),
\]
where $\nabla^{\Sigma}$ is the Levi-Civita connection on $\Sigma$ and $A$ is the 2-covariant tensor with values in $\N \Sigma$ defined by $A_{x}(X,Y) := \mathbf{p}_{x}^{\perp} \cdot \nabla_{X}Y$ for any $x \in \Sigma$, for any $X, Y \in T_{x}\Sigma$. $A$ is called the \emph{second fundamental form} of the embedding $\Sigma \hookrightarrow \M$ by some authors (cf. \cite[Section 7]{Simon83}, where the tensor is denoted $B$, or \cite[Chapter 8]{Lee97}, where the author uses the notation ${\rm II}$) and we will use the same terminology, although in the literature in differential geometry (above all when working with embedded hypersurfaces, that is in case the codimension of the submanifold is $k=1$) it is sometimes more customary to call $A$ ``shape operator'' and to use ``second fundamental form'' for scalar products $h(X,Y) = \langle A(X,Y), \eta \rangle$ with a fixed normal vector field $\eta$ (cf. \cite[Chapter 6, Section 2]{doC92}).

Observe that, since we have assumed $\Sigma$ to be minimal in $\M$, the mean curvature $H := \tr(A)$ is everywhere vanishing on $\Sigma$.

The curvature endomorphism of the ambient manifold $\M$ is denoted by $R$: we recall that this is a tensor field on $\M$ of type $(3,1)$, whose action on vector fields is defined by
\[
R(X,Y)Z := \nabla_{X}\nabla_{Y}Z - \nabla_{Y}\nabla_{X}Z - \nabla_{\left[X,Y\right]}Z,
\]
where $\left[X,Y\right]$ is the Lie bracket of the vector fields $X$ and $Y$.

Recall also that the Riemann tensor can be defined by setting
\[
\Rm(X,Y,Z,W) := \langle R(X,Y)Z, W \rangle
\]
for any choice of the vector fields $X,Y,Z,W$, and that the Ricci tensor is the trace of the curvature endomorphism with respect to its first and last indices, that is $\Ric(X,Y)$ is the trace of the linear map
\[
Z \mapsto R(Z,X)Y.
\]

Observe that $\Sigma$ has a natural structure of metric measure space: for any pair of points $x,y \in \Sigma$, $\mathbf{d}(x,y)$ will be their Riemannian geodesic distance, while measures and integrals will be computed with respect to the $m$-dimensional Hausdorff measure $\Ha^m$ defined in the ambient space $\R^{d}$ (note that the Hausdorff measure can be defined also intrinsically in terms of the distance ${\bf d}$: however, since $\Sigma$ is isometrically embedded in $\R^d$, the intrinsic $\Ha^m$ measure coincides with the restriction of the ``Euclidean one''). Boldface characters will always be used to denote quantities which are related to the Riemannian geodesic distance: for instance, if $x \in \Sigma$ and $r$ is a positive number, $\B_{r}(x)$ is the geodesic ball with center $x$ and radius $r$, namely the set of points $y \in \Sigma$ such that $\mathbf{d}(y,x) < r$. In the same fashion, if $U$ and $V$ are two subsets of $\Sigma$ we will set
\[
\ddist(U,V) := \inf\lbrace \mathbf{d}(x,y) \, \colon \, x \in U, y \in V \rbrace.
\] 

Finally, constants will be usually denoted by $C$. The precise value of $C$ may change from line to line throughout a computation. Moreover, we will write $C(a,b,\dots)$ or $C_{a,b,\dots}$ to specify that $C$ depends on previously introduced quantities $a,b,\dots$.  

\subsection{Multiple valued functions} \label{ssec:Q_func}

In this subsection, we briefly recall the relevant definitions and properties concerning $Q$-valued functions. First introduced by Almgren in his groundbreaking Big Regularity Paper \cite{Almgren00}, multiple valued functions have proved themselves to be a fundamental tool to tackle the problem of the interior regularity of area minimizing integral currents in codimension higher than one. The interested reader can see \cite{DLS11a} for a simple, complete and self-contained reference for Almgren's theory of multiple valued functions, \cite{DLS13a} for a nice presentation of their link with integral currents, and \cite{DL16, DeLellis2015} for a nice survey of the strategy adopted in \cite{DLS14, DLS13b, DLS13c} to revisit Almgren's program and obtain a much shorter proof of his celebrated partial regularity result for area minimizing currents in higher codimension. Other remarkable references where the theory of Dirichlet minimizing multiple valued functions plays a major role include the papers \cite{DLSS15a, DLSS15b, DLSS15c}, where the authors investigate the regularity of suitable classes of almost-minimizing two-dimensional integral currents. 


\subsubsection{The metric space of $Q$-points}

From now on, let $Q \geq 1$ be a fixed positive integer.
\begin{definition}[$Q$-points] \label{Q-points}
The space of $Q$-points in the Euclidean space $\R^{d}$ is denoted $\A_Q(\R^{d})$ and defined as follows:
\begin{equation}
\A_Q(\R^{d}) := \left\lbrace T = \sum_{\ell=1}^{Q} \llbracket p_{\ell} \rrbracket \, \colon \, p_{\ell} \in \R^{d} \mbox{ for every } \ell = 1,\dots,Q \right\rbrace,
\end{equation}
where $\llbracket p_{\ell} \rrbracket$ is the Dirac mass $\delta_{p_\ell}$ centered at the point $p_{\ell} \in \R^{d}$. Hence, every $Q$-point $T$ is in fact a purely atomic non-negative measure of mass $Q$ in $\R^d$.

For the sake of notational simplicity, we will sometimes write $\A_{Q}$ instead of $\A_{Q}(\R^{d})$ if there is no chance of ambiguity.
\end{definition}

The space $\A_{Q}(\R^{d})$ has a natural structure of complete separable metric space.
\begin{definition}\label{Distance}
If $T = \sum \llbracket p_{\ell} \rrbracket$ and $S = \sum \llbracket q_{\ell} \rrbracket$, then the distance between $T$ and $S$ is denoted $\G(T,S)$ and given by
\begin{equation}
\G(T,S)^{2} := \min_{\sigma \in \mathcal{P}_{Q}} \sum_{\ell=1}^{Q} |p_{\ell} - q_{\sigma(\ell)}|^{2},
\end{equation}
where $\mathcal{P}_{Q}$ is the group of permutations of $\{1, \dots, Q\}$.
\end{definition}

To any $T = \sum \llbracket p_{\ell} \llbracket \in \A_{Q}(\R^{d})$ we associate its center of mass $\bfeta(T) \in \R^{d}$, classically defined by:
\begin{equation} \label{eta}
\bfeta(T) := \frac{1}{Q} \sum_{\ell=1}^{Q} p_{\ell}.
\end{equation}

\subsubsection{$Q$-valued maps}

Given an open subset $\Omega \subset \Sigma$, continuous, Lipschitz, H\"older and measurable functions $u \colon \Omega \to \A_{Q}(\R^{d})$ can be straightforwardly defined taking advantage of the metric space structure of both the domain and the target. As for the spaces $L^{p}\left( \Omega, \A_{Q} \right)$, $1 \leq p \leq \infty$, they consist of those measurable maps $u \colon \Omega \to \A_{Q}(\R^{d})$ such that $\| u \|_{L^p} := \| \G(u,Q\llbracket 0 \rrbracket) \|_{L^p(\Omega)}$ is finite. We will systematically use the notation $|u| := \G(u, Q\llbracket 0 \rrbracket)$, so that
\[
\| u \|_{L^p}^{p} = \int_{\Omega} |u|^{p} \, \dHa^m
\]
for $1 \leq p < \infty$ and 
\[
\| u \|_{L^{\infty}} = \esssup_{\Omega} |u|.
\]
In spite of this notation, we remark here that, when $Q > 1$, $\A_{Q}(\R^d)$ is not a linear space: thus, in particular, the map $T \mapsto |T|$ is not a norm.

Any measurable $Q$-valued function can be thought as coming together with a measurable selection, as specified in the following proposition.
\begin{proposition}[Measurable selection, cf. {\cite[Proposition 0.4]{DLS11a}}] \label{meas_select}
Let $B \subset \Sigma$ be a $\Ha^{m}$-measurable set and $u \colon B \to \A_{Q}(\R^{d})$ be a measurable function. Then, there exist measurable functions $u_{1},\dots,u_{Q} \colon B \to \R^{d}$ such that
\begin{equation}
u(x) = \sum_{\ell=1}^{Q} \llbracket u_{\ell}(x) \rrbracket \hspace{0.3cm} \mbox{ for a.e. } x \in B.
\end{equation}
\end{proposition}

It is possible to introduce a notion of differentiability for multiple valued maps.
\begin{definition}[Differentiable $Q$-valued functions] \label{diff}
A map $u \colon \Omega \to \A_{Q}(\R^d)$ is said to be \emph{differentiable} at $x \in \Omega$ if there exist $Q$ linear maps $\lambda_{\ell} \colon T_{x}\Sigma \to \R^d$ satisfying:
\begin{itemize}
\item[$(i)$] $\G\left( u(\exp_{x}(\xi)), T_{x}u(\xi) \right) = o(|\xi|)$ as $|\xi| \to 0$ for any $\xi \in T_{x}\Sigma$, where $\exp$ is the exponential map on $\Sigma$ and 
\begin{equation}
T_{x}u(\xi) := \sum_{\ell=1}^{Q} \llbracket u_{\ell}(x) + \lambda_{\ell} \cdot \xi \rrbracket;
\end{equation}
\item[$(ii)$] $\lambda_{\ell} = \lambda_{\ell'} \mbox{ if } u_{\ell}(x) = u_{\ell'}(x)$.
\end{itemize}
\end{definition}
We will use the notation $Du_{\ell}(x)$ for $\lambda_{\ell}$, and formally set $Du(x) = \sum_{\ell} \llbracket Du_{\ell}(x) \rrbracket$: observe that one can regard $Du(x)$ as an element of $\A_{Q}(\R^{d \times m})$ as soon as a basis of $T_{x}\Sigma$ has been fixed. For any $\xi \in T_{x}\Sigma$, we define the directional derivative of $u$ along $\xi$ to be $D_{\xi}u(x) := \sum_{\ell} \llbracket Du_{\ell}(x) \cdot \xi \rrbracket$, and establish the notation $D_{\xi}u = \sum_{\ell} \llbracket D_{\xi}u_\ell \rrbracket$.

Differentiable functions enjoy a chain rule formula.
\begin{proposition}[Chain rules, cf. {\cite[Proposition 1.12]{DLS11a}}] \label{chain}
Let $u \colon \Omega \to \A_{Q}(\R^d)$ be differentiable at $x_0$.
\begin{itemize}
\item[$(i)$] Consider $\Phi \colon \tilde{\Omega} \to \Omega$ such that $\Phi(y_0) = x_0$, and assume that $\Phi$ is differentiable at $y_0$. Then, $u \circ \Phi$ is differentiable at $y_0$ and
\begin{equation} \label{chain:1}
D(u \circ \Phi)(y_0) = \sum_{\ell=1}^{Q} \llbracket Du_{\ell}(x_0) \cdot D\Phi(y_0) \rrbracket.
\end{equation}
\item[$(ii)$] Consider $\Psi \colon \Omega_{x} \times \R^{d}_{p} \to \R^{q}$ such that $\Psi$ is differentiable at the point $(x_0, u_{\ell}(x_0))$ for every $\ell$. Then, the map $\Psi(x, u) \colon x \in \Omega \mapsto \sum_{\ell=1}^{Q} \llbracket \Psi(x, u_{\ell}(x)) \rrbracket \in \A_{Q}(\R^{q})$ fulfills $(i)$ of Definition \ref{diff}. Moreover, if also $(ii)$ holds, then
\begin{equation} \label{chain:2}
D\Psi(x,u)(x_0) = \sum_{\ell=1}^{Q} \llbracket D_{x}\Psi(x_0, u_{\ell}(x_0)) + D_{p}\Psi(x_0, u_{\ell}(x_0)) \cdot Du_{\ell}(x_0) \rrbracket.
\end{equation}
\item[$(iii)$] Consider a map $F \colon (\R^d)^{Q} \to \R^{q}$ with the property that, for any choice of $Q$ points $(y_{1}, \dots, y_{Q}) \in (\R^d)^{Q}$, for any permutation $\sigma \in \mathcal{P}_{Q}$
\[
F(y_{1}, \dots, y_{Q}) = F(y_{\sigma(1)}, \dots, y_{\sigma(Q)}).
\]
Then, if $F$ is differentiable at $(u_{1}(x_0), \dots, u_{Q}(x_0))$ the composition $F \circ u$ \footnote{Observe that $F \circ u$ is a well defined function $\Omega \to \R^{q}$, because $F$ is, by hypothesis, a well defined map on the quotient $\A_{Q}(\R^d) = (\R^d)^{Q}/\mathcal{P}_{Q}$.} is differentiable at $x_0$ and
\begin{equation} \label{chain:3}
D(F \circ u)(x_0) = \sum_{\ell=1}^{Q} D_{y_\ell}F(u_{1}(x_0), \dots, u_{Q}(x_0)) \cdot Du_{\ell}(x_0).
\end{equation}
\end{itemize}
\end{proposition}

Rademacher's theorem extends to the $Q$-valued setting, as shown in \cite[Theorem 1.13]{DLS11a}: Lipschitz $Q$-valued functions are differentiable $\Ha^{m}$-almost everywhere in the sense of Definition \ref{diff}. Moreover, for a Lipschitz $Q$-valued function the decomposition result stated in Proposition \ref{meas_select} can be improved as follows.
\begin{proposition}[Lipschitz selection, cf. {\cite[Lemma 1.1]{DLS13a}}] \label{Lip-select}
Let $B \subset \Sigma$ be measurable, and assume $u \colon B \to \A_{Q}(\R^{d})$ is Lipschitz. Then, there are a countable partition of $B$ in measurable subsets $B_{i}$ ($i \in \mathbb{N}$) and Lipschitz functions $u_{i}^{\ell} \colon B_i \to \R^{d}$ ($\ell \in \{1,\dots,Q\}$) such that
\begin{itemize}
\item[$(a)$] $u|_{B_i} = \sum_{\ell=1}^{Q} \llbracket u_{i}^{\ell} \rrbracket$ for every $i \in \mathbb{N}$, and $\Lip(u_{i}^{\ell}) \leq \Lip(u)$ for every $i,\ell$;
\item[$(b)$] for every $i \in \mathbb{N}$ and $\ell,\ell' \in \{1,\dots,Q\}$, either $u_{i}^{\ell} \equiv u_{i}^{\ell'}$ or $u_{i}^{\ell}(x) \neq u_{i}^{\ell'}(x) \, \forall x \in B_i$;
\item[$(c)$] for every $i$ one has $Du(x) = \sum_{\ell=1}^{Q} \llbracket Du_{i}^{\ell}(x) \rrbracket$ for a.e. $x \in B_i$.
\end{itemize}
\end{proposition}

\subsubsection{Push-forward through multiple valued functions of $C^1$ submanifolds}

A useful fact, which will indeed be the starting point of our analysis of multivalued normal variations of $\Sigma$ in $\M$, is that it is possible to push-forward $C^1$ submanifolds of the Euclidean space through $Q$-valued Lipschitz functions. Before giving the rigorous definition of a $Q$-valued push-forward, it will be useful to introduce some further notation. We will assume the reader to be familiar with the basic concepts and notions related to the theory of currents: standard references on this topic include the textbooks \cite{Simon83} and \cite{KP08}, the monograph \cite{GMS98} and the treatise \cite{Federer69}. The space of smooth and compactly supported differential $m$-forms in $\R^{d}$ will be denoted $\mathcal{D}^{m}(\R^{d})$, and $T(\omega)$ will be the action of the $m$-current $T$ on $\omega \in \mathcal{D}^{m}(\R^{d})$. If $T$ is a current, then $\partial T$ and $\mass(T)$ are its boundary and its mass respectively. If $B \subset \R^{d}$ is $m$-rectifiable with orientation $\vec{\xi}$ and multiplicity $\theta \in L^{1}(B,\Z)$, then the integer rectifiable current $T$ associated to the triple $\left( B, \vec{\xi}, \theta \right)$ will be denoted $T = \llbracket B, \vec{\xi}, \theta \rrbracket$. In particular, if $\Sigma \subset \R^{d}$ is an $m$-dimensional $C^1$ oriented submanifold with finite $\Ha^{m}$-measure and orientation $\vec{\xi} = \xi_1 \wedge \dots \wedge \xi_m$ \footnote{That is, $\vec{\xi}(x)$ is a continuous unit $m$-vector field on $\Sigma$ with $(\xi_i)_{i=1}^{m}$ an orthonormal frame of the tangent bundle $\T \Sigma$.}, and $B \subset \Sigma$ is a measurable subset, then we will simply write $\llbracket B \rrbracket$ instead of the more cumbersome $\llbracket B, \vec{\xi}, 1 \rrbracket$ to denote the current associated to $B$. We remark that the action of $\llbracket B \rrbracket$ on a form $\omega \in \mathcal{D}^{m}(\R^{d})$ is given by
\[
\llbracket B \rrbracket(\omega) := \int_{B} \langle \omega(x), \vec{\xi}(x) \rangle \, \dHa^m(x).
\]
In particular, the $m$-current $\llbracket \Sigma \rrbracket$ is obtained by integration of $m$-forms over $\Sigma$ in the usual sense of differential geometry: $\llbracket \Sigma \rrbracket(\omega) = \int_{\Sigma} \omega$. \footnote{Observe that this convention is coherent with the use of $\llbracket p \rrbracket$, $p \in \R^d$, to denote the Dirac delta $\delta_{p}$, considered as a $0$-dimensional current in $\R^d$.}
Since we will always deal with compact manifolds, we continue to assume that $\Sigma$ is compact, in order to avoid some technicalities which are instead necessary when dealing with the non-compact case (see \cite[Definition 1.2]{DLS13a}). 
\begin{definition}[$Q$-valued push-forward, cf. {\cite[Definition 1.3]{DLS13a}}] \label{Q_pf}
Let $\Sigma$ be as above, $B \subset \Sigma$ a measurable subset and $u \colon B \to \A_{Q}(\R^d)$ a Lipschitz map. Then, the \emph{push-forward} of $B$ through $u$ is the current $\mathbf{T}_{u} := \sum_{i,\ell} (u_{i}^{\ell})_{\sharp}\llbracket B_i \rrbracket$, where $B_i$ and $u_{i}^{\ell}$ are as in Proposition \ref{Lip-select}: that is,
\begin{equation} \label{Q_pf:eq}
\mathbf{T}_{u}(\omega) := \sum_{i \in \mathbb{N}} \sum_{\ell=1}^{Q} \int_{B_i} \left\langle \omega\left(u_{i}^{\ell}(x)\right), Du_{i}^{\ell}(x)_{\sharp}\vec{\xi}(x) \right\rangle \, \dHa^m(x) \hspace{0.5cm} \forall \, \omega \in \mathcal{D}^{m}(\R^d),
\end{equation}
where $Du_{i}^{\ell}(x)_{\sharp}\vec{\xi}(x) := D_{\xi_1}u_{i}^{\ell}(x) \wedge \dots \wedge D_{\xi_m}u_{i}^{\ell}(x)$ for a.e. $x \in B_i$.
\end{definition}

It is straightforward, using the properties of the Lipschitz decomposition outlined in Proposition \ref{Lip-select} and recalling the standard theory of rectifiable currents (cf. \cite[Section 27]{Simon83}) and the area formula (cf. \cite[Section 8]{Simon83}), to conclude the following proposition.
\begin{proposition}[Representation of the push-forward, cf. {\cite[Proposition 1.4]{DLS13a}}] \label{Q_pf:thm}
The definition of the action of $\mathbf{T}_{u}$ in \eqref{Q_pf:eq} does not depend on the chosen partition $B_i$, nor on the chosen decomposition $\{u_{i}^{\ell}\}$. If $u = \sum_{\ell} \llbracket u^{\ell} \rrbracket$, we are allowed to write
\begin{equation}\label{Q_pf:eq2}
\mathbf{T}_{u}(\omega) = \int_{B} \sum_{\ell=1}^{Q} \left\langle \omega\left(u^{\ell}(x)\right), Du^{\ell}(x)_{\sharp}\vec{\xi}(x) \right\rangle \, \dHa^m(x) \hspace{0.5cm} \forall \, \omega \in \mathcal{D}^{m}(\R^d).
\end{equation}
Thus, $\mathbf{T}_{u}$ is a well-defined integer rectifiable $m$-current in $\R^d$ given by $\mathbf{T}_{u} = \llbracket \mathrm{Im}(u), \vec{\tau}, \Theta \rrbracket$, where:
\begin{itemize}
\item[$(R1)$] $\mathrm{Im}(u) = \bigcup_{x \in B} \spt(u(x)) = \bigcup_{i \in \mathbb{N}} \bigcup_{\ell=1}^{Q} u_{i}^{\ell}(B_i)$ is an $m$-rectifiable set in $\R^d$;
\item[$(R2)$] $\vec{\tau}$ is a Borel unit $m$-vector field orienting $\mathrm{Im}(u)$; moreover, for $\Ha^m$-a.e. $p \in \mathrm{Im}(u)$, we have $Du_{i}^{\ell}(x)_{\sharp}\vec{\xi}(x) \neq 0$ for every $i,\ell,x$ such that $u_{i}^{\ell}(x) = p$ and
\begin{equation}\label{n_vect}
\frac{Du_{i}^{\ell}(x)_{\sharp}\vec{\xi}(x)}{|Du_{i}^{\ell}(x)_{\sharp}\vec{\xi}(x)|} = \pm \vec{\tau}(p);
\end{equation}
\item[$(R3)$] for $\Ha^m$-a.e. $p \in \mathrm{Im}(u)$, the (Borel) multiplicity function $\Theta$ equals
\begin{equation}\label{mult}
\Theta(p) = \sum_{i,\ell,x \, \colon \, u_{i}^{\ell}(x) = p} \left\langle \vec{\tau}(p), \frac{Du_{i}^{\ell}(x)_{\sharp}\vec{\xi}(x)}{|Du_{i}^{\ell}(x)_{\sharp}\vec{\xi}(x)|}\right\rangle.
\end{equation}
\end{itemize}
\end{proposition}

\begin{remark}
The definition of a $Q$-valued push-forward can be extended to more general objects than $C^{1}$ submanifolds of the Euclidean space. Already in \cite{DLS13a} it is indeed observed that, using standard methods in measure theory, it is possible to define a multiple valued push forward of Lipschitz manifolds. Furthermore, a simple application of the polyhedral approximation theorem \cite[Theorem 4.2.22]{Federer69} allows one to actually give a definition of the $Q$-valued push-forward of any $m$-dimensional flat chain with compact support in $\R^{d}$: the interested reader can refer to our note \cite{SS17a} for the details.
\end{remark}

The next proposition is the key tool to compute explicitly the mass of the current $\mathbf{T}_{u}$. Following standard notation, we will denote by $\mathbf{J}u^{\ell}(x)$ the Jacobian determinant of $Du^{\ell}$, i.e. the number
\begin{equation}\label{Jacobian}
\mathbf{J}u^{\ell}(x) := |Du^{\ell}(x)_{\sharp}\vec{\xi}(x)| = \sqrt{\det\left((Du^{\ell}(x))^{T} \cdot Du^{\ell}(x)\right)}.
\end{equation} 
\begin{proposition}[$Q$-valued area formula, cf. {\cite[Lemma 1.9]{DLS13a}}] \label{AF}
Let $B \subset \Sigma$ be as above, and $u = \sum_{\ell} \llbracket u^{\ell} \rrbracket$ a Lipschitz $Q$-function. Then, for every Borel function $h \colon \R^d \to \left[0,\infty\right)$, we have
\begin{equation}\label{AF:eq}
\int h(p) \, d\|\mathbf{T}_{u}\|(p) \leq \int_{B} \sum_{\ell=1}^{Q} h\left(u^{\ell}(x)\right) \mathbf{J}u^{\ell}(x) \, \dHa^m(x).
\end{equation}
Equality holds in \eqref{AF:eq} if there is a set $B' \subset B$ of full $\Ha^{m}$-measure for which
\begin{equation}\label{AF:condition}
\langle Du^{\ell}(x)_{\sharp}\vec{\xi}(x), Du^{h}(y)_{\sharp}\vec{\xi}(y) \rangle \geq 0 \hspace{0.5cm} \forall \, x,y \in B' \mbox{ and } \ell,h \mbox{ with } u^{\ell}(x) = u^{h}(y).
\end{equation}
\end{proposition}

\subsubsection{$Q$-valued Sobolev functions and their properties}

Next, we study the Sobolev spaces $W^{1,p}\left(\Omega,\A_Q\right)$. The definition that we use here was proposed by C. De Lellis and E. Spadaro (cf. \cite[Definition 0.5 and Proposition 4.1]{DLS11a}), and allowed the authors to develop an alternative, intrinsic approach to the study of $Q$-valued Sobolev mappings minimizing a suitable generalization of the Dirichlet energy ($\Dir$-minimizing multiple valued maps), which does not rely on Almgren's embedding of the space $\A_{Q}(\R^d)$ in a larger Euclidean space (cf. \cite{Almgren00} and \cite[Chapter 2]{DLS11a}). Such an approach is close in spirit to the general theory of Sobolev maps taking values in abstract metric spaces and started in the works of Ambrosio \cite{Ambrosio90} and Reshetnyak \cite{Res1, Res2, Res3}.

\begin{definition}[Sobolev $Q$-valued functions] \label{Sobolev}
A measurable function $u \colon \Omega \to \A_{Q}(\R^{d})$ is in the Sobolev class $W^{1,p}$, $1 \leq p \leq \infty$ if and only if there exists a non-negative function $\psi \in L^{p}(\Omega)$ such that, for every Lipschitz function $\phi \colon \A_{Q}(\R^{d}) \to \R$, the following two properties hold:
\begin{itemize}
\item[$(i)$]  $\phi \circ u \in W^{1,p}(\Omega)$ {\footnote{{Here, the Sobolev space $W^{1,p}(\Omega)$ is classically defined as the completion of $C^{1}(\Omega)$ with respect to the $W^{1,p}$-norm
\[
\| f \|_{W^{1,p}(\Omega)}^{p} := \int_{\Omega} \left( |f(x)|^{p} + |Df(x)|^{p} \right) \, \dHa^{m}(x)
\] 
for $1 \leq p < \infty$ and
\[
\| f \|_{W^{1,\infty}(\Omega)} := \esssup_{\Omega} \left( |f(x)| + |Df(x)| \right).
\]
}}};
\item[$(ii)$] $|D(\phi \circ u)(x)| \leq \Lip(\phi) \psi(x)$ for almost every $x \in \Omega$.
\end{itemize}
\end{definition}

We also recall (cf. \cite[Proposition 4.2]{DLS11a}) that if $u \in W^{1,p}\left(\Omega,\A_{Q}(\R^{d})\right)$ and $\xi$ is a tangent vector field defined in $\Omega$, there exists a non-negative function $g_{\xi} \in L^{p}(\Omega)$ with the following two properties:
\begin{itemize}
\item[$(i)$] $|D_{\xi} \G(u,T)| \leq g_{\xi}$ a.e. in $\Omega$ for all $T \in \A_{Q}$;
\item[$(ii)$] if $h_{\xi} \in L^{p}(\Omega)$ satisfies $|D_{\xi} \G(u,T)| \leq h_{\xi}$ for all $T \in \A_{Q}$, then $g_{\xi} \leq h_{\xi}$ a.e. 
\end{itemize}
Such a function is clearly unique (up to sets of $\Ha^{m}$-measure zero), and will be denoted by $|D_{\xi}u|$. Moreover, chosen a countable dense subset $\{T_{i}\}_{i=0}^{\infty} \subset \A_{Q}$, it satisfies
\begin{equation} \label{metric der}
|D_{\xi}u| = \sup_{i \in \mathbb{N}} |D_{\xi} \G(u,T_{i})|
\end{equation}
almost everywhere in $\Omega$.

As in the classical theory, Sobolev $Q$-valued maps can be approximated by Lipschitz maps. 
\begin{proposition}[Lipschitz approximation, cf. {\cite[Proposition 4.4]{DLS11a}}] \label{LipSob_app}
Let $u$ be a function in $W^{1,p}(\Omega, \A_Q)$. For every $\lambda > 0$, there exists a Lipschitz $Q$-function $u_{\lambda}$ such that $\Lip(u_{\lambda}) \leq C \lambda$ and 
\begin{equation} \label{LipSob_est}
\Ha^{m}\left(\{x \in \Omega \, \colon \, u_{\lambda}(x) \neq u(x)\}\right) \leq \frac{C}{\lambda^p} \int_{\Omega} |Du|^p \, \dHa^m,
\end{equation}
where the constant $C$ depends only on $Q$, $m$ and $\Omega$.
\end{proposition}

As a corollary, Proposition \ref{LipSob_app} allows to prove that Sobolev $Q$-valued maps are approximately differentiable almost everywhere.
\begin{corollary}[cf. {\cite[Corollary 2.7]{DLS11a}}]\label{Sob_ap_dif}
Let $u \in W^{1,p}(\Omega , \A_Q)$. Then, $u$ is \emph{approximately differentiable} $\Ha^m$-a.e. in $\Omega$: precisely, for $\Ha^m$-a.e. $x \in \Omega$ there exists a measurable set $\tilde{\Omega} \subset \Omega$ containing $x$ such that $\tilde{\Omega}$ has density $1$ at $x$ and $u|_{\tilde{\Omega}}$ is differentiable at $x$.
\end{corollary} 

The next proposition explores the link between the metric derivative defined in \eqref{metric der} and the approximate differential of a $Q$-valued Sobolev function.
\begin{proposition}[cf. {\cite[Proposition 2.17]{DLS11a}}] \label{uniq_Dir}
Let $u$ be a map in $W^{1,2}\left( \Omega, \A_{Q}(\R^d) \right)$. Then, for any vector field $\xi$ defined in $\Omega$ and tangent to $\Sigma$ the metric derivative $|D_{\xi}u|$ defined in \eqref{metric der} satisfies
\begin{equation}
|D_{\xi}u|^{2} = \sum_{\ell=1}^{Q} |D_{\xi}u^\ell|^{2} \hspace{0.5cm} \Ha^{m}\mbox{- a.e. in }\Omega, 
\end{equation}
where $\sum_{\ell} |D_{\xi}u^\ell|^2 = \G(D_{\xi}u, Q\llbracket 0 \rrbracket)^{2}$ and $D_{\xi}u(x) \in \A_{Q}(\R^d)$ is the approximate directional derivative of $u$ along $\xi$ at the point $x \in \Omega$. In particular, we will set
\begin{equation} \label{Dir:density}
|Du|^{2}(x) := \sum_{i=1}^{m} |D_{\xi_i}u|^{2}(x) = \sum_{i=1}^{m} \sum_{\ell=1}^{Q} |D_{\xi_i}u^{\ell}|^{2}(x),
\end{equation}
with $\left( \xi_i \right)_{i=1}^{m}$ any orthonormal frame of $\T\Sigma$, at all points $x$ of approximate differentiability for $u$ in $\Omega$.
\end{proposition} 

\begin{remark} \label{Dir:independence}
Observe that the definition in \eqref{Dir:density} is independent of the choice of the frame $\left( \xi_{i} \right)$, as in fact one has
\[
|Du|^{2}(x) = \sum_{\ell=1}^{Q} |Du^{\ell}(x)|^{2},
\]
where $|Du^{\ell}(x)|$ is the Hilbert-Schmidt norm of the linear map $Du^{\ell}(x) \colon T_{x}\Sigma \to \R^{d}$ at every point of approximate differentiability for $u$.
\end{remark}

The main consequence of the above proposition is that essentially all the conclusions of the usual Sobolev space theory for single-valued functions can be recovered in the multivalued setting modulo routine modifications of the usual arguments. Some of these conclusions will be useful in the coming sections, thus we will list them here, again referring the interested reader to \cite{DLS11a} for their proofs and other useful considerations. In what follows, $\Omega \subset \Sigma$ is an open set with Lipschitz boundary.

\begin{definition}[Trace of Sobolev $Q$-functions]\label{trace}
Let $u \in W^{1,p}\left( \Omega, \A_{Q}(\R^d) \right)$. A function $g$ belonging to $L^{p}\left( \partial \Omega, \A_{Q}(\R^d) \right)$ is said to be the \emph{trace} of $u$ at $\partial \Omega$ (and we write $g = u|_{\partial \Omega}$) if for any $T \in \A_{Q}$ the trace of the real-valued Sobolev function $\G\left( u,T \right)$ coincides with $\G\left( g,T \right)$.
\end{definition}

\begin{definition}[Weak convergence] \label{weak}
Let $\{u_{h}\}_{h=1}^{\infty}$ be a sequence of maps in $W^{1,p}(\Omega,\A_{Q})$. We say that $u_{h}$ \emph{converges weakly} to $u \in W^{1,p}(\Omega,\A_{Q})$ for $h \to \infty$, and we write $u_{h} \weak u$, if
\begin{itemize}
\item[$(i)$] $\lim_{h \to \infty} \int_{\Omega} \G(u_{h},u)^{p} \, \dHa^m = 0$;
\item[$(ii)$] there exists a constant $C$ such that $\sup_{h} \int_{\Omega} |Du_{h}|^{p} \, \dHa^m \leq C$.
\end{itemize}
\end{definition}

\begin{proposition}[Weak sequential closure, cf. {\cite[Proposition 2.10, Proposition 4.5]{DLS11a}}] \label{weak_clos}
Let $u \in W^{1,p}(\Omega,\A_{Q})$. Then, there is a unique function $g \in L^{p}(\partial \Omega,\A_{Q})$ such that $g = u|_{\partial \Omega}$ in the sense of Definition \ref{trace}. Moreover, the set
\[
W^{1,p}_{g}(\Omega,\A_{Q}) := \lbrace u \in W^{1,p}(\Omega,\A_{Q}) \, \colon \, u|_{\partial \Omega} = g \rbrace
\]
is sequentially closed with respect to the notion of weak convergence introduced in Definition \ref{weak}.
\end{proposition} 

\begin{proposition}[Sobolev embeddings, cf. {\cite[Proposition 2.11, Proposition 4.6]{DLS11a}}] \label{embeddings}
The following embeddings hold:
\begin{itemize}
\item[$(i)$] if $p < m$, then $W^{1,p}(\Omega,\A_{Q}) \subset L^{q}(\Omega, \A_{Q})$ for every $q \in \left[1,p^{*}\right]$, $p^{*} := \frac{mp}{m-p}$, and the inclusion is compact when $q < p^{*}$;
\item[$(ii)$] if $p = m$, then $W^{1,p}(\Omega, \A_{Q}) \subset L^{q}(\Omega, \A_{Q})$ for all $q \in \left[1, \infty \right)$, with compact inclusion;
\item[$(iii)$] if $p > m$, then $W^{1,p}(\Omega, \A_{Q}) \subset C^{0,\alpha}(\Omega, \A_{Q})$ for all $\alpha \in \left[0, 1 - \frac{m}{p} \right]$, with compact inclusion if $\alpha < 1 - \frac{m}{p}$.
\end{itemize}
\end{proposition}

\begin{proposition}[Poincaré inequality, cf. {\cite[Proposition 2.12, Proposition 4.9]{DLS11a}}] \label{P_ineq}
Let $\Omega$ be a connected open subset of $\Sigma$ with Lipschitz boundary, and let $p < m$. There exists a constant $C = C(p,m,d,Q,\Omega)$ with the following property: for every $u \in W^{1,p}\left( \Omega , \A_{Q}(\R^d) \right)$ there exists a point $\overline{u} \in \A_{Q}(\R^d)$ such that
\begin{equation} \label{P_ineq:eq}
\left( \int_{\Omega} \G(u, \overline{u})^{p^{*}} \, \dHa^m \right)^{\frac{1}{p^{*}}} \leq C \left( \int_{\Omega} |Du|^{p} \, \dHa^m \right)^{\frac{1}{p}}.
\end{equation}
\end{proposition}

\begin{proposition}[Campanato-Morrey estimates, cf. {\cite[Proposition 2.14]{DLS11a}}] \label{CM}
Let $u$ be a $W^{1,2}(B_1,\A_{Q})$ function, with $B_1 = B_{1}(0) \subset \R^{m}$, and assume $\alpha \in \left(0, 1 \right]$ is such that
\[
\int_{B_{r}(y)} |Du|^{2} \leq A r^{m-2+2\alpha} \hspace{0.5cm} \mbox{ for every } y \in B_{1} \mbox{ and a.e. } r \in \left(0, 1 - |y| \right).
\]
Then, for every $0 < \delta < 1$ there is a constant $C = C(m,d,Q,\delta)$ such that
\begin{equation} \label{CM:eq}
\left[ u \right]_{C^{0,\alpha}(\overline{B}_{\delta})} := \sup_{x,y \in \overline{B}_{\delta}} \frac{\G\left(u(x), u(y)\right)}{|x-y|^{\alpha}} \leq C \sqrt{A}. 
\end{equation}
\end{proposition}

\subsubsection{The Dirichlet energy. $\Dir$-minimizers}

A simple corollary of Proposition \ref{uniq_Dir} and Remark \ref{Dir:independence} is that the \emph{Dirichlet energy} of a map $u \in W^{1,2}\left(\Omega,\A_{Q}(\R^{d})\right)$ can be defined in a unique way by setting
\begin{equation} \label{Dirichlet}
\Dir(u,\Omega) := \int_{\Omega} \sum_{i = 1}^{m} |D_{\xi_{i}} u|^{2} \, \dHa^{m} = \int_{\Omega} \sum_{i=1}^{m} \sum_{\ell=1}^{Q} |D_{\xi_i}u^\ell|^{2} \, \dHa^m,
\end{equation}
for any choice of a (local) orthonormal frame $\left( \xi_{1},\dots,\xi_m \right)$ of the tangent bundle of $\Sigma$.

Another interesting quantity that can be defined in our setting, the importance of which will become apparent in the sequel, is the \emph{Dirichlet energy of a tangent vector field} to the manifold $\M$.
\begin{definition}[Dirichlet energy of a tangent $Q$-field]
Let $\Omega$ be an open subset of $\Sigma \hookrightarrow \M$ as above. Let $u \in W^{1,2}\left( \Omega, \A_{Q}(\R^d) \right)$ be a Sobolev $Q$-valued tangent vector field to $\M$: that is, assume that $\spt(u(x)) \subset T_{x}\M$ for $\Ha^m$-a.e. $x \in \Omega$.
Then, for any point $x$ of approximate differentiability for $u$ in $\Omega$, and for any tangent vector field $\xi$, we set
\begin{equation} \label{app_cov}
\nabla_{\xi}u(x) := \sum_{\ell=1}^{Q} \llbracket \mathbf{p}^{\M}_{x} \cdot D_{\xi}u^{\ell}(x) \rrbracket.
\end{equation}
The Dirichlet energy of the vector field $u$ in $\Omega$ is thus given by
\begin{equation} \label{M-Dir}
\Dir^{\T\M}(u,\Omega) := \int_{\Omega} \sum_{i=1}^{m} |\nabla_{\xi_i}u|^{2} \, \dHa^{m}
\end{equation}
for any orthonormal frame $\left( \xi_1, \dots, \xi_m \right)$ of $\T\Sigma$.
\end{definition}

\begin{remark}
Observe that, when $u$ is Lipschitz continuous and $u|_{B_i} = \sum_{\ell=1}^{Q} \llbracket u_{i}^{\ell} \rrbracket$ is a local Lipschitz selection of $u$ as in Proposition \ref{Lip-select}, one has 
\[
|\nabla_{\xi}u(x)|^{2} = \sum_{\ell=1}^{Q} |\nabla_{\xi}u_{i}^{\ell}(x)|^{2} \hspace{0.5cm} \mbox{ for } \Ha^m-\mbox{a.e. } x \in B_i, \mbox{ for all vector fields } \xi,
\]
where the $\nabla$ on the right-hand side has to be intended as the classical covariant derivative (which can be extended to Lipschitz maps by means of Rademacher's theorem). 

The functional $\Dir^{\T\M}$ defined in \eqref{M-Dir} is the ``right'' geometric quantity to consider when dealing with tangent vector fields, since it does not involve any geometric structure external to the manifold $\M$. In particular, it does not depend on the isometric embedding of the Riemannian manifold $\M$ in the Euclidean space $\R^d$.
\end{remark}

As already mentioned before, a theory concerning existence and regularity properties of minimizers of the Dirichlet energy in $W^{1,2}$ (the so called $\Dir$-minimizers) has been extensively studied by Almgren in \cite{Almgren00} and revisited by De Lellis and Spadaro in \cite{DLS11a}. The theory can be summarized in three main theorems.
\begin{theorem}[Existence and H\"older regularity, cf. {\cite[Theorems 0.8 and 0.9]{DLS11a}}] \label{ex_reg}
Let $\Omega \subset \R^m$ be a bounded open subset with Lipschitz boundary. Let $g \in W^{1,2}(\Omega,\A_{Q})$. Then, there exists a function $u \in W^{1,2}(\Omega, \A_{Q})$ minimizing the Dirichlet energy \eqref{Dirichlet} among all $W^{1,2}$ $Q$-valued functions $v$ such that $v|_{\partial\Omega} = g|_{\partial\Omega}$. Furthermore, any $\Dir$-minimizer $u$ is in $C^{0,\alpha}(\Omega',\A_{Q})$ for every $\Omega' \Subset \Omega$, for some exponent $\alpha = \alpha(m,Q)$. 
\end{theorem}

The statement of the other two results requires the definition of regular and singular points of a $\Dir$-minimizer $u$.
\begin{definition}[Regular and singular points of a $\Dir$-minimizing map] \label{reg_sing}
A $Q$-valued $\Dir$-minimizer $u$ is regular at a point $x \in \Omega$ if there exist a neighborhood $B$ of $x$ in $\Omega$ and $Q$ harmonic functions $u_{\ell} \colon B \to \R^d$ such that
\[
u(y) = \sum_{\ell=1}^{Q} \llbracket u_{\ell}(y) \rrbracket \hspace{0.5cm} \mbox{ for almost every } y \in B
\]
and either $u_{\ell}(y) \neq u_{h}(y)$ for every $y \in B$ or $u_{\ell} \equiv u_{h}$. We will write $x \in \reg(u)$ if $x$ is a regular point. The complement of $\reg(u)$ in $\Omega$ is the \emph{singular set}, and will be denoted $\sing(u)$.
\end{definition}

\begin{theorem}[Estimate of the singular set, cf. {\cite[Theorem 0.11]{DLS11a}}] \label{est_sing}
Let $u$ be a $\Dir$-minimizer. Then, the Hausdorff dimension of $\sing(u)$ is at most $m - 2$. If $m = 2$, then $\sing(u)$ is at most countable.
\end{theorem}

\begin{theorem}[Improved estimate of the singular set for $m=2$, cf. {\cite[Theorem 0.12]{DLS11a}}] \label{improved}
Let $u$ be $\Dir$-minimizing, and $m=2$. Then, the singular set $\sing(u)$ consists of isolated points.
\end{theorem}

\begin{remark}
It is worth observing that here we have only discussed those results in the theory of $\Dir$-minimizing multiple valued functions which will be useful for our purposes at a later stage of this paper, and therefore our summary is far from being complete. Among the results that we have not included in the above presentation, we mention the paper \cite{Hir14}, concerned with the problem of extending the H\"older regularity in Theorem \ref{ex_reg} up to the boundary of $\Omega$, and the recent result \cite{DLMSV16}, where the authors prove that if $u$ is $\Dir$-minimizing then $\sing(u)$ is actually countably $(m-2)$-rectifiable (and hence $\Ha^{m-2}$ $\sigma$-finite), thus extending to general $Q$ a previous result obtained for $Q=2$ by Krummel and Wickramasekera in \cite{KW13} and considerably improving Almgren's original theory.
\end{remark}

\section{$Q$-valued second variation of the area functional} \label{sec:2nd_var}

Let $\M$ and $\Sigma$ be as in Assumption \ref{assumptions}. The goal of this section is to define the admissible $Q$-valued normal variations of $\Sigma$ in $\M$ and to compute the associated second variation functional. In what follows, we will denote by $\A_{Q}(\M)$ the space of $Q$-points $T = \sum_{\ell} \llbracket p_\ell \rrbracket \in \A_{Q}(\R^d)$ with each $p_\ell$ in $\M$.

\begin{definition} \label{Var_field}
An \emph{admissible} variational $Q$-field of $\Sigma$ in $\M$ is a Lipschitz map
\[
N := \sum_{\ell=1}^{Q} \llbracket N^\ell \rrbracket \colon \Sigma \to \A_{Q}(\R^d)
\]
satisfying the following assumptions:
\begin{itemize}
\item[$(H1)$] $N^{\ell}(x) \in T_{x}^{\perp}\Sigma \subset T_{x}\M$ for every $\ell \in \{1,\dots,Q\}$, for every $x \in \Sigma$;
\item[$(H2)$] $N^\ell$ vanishes in a neighborhood of $\partial \Sigma$ for every $\ell \in \{1,\dots,Q\}$.
\end{itemize}
\end{definition}

\begin{definition} \label{Variations}
Given an admissible variational $Q$-field $N$, the one-parameter family of $Q$-valued deformations of $\Sigma$ in $\M$ induced by $N$ is the map
\[
F \colon \Sigma \times \left( -\delta, \delta \right) \to \A_{Q}(\M)
\]
defined by
\begin{equation}
F(x,t) := \sum_{\ell=1}^{Q} \llbracket \exp_{x}(tN^{\ell}(x)) \rrbracket,
\end{equation}
where $\exp$ denotes the exponential map on $\M$. 
\end{definition}

Observe that, for any given $N$ as in Definition \ref{Var_field}, the induced one-parameter family of $Q$-valued deformations is always well defined for a positive $\delta$ which depends on the $L^{\infty}$ norm of $N$ and on the injectivity radius of $\M$. Note, furthermore, that $F(x,0) = Q \llbracket x \rrbracket$ for every $x \in \Sigma$, and that $F(x,t) = Q \llbracket x \rrbracket$ for all $t$ if $x \in \partial \Sigma$.

If $F$ is an admissible one-parameter family of $Q$-valued deformations, we will often write $F_{t}(x)$ instead of $F(x,t)$. Moreover, we will set $F_{t}^{\ell}(x) := \exp_{x}(t N^{\ell}(x))$.

In what follows, we will always assume to have fixed an orthonormal frame $\left( \xi_{1},\dots,\xi_{m} \right)$ of the tangent bundle $\T\Sigma$, so that $\vec{\xi} = \xi_1 \wedge \dots \wedge \xi_m$ is a continuous simple unit $m$-vector field orienting $\Sigma$. Given any admissible variational $Q$-field $N$, we can now apply the results of the previous section, and consider the push-forward of $\Sigma$ through the family $F_t$ induced by $N$. An immediate consequence of Proposition \ref{Q_pf:thm} is that the resulting object is a one-parameter family of integer rectifiable $m$-currents, denoted $\Sigma_{t} := \mathbf{T}_{F_t} = (F_{t})_{\sharp}\llbracket \Sigma \rrbracket$ with $\spt(\Sigma_{t}) \subset \M$. From \eqref{Q_pf:eq2}, we have also the explicit representation formula
\begin{equation} \label{Var_curr}
\Sigma_{t}(\omega) = \int_{\Sigma} \sum_{\ell=1}^{Q} \left\langle \omega \left( F_{t}^{\ell}(x) \right), DF_{t}^{\ell}(x)_{\sharp} \vec{\xi}(x) \right\rangle \, \dHa^{m}(x) \hspace{0.5cm} \forall \, \omega \in \mathcal{D}^{m}(\R^{d}).
\end{equation}

We will denote $\mu(t)$ the \emph{mass} $\mass(\Sigma_t)$ of the current $\Sigma_{t}$.

\begin{definition} \label{hi_ord_var}
Let $\Sigma \subset \M$, and let $N$ be an admissible variational $Q$-field. For any integer $j \geq 1$, the $j$\textsuperscript{th} order \emph{variation} of $\Sigma$ generated by $N$ is the quantity
\begin{equation} \label{hi_ord_var:eq}
\delta^{j}\llbracket \Sigma \rrbracket(N) := \left.\frac{d^{j}\mu}{dt^{j}}\right|_{t=0}.
\end{equation}
$\delta^{1}\llbracket \Sigma \rrbracket$ is usually denoted $\delta\llbracket \Sigma \rrbracket$, and called first variation. $\delta^{2}\llbracket \Sigma \rrbracket$ is called second variation.
\end{definition}

For every $j$, $\delta^{j}\llbracket \Sigma \rrbracket$ is a functional defined on the space of admissible variational $Q$-fields. In the following theorem we show that the first variation functional $\delta\llbracket \Sigma \rrbracket$ is identically zero under the assumption that $\Sigma$ is minimal in $\M$. Furthermore, and more importantly for our purposes, we provide an explicit representation formula for $\delta^{2}\llbracket \Sigma \rrbracket$.

\begin{theorem} \label{thm:2nd_var_repr}
Let $\Sigma \hookrightarrow \M$ be as in Assumption \ref{assumptions}. If $N$ is an admissible variational $Q$-field of $\Sigma$ in $\M$, then 
\begin{equation} \label{eq:1st_var}
\delta\llbracket \Sigma \rrbracket(N) = 0,
\end{equation}
and
\begin{equation} \label{eq:2nd_var_repr}
\delta^{2}\llbracket \Sigma \rrbracket(N) = \Dir^{\T\M}(N,\Sigma) - 2 \int_{\Sigma} \sum_{\ell=1}^{Q} |A \cdot N^\ell|^{2} \, \dHa^m - \int_{\Sigma} \sum_{\ell=1}^{Q} \Ri(N^\ell,N^\ell) \, \dHa^{m}, 
\end{equation}
where
\begin{equation}
|A \cdot N^{\ell}|^{2} := \sum_{i,j=1}^{m} |\langle A(\xi_i,\xi_j), N^{\ell} \rangle|^{2}
\end{equation}
and 
\begin{equation}
\Ri(N^\ell,N^\ell) := \sum_{i=1}^{m} \langle R(N^{\ell}, \xi_i)\xi_i, N^{\ell} \rangle.
\end{equation}
\end{theorem}

\begin{remark} \label{intrinsic}
Observe that formula \eqref{eq:2nd_var_repr} makes sense because the quantity on the right-hand side does not depend on the particular selection chosen for $N$, nor on the orthonormal frame chosen for the tangent bundle $\T \Sigma$.
 
The first addendum in the sum is the Dirichlet energy of the multivalued vector field $N$ on the manifold $\M$ as defined in \eqref{M-Dir}.

The second term in the sum can as well be given an intrinsic formulation, once we observe that $|A \cdot N^{\ell}|$ is the Hilbert-Schmidt norm of the symmetric bilinear form $A \cdot N^\ell \colon \T \Sigma \times \T \Sigma \to \R$ defined by $A \cdot N^\ell (\xi, \eta) := \langle A(\xi,\eta),N^\ell \rangle$.

Regarding the third term, the symmetry properties of the Riemann tensor allow to write
\[
\langle R(N^{\ell}, \xi_i)\xi_i, N^{\ell} \rangle = \langle R(\xi_i, N^{\ell}) N^{\ell}, \xi_i \rangle = \langle \p \cdot R(\xi_i,N^\ell)N^\ell, \xi_i \rangle,
\]
which in turn implies that $\Ri(N^{\ell}, N^{\ell})$ coincides with the trace of the endomorphism
\[
\xi \mapsto \p \cdot R(\xi,N^\ell)N^\ell
\]
of the tangent bundle $\T \Sigma$. In other words, this term is a partial Ricci curvature in the direction of the vector field $N^\ell$. 
\end{remark}

\begin{proof}[Proof of Theorem \ref{thm:2nd_var_repr}]
Let $N$ be an admissible variational $Q$-field of $\Sigma$ in $\M$, and let $F = F(x,t)$ denote the induced one-parameter family of $Q$-valued deformations. The proof of the representation formulae \eqref{eq:1st_var} and \eqref{eq:2nd_var_repr} will be obtained by direct computation.

The starting point is the $Q$-valued area formula, Proposition \ref{AF}, which yields an explicit formula for the function $\mu(t)$. Indeed, we may write
\begin{equation} \label{Q-area}
\mu(t) = \int_{\Sigma} \sum_{\ell=1}^{Q} \mathbf{J}F_{t}^{\ell}(x) \, \dHa^{m}(x),
\end{equation}
\emph{provided} condition \eqref{AF:condition} is satisfied: that is, provided there is a set $B \subset \Sigma$ of full measure for which
\begin{equation} \label{Condition}
\langle DF_{t}^{\ell}(x)_{\sharp}\vec{\xi}(x),DF_{t}^{\ell'}(y)_{\sharp}\vec{\xi}(y) \rangle \geq 0 \hspace{0.4cm} \forall \, x,y \in B \mbox{ and }\ell,\ell' \mbox{ with } F_{t}^{\ell}(x) = F_{t}^{\ell'}(y).
\end{equation} 
Now, it is not difficult to show that in fact condition \eqref{Condition} holds with $B = \Sigma$: to see this, first observe that since $\Sigma$ is compact there exists a number $\varepsilon > 0$ such that $\langle \vec{\xi}(x),\vec{\xi}(y) \rangle \geq \frac{1}{2}$ for all points $x,y \in \Sigma$ such that $\mathbf{d}(x,y) \leq \varepsilon$. On the other hand, the very definition of $F$ implies that for any $x \in \Sigma$ one may write $F_{t}^{\ell}(x) = x + t N^{\ell}(x) + o(t)$ for $t \to 0$. Therefore, if $|t|$ is chosen small enough, depending on $\Sigma$, $\varepsilon$ and on the $L^{\infty}$ norm of $N$ in $\Sigma$, the condition $F_{t}^{\ell}(x) = F_{t}^{\ell'}(y)$ implies $\mathbf{d}(x,y) \leq \varepsilon$ and consequently the condition $\langle \vec{\xi}(x),\vec{\xi}(y) \rangle \geq \frac{1}{2}$. But now, since $DF_{t}^{\ell}(x) = \mathrm{Id} + t DN^{\ell}(x) + o(t)$, we easily infer that $\langle DF_{t}^{\ell}(x)_{\sharp}\vec{\xi}(x),DF_{t}^{\ell'}(y)_{\sharp}\vec{\xi}(y) \rangle \geq \frac{1}{4}$ for all $x,y \in \Sigma$ and $\ell,\ell'$ with $F_{t}^{\ell}(x) = F_{t}^{\ell'}(y)$ provided $|t| \leq \delta_0$ for some $\delta_0 = \delta_0\left(\Sigma, \varepsilon, \|N\|_{L^{\infty}(\Sigma)},\Lip(N) \right)$.

Thus, we can work on each component $F^\ell$ of the decomposition of $F$ separately: in the end, we will just apply \eqref{Q-area} to obtain the desired variation formulae. Moreover, since the coming arguments are local, we will assume in what follows that the frame $\{ \xi_{i} \}_{i=1}^{m}$ is $C^{2}$ and that the selection $N = \sum_{\ell} \llbracket N^{\ell} \rrbracket$ is Lipschitz in a neighborhood of any given point $x$. 

With that being said, let us now consider a fixed value of $\ell \in \{1,\dots,Q\}$ and introduce the following quantities. For any $x$ point of differentiability for $N$ in $\Sigma$, let $Z^{\ell}(x) := \partial_{tt} F^{\ell}(x,0)$ denote the initial acceleration of the $\ell^{{\rm th}}$ sheet at the point $x$, so that the second order Taylor expansion of $F^{\ell}(x, \cdot)$ around $t = 0$ is   
\[
F^{\ell}(x,t) = x + t N^{\ell}(x) + \frac{1}{2} t^{2} Z^{\ell}(x) + o(t^{2})
\]
in a suitable $\delta$-neighborhood of $t = 0$. Then, for any $i \in \{1,\dots,m\}$, define
\begin{equation}
e^{\ell}_{i} = e^{\ell}_{i}(x,t) := D_{\xi_i}F_{t}^{\ell}(x) = \xi_{i}(x) + t D_{\xi_i}N^{\ell}(x) + \frac{1}{2} t^{2} D_{\xi_{i}}Z^{\ell}(x) + o(t^2)
\end{equation}
and
\begin{equation}
V^{\ell} = V^{\ell}(x,t) := \partial_{t}F^{\ell}(x,t).
\end{equation}
Observe that $e_{i}^{\ell}$ and $V^{\ell}$ are tangent vector fields to $\M$. 

Next, for $i,j \in \{1,\dots,m\}$ denote
\begin{equation}
{\bf g}^{\ell}_{ij} = {\bf g}^{\ell}_{ij}(x,t) := \langle e^{\ell}_{i}(x,t),e^{\ell}_{j}(x,t) \rangle
\end{equation}
and 
\begin{equation}
{\bf g}^{\ell} = {\bf g}^{\ell}(x,t) := \det({\bf g}^{\ell}_{ij}(x,t)).
\end{equation}

Using the above notation, we readily see that the Jacobian determinant $\mathbf{J}F_{t}^{\ell}$ can be written as follows:
\begin{equation}
J^{\ell} = J^{\ell}(x,t) := \mathbf{J}F_{t}^{\ell}(x) = \sqrt{{\bf g}^{\ell}(x,t)},
\end{equation}
so that, finally, the mass of the push-forwarded current is given by
\begin{equation}
\mu(t) = \sum_{\ell=1}^{Q} \mu^{\ell}(t),
\end{equation}
where 
\begin{equation}
\mu^{\ell}(t) := \int_{\Sigma} J^{\ell}(x,t) \, \dHa^{m}(x).
\end{equation}
Thus, we conclude that the first and second variation of $\Sigma$ under the deformation generated by $N$ can be represented in the following way:
\begin{equation} \label{1V}
\delta \llbracket \Sigma \rrbracket(N) = \sum_{\ell=1}^{Q} \left.\frac{d\mu^{\ell}}{dt}\right|_{t=0}
\end{equation}
and
\begin{equation} \label{2V}
\delta^{2} \llbracket \Sigma \rrbracket(N) = \sum_{\ell=1}^{Q} \left.\frac{d^{2}\mu^{\ell}}{dt^{2}}\right|_{t=0}.
\end{equation}

In what follows, in order to simplify the notation, we will drop the superscript $\ell$ when carrying on the computation.

One has:
\begin{equation}
\left.\frac{d\mu}{dt}\right|_{t=0} = \int_{\Sigma} \partial_{t}J(x,0) \, \dHa^{m}(x).
\end{equation}

Now, since
\[
\partial_{t}J = \frac{1}{2J} \partial_{t}{\bf g},
\]
and since ${\bf g}_{ij} = \delta_{ij}$ at time $t=0$, easy computations show that
\begin{equation} \label{derivative}
\partial_{t}J\big|_{t=0} = \frac{1}{2} \sum_{i=1}^{m} \partial_{t}{\bf g}_{ii}\big|_{t=0} = \sum_{i=1}^{m} \langle e_i,\partial_{t}e_i \rangle\big|_{t=0}\,,
\end{equation}
and thus
\[
\delta \llbracket \Sigma \rrbracket(N) = \int_{\Sigma} \sum_{\ell=1}^{Q} \sum_{i=1}^{m} \langle \xi_i,D_{\xi_i}N^{\ell} \rangle \, \dHa^m = \int_{\Sigma} \sum_{\ell=1}^{Q} \sum_{i=1}^{m} \langle \xi_{i}, \nabla_{\xi_i}N^{\ell} \rangle \, \dHa^m = \int_{\Sigma} \sum_{\ell=1}^{Q} \di_{\Sigma}(N^{\ell}) \, \dHa^m.
\]

In particular, recalling the definition of the map $\bfeta$ in \eqref{eta}, we deduce from the linearity of the divergence operator that
\begin{equation} \label{first_var}
\delta\llbracket \Sigma \rrbracket(N) = Q \int_{\Sigma} \di_{\Sigma}(\bfeta \circ N) \, \dHa^{m},
\end{equation}
where $\bfeta \circ N \colon \Sigma \to \R^{d}$, the ``average'' of the sheets of the vector field $N$, is a classical single-valued Lipschitz map. Note that if $N$ is single-valued then $\bfeta \circ N \equiv N$, and we recover the usual formulation of the first variation formula in terms of the divergence of the variational vector field. Observe now that the average $\bfeta \circ N$ vanishes in a neighborhood of $\partial\Sigma$ and satisfies $\bfeta \circ N(x) \in T_{x}^{\perp}\Sigma \subset T_{x}\M$ for every $x \in \Sigma$. Hence, for every $i \in \{1,\dots,m\}$ the scalar product $\langle \xi_{i}, \bfeta \circ N \rangle$ is everywhere vanishing, and we have that $\langle \xi_i,\nabla_{\xi_i} (\bfeta \circ N) \rangle = - \langle \nabla_{\xi_i}\xi_i, \bfeta \circ N \rangle = - \langle A(\xi_i,\xi_i), \bfeta \circ N \rangle$. Therefore, recalling the definition of the mean curvature vector $H$ as the trace of the second fundamental form, one can also write
\begin{equation}
\delta \llbracket \Sigma \rrbracket(N) = Q \int_{\Sigma} \sum_{i=1}^{m} \langle \xi_{i}, \nabla_{\xi_i} (\bfeta \circ N) \rangle \, \dHa^{m} =- Q \int_{\Sigma} \langle H, \bfeta \circ N \rangle \, \dHa^m = 0
\end{equation}
because $\Sigma$ is minimal in $\M$. This proves \eqref{eq:1st_var}.

Next, we go further and we compute the second variation of the mass. We first write, for every $t$ and for every $x \in \Sigma$ of differentiability for the variational field:
\[
\partial_{t}J = \frac{1}{2\sqrt{{\bf g}}} \partial_{t}{\bf g} = \frac{1}{2} J \frac{1}{{\bf g}} \partial_{t}{\bf g} = \frac{1}{2} J \partial_{t} \left( \log({\bf g}) \right) = \frac{1}{2} J {\bf g}^{ij} \partial_{t} {\bf g}_{ji},
\] 
where in the last identity we have used Jacobi's formula
\[
\partial_{t} \log \det A(t) = \tr \left( A(t)^{-1} \cdot \partial_{t}A(t) \right)
\]
for any invertible matrix $A(t)$ with positive determinant. Moreover, $\left( {\bf g}^{ij} \right)$ is the inverse matrix of $\left( {\bf g}_{ij} \right)$, and Einstein's convention on the summation of repeated indices has been used. Now, since
\[
\partial_{t}{\bf g}_{ji} = \partial_{t} \left( \langle e_{j},e_{i} \rangle \right) = \langle \partial_{t} e_{j},e_{i} \rangle + \langle e_{j}, \partial_{t} e_{i} \rangle,
\]
and using the fact that the matrix $\left( {\bf g}^{ij} \right)$ is symmetric, we can conclude the following identity:
\[
\partial_{t} J = J {\bf g}^{ij} \langle e_{i},\partial_{t}e_j \rangle.
\]

In turn, this produces:
\begin{equation} \label{sec_var}
\partial_{tt}J = \underbrace{\left(\partial_{t}J\right){\bf g}^{ij}\langle e_{i},\partial_{t}e_j \rangle}_{=: I} 
               + \underbrace{J \left(\partial_{t}{\bf g}^{ij}\right) \langle e_{i},\partial_{t}e_j \rangle}_{=: II} 
               + \underbrace{J {\bf g}^{ij} \partial_{t}\left( \langle e_{i},\partial_{t}e_j \rangle \right)}_{=: III}.
\end{equation}

Now, we evaluate equation \eqref{sec_var} at time $t = 0$. Regarding the first term in the sum, we use \eqref{derivative}, the orthonormality condition ${\bf g}^{ij}\big|_{t=0} = \delta^{ij}$ and the fact that $e_{i}\big|_{t=0} = \xi_i$, $\partial_{t}e_{i}\big|_{t=0} = D_{\xi_i}N$ (here, of course, we are writing $N$ instead of $N^\ell$) to conclude
\begin{equation}
I\big|_{t=0} = \left( \sum_{i=1}^m \langle \xi_{i},\nabla_{\xi_i}N \rangle \right)^{2}.
\end{equation}
Since $N = N^{\ell}$ is Lipschitz, and since $\langle \xi_{i}, N \rangle \equiv 0$, we have $\langle \xi_{i}, \nabla_{\xi_i}N \rangle = - \langle A(\xi_i, \xi_i), N \rangle$, and thus
\begin{equation} \label{sec_var1}
I\big|_{t=0} = \left( \langle H,N \rangle \right)^{2} = 0
\end{equation}
due to the minimality of $\Sigma$.

In order to derive a formula for $II\big|_{t=0}$, we first differentiate the identity
\[
{\bf g}^{ij}{\bf g}_{jh} = \delta^{i}_{h}
\]
to obtain that
\[
\partial_{t}{\bf g}^{ij} = - {\bf g}^{ik} \left( \partial_{t}{\bf g}_{kh} \right) {\bf g}^{hj},
\]
whence
\begin{equation}
\partial_{t}{\bf g}^{ij}\big|_{t=0} = - \partial_{t}{\bf g}_{ij}\big|_{t=0} = - \left( \langle \nabla_{\xi_i}N,\xi_j \rangle + \langle \xi_i,\nabla_{\xi_j}N \rangle \right).
\end{equation}
Since $\langle \nabla_{\xi_i}N,\xi_j \rangle = - \langle A\left(\xi_i,\xi_j\right),N \rangle$, the symmetry of the second fundamental form implies 
\begin{equation}
\partial_{t}{\bf g}^{ij}\big|_{t=0} = 2 \langle A\left( \xi_i,\xi_j \right),N \rangle.
\end{equation}
Again, since
\[
\langle e_{i},\partial_{t}e_j \rangle\big|_{t=0} = \langle \xi_i, \nabla_{\xi_j}N \rangle = - \langle A\left( \xi_i, \xi_j \right),N \rangle,
\]
we can finally obtain
\begin{equation} \label{sec_var2}
II\big|_{t=0} = - 2 \sum_{i,j=1}^{m} |\langle A(\xi_i,\xi_j),N \rangle|^{2}.
\end{equation}

Finally, we compute $III\big|_{t=0}$. The simplest way to do it is to regard the operator $\partial_t$ as the covariant derivative along the vector field $V = V^{\ell}$. One therefore has:
\[
\begin{split}
\partial_{t}\left( \langle e_i,\partial_{t}e_j \rangle \right) &= V \langle e_i, \nabla_{V} e_{j} \rangle \\
&= \langle \nabla_{V} e_i, \nabla_{V} e_j \rangle + \langle e_{i}, \nabla_{V} \nabla_{V} e_j \rangle \\
&= \langle \nabla_{e_i} V, \nabla_{e_j} V \rangle + \langle e_i, \nabla_{V} \nabla_{e_j} V \rangle,
\end{split}
\]
where in the last identity we have used the fact that the vector fields $e_i$ and $V$ commute, and, of course, that the Riemannian connection on $\M$ is torsion-free. Now, using again that $\left[ V, e_j \right] = 0$ and the definition of the curvature tensor $R$, we may write
\[
\nabla_{V} \nabla_{e_j} V = \nabla_{e_j} \nabla_{V} V + R(V,e_j)V,
\]
so that, finally, the evaluation of $\partial_{t}\left( \langle e_{i}, \partial_{t}e_{j} \rangle \right)$ at time $t = 0$ yields
\[
\partial_{t}\left( \langle e_{i}, \partial_{t}e_{j} \rangle \right)\big|_{t=0} = \langle 
\nabla_{\xi_{i}} N, \nabla_{\xi_{j}} N \rangle + \langle \xi_{i}, \nabla_{\xi_{j}} Z \rangle + \langle \xi_{i}, R(N,\xi_{j})N \rangle,
\]
with $Z = Z^{\ell}$. Since $J\big|_{t=0} = 1$ and ${\bf g}^{ij}\big|_{t=0} = \delta^{ij}$, we conclude the following identity:
\begin{equation} \label{sec_var3}
III\big|_{t=0} = \sum_{i=1}^m |\nabla_{\xi_i}N|^2 + \di_{\Sigma}Z - \sum_{i=1}^{m} \langle R(N,\xi_i)\xi_i,N \rangle.
\end{equation}

Observe that, in deriving formula \eqref{sec_var3}, we have used that $\langle R(X,Y)U, W \rangle = - \langle R(X,Y)W, U \rangle$ for any choice of $X,Y,U,W$ vector fields on $\M$.

We have now all the tools to conclude: from the $Q$-valued area formula \eqref{Q-area} it follows that
\[
\left.\frac{d^{2}\mu}{dt^2}\right|_{t=0} = \int_{\Sigma} \sum_{\ell=1}^{Q} \partial_{tt}J^{\ell}(x,0) \, \dHa^{m}(x),
\]
thus it suffices to plug equations \eqref{sec_var1}, \eqref{sec_var2}, \eqref{sec_var3} in \eqref{sec_var} to get
\begin{equation} \label{SVF}
\begin{split}
\delta^{2}\llbracket \Sigma \rrbracket(N) = &\int_{\Sigma} \sum_{\ell=1}^{Q} \left( \sum_{i=1}^{m} |\nabla_{\xi_i}N^{\ell}|^{2} - 2 \sum_{i,j=1}^{m} |\langle A(\xi_i,\xi_j), N^{\ell}\rangle|^{2} - \sum_{i=1}^{m} \langle R(N^{\ell},\xi_i)\xi_i, N^{\ell} \rangle \right) \dHa^{m} \\
&+ Q \int_{\Sigma} \di_{\Sigma}(\bfeta \circ Z) \, \dHa^{m},
\end{split}
\end{equation}
where $Z := \sum_{\ell} \llbracket Z^{\ell} \rrbracket$. Now, we decompose
\begin{equation}
\bfeta \circ Z = \p \cdot (\bfeta \circ Z) + \p^{\perp} \cdot (\bfeta \circ Z) = (\bfeta \circ Z)^{\top} + (\bfeta \circ Z)^{\perp},
\end{equation}
and we see that, since $\langle \xi_i, (\bfeta \circ Z)^{\perp} \rangle = 0$ for all $i$,
\begin{equation} \label{null_tang}
\di_{\Sigma}((\bfeta \circ Z)^{\perp}) = \sum_{i=1}^{m} \langle \xi_i, \nabla_{\xi_i}(\bfeta \circ Z)^{\perp} \rangle = - \sum_{i=1}^{m} \langle A(\xi_i,\xi_i), \bfeta \circ Z \rangle = - \langle H, \bfeta \circ Z \rangle = 0.
\end{equation}
On the other hand, Stokes' theorem and the fact that $N$ is vanishing in a neighborhood of $\partial \Sigma$ readily imply that
\begin{equation} \label{null_ort}
\int_{\Sigma} \di_{\Sigma}((\bfeta \circ Z)^{\top}) \, \dHa^{m} = 0,
\end{equation}
and thus the last addendum on the right-hand side of \eqref{SVF} vanishes. This completes the proof of formula \eqref{eq:2nd_var_repr}.
\end{proof}

We note now that the quantity appearing on the right-hand side of formula \eqref{SVF} is in fact well defined for any $Q$-valued vector field tangent to $\M$ and belonging to the class $W^{1,2}\left( \Sigma, \A_{Q}(\R^{d}) \right)$. This motivates the following definitions.

\begin{definition}[$W^{1,2}$ sections of the normal bundle] \label{sections}
Let $\Sigma \hookrightarrow \M$ be as above, and let $\Omega \subset \Sigma$ be open. We define the class of $W^{1,2}$ \emph{sections} of the normal bundle of $\Omega$ in $\M$, denoted $\Gamma_{Q}^{1,2}(\N \Omega)$, as follows: 
\begin{equation} \label{sections:eq}
\Gamma_{Q}^{1,2}(\N \Omega) := \left\lbrace N \in W^{1,2}\left(\Omega,\A_{Q}(\R^d)\right) \, \colon \, \spt(N(x)) \subset T_{x}^{\perp}\Sigma \subset T_{x}\M \mbox{ for } \Ha^{m}\mbox{-a.e. } x \in \Omega \right\rbrace.
\end{equation}
\end{definition}

\begin{definition}[Jacobi functional]\label{Jacobi}
For a section $N \in \Gamma_{Q}^{1,2}(\N \Omega)$, the \emph{Jacobi functional}, or \emph{stability functional}, is defined by:
\begin{equation} \label{Jacobi:eq}
\Jac(N,\Omega) := \Dir^{\T\M}(N,\Omega) - 2 \int_{\Omega} \sum_{\ell=1}^{Q} |A \cdot N^\ell|^{2} \, \dHa^m - \int_{\Omega} \sum_{\ell=1}^{Q} \Ri(N^\ell,N^\ell) \, \dHa^{m}.
\end{equation}
\end{definition}

Our first observation is that the classical theory of the Jacobi normal operator can be recovered within the above framework by simply setting $Q = 1$.

\begin{remark} \label{rmk:1-valued}
Consider the classical single-valued setting, corresponding to $Q = 1$, let $\Omega = \Sigma$ and recall that
\[
\Dir^{\T\M}(N,\Sigma) = \int_{\Sigma} \sum_{i=1}^{m} |\nabla_{\xi_i}N|^{2} \, \dHa^m
\] 
for any orthonormal frame $(\xi_1, \dots, \xi_m)$ of $\T \Sigma$. Assume also that $N$ is Lipschitz continuous for convenience. Let $(\nu_1, \dots, \nu_k)$ be local sections of the normal bundle $\N \Sigma$ of $\Sigma$ in $\M$ such that, at each point $x \in \Sigma$, the system $\left( (\xi_{j}(x))_{j=1}^{m}, (\nu_{\alpha}(x))_{\alpha=1}^{k} \right)$ is an orthonormal basis of $T_{x}\M$. Then, for every point of differentiability for $N$ and for every $i = 1,\dots,m$ we have:
\[
|\nabla_{\xi_i}N|^2 = \sum_{j=1}^{m} |\langle \nabla_{\xi_i}N, \xi_j \rangle|^{2} + \sum_{\alpha=1}^{k} |\langle \nabla_{\xi_i}N, \nu_{\alpha} \rangle|^{2}.
\]
Now, the usual considerations about the orthogonality of $N$ and $\xi_j$ imply that $\langle \nabla_{\xi_i}N, \xi_j \rangle = - \langle A(\xi_i, \xi_j), N \rangle$. We therefore obtain that
\[
\int_{\Sigma} \left( \sum_{i=1}^{m} |\nabla_{\xi_i}N|^{2} - \sum_{i,j=1}^{m} |\langle A(\xi_i,\xi_j), N\rangle|^{2} \right) \, \dHa^m = \int_{\Sigma} \sum_{i=1}^{m} \sum_{\alpha=1}^{k} |\langle \nabla_{\xi_i}N, \nu_{\alpha} \rangle|^{2} \, \dHa^m,
\]
and finally conclude the identity
\begin{equation} \label{cl_Jac}
\Jac(N,\Sigma) = \int_{\Sigma} \left( \sum_{i=1}^{m} \sum_{\alpha=1}^{k} |\langle \nabla_{\xi_i}N, \nu_{\alpha} \rangle|^{2} - \sum_{i,j=1}^{m} |\langle A(\xi_i,\xi_j), N\rangle|^{2} - \sum_{i=1}^{m} \langle R(N,\xi_i)\xi_i, N\rangle \right) \, \dHa^m.
\end{equation}

It is immediately seen that the Euler-Lagrange operator associated to the second variation functional \eqref{cl_Jac} is the linear elliptic operator $\mathcal{L}$ defined on the space of sections of $\N \Sigma$ and given by
\begin{equation}
\mathcal{L} := - \Delta_{\Sigma}^{\perp} - \mathscr{A} - \mathscr{R},
\end{equation}
where $\Delta_{\Sigma}^{\perp}$ is the Laplacian on the normal bundle of $\Sigma$, $\mathscr{A}$ is Simons' operator, defined by
\begin{equation}
\mathscr{A}(N) := \sum_{i,j=1}^{m} \langle A(\xi_i,\xi_j),N \rangle A(\xi_i,\xi_j),
\end{equation}
and $\mathscr{R}$ is given by
\begin{equation}
\mathscr{R}(N) := \sum_{i=1}^{m} \mathbf{p}^{\perp} \cdot R(N,\xi_i)\xi_i.
\end{equation}
As already anticipated in the Introduction, the operator $\mathcal{L}$ is classically called \emph{Jacobi normal operator}, and the solutions of the differential equation $\mathcal{L}(N) = 0$ (that is, the normal vector fields that are in its kernel) are known in the literature as \emph{Jacobi fields}. The importance of studying the second variation operator of minimal submanifolds into Riemannian manifolds is well justified by the arguments given earlier on in this section: in the single valued case $Q = 1$, the Jacobi operator $\mathcal{L}$ carries the information about the stability properties of the submanifold itself, when it is thought of as a stationary point for the $m$-dimensional volume. In particular, non-trivial Jacobi fields vanishing on $\partial \Sigma$ are, when they exist, the infinitesimal normal deformations of $\Sigma$ which preserve the volume up to second order. From a functional analytic point of view, $\mathcal{L}$ is a second-order strongly elliptic operator. When diagonalized on the space of sections of $\N\Sigma$ vanishing on $\partial\Sigma$ with respect to the standard inner product, it exhibits distinct, real eigenvalues $\{ \lambda_{h} \}_{h=1}^{\infty}$ (counted with multiplicities) such that
\[
\lambda_{1} < \lambda_{2} < \dots < \lambda_{h} < \dots \to + \infty.
\]   
Moreover, the dimension of each eigenspace is finite. The sum of the dimensions of the eigenspaces corresponding to negative eigenvalues is called the \emph{Morse index} of $\Sigma$: it accounts for the number of independent infinitesimal normal deformations of $\Sigma$ which do decrease the volume at second order. If $\lambda = 0$ is an eigenvalue, then the dimension of ${\rm ker}(\mathcal{L})$ is called \emph{nullity}. We recall that $\Sigma$ is called \emph{stable} if its Morse index is $0$, and \emph{strictly stable} if there exist no non-trivial Jacobi fields vanishing at the boundary, i.e. if ${\rm index}(\Sigma) + {\rm nullity}(\Sigma) = 0$.

A systematic study of the Jacobi normal operator was initiated by J. Simons in \cite{Simons68}. One of Simons' main results was to prove that if $\M = \Sf^{m+1}$ and $\Sigma^{m}$ is a closed minimal hypersurface immersed in $\Sf^{m+1}$ which is not totally geodesic then the first eigenvalue of the operator $\mathcal{L}$ satisfies the upper bound $\lambda_{1} \leq -m$. As a consequence of this, he was able to show that no non-trivial stable minimal hypercones exist in $\R^{m+1}$ for $m \leq 6$. In turn, this led to the proof of the Bernstein conjecture, stating that the only entire solutions $u \colon \R^{m} \to \R$ of the minimal surface equation are linear, for every $m \leq 7$. The result is sharp, as the Bernstein conjecture was proved to be false for $m > 7$ by E. Bombieri, E. De Giorgi and E. Giusti in \cite{BDGG}.
\end{remark}

The considerations leading to formula \eqref{cl_Jac} can be repeated in the $Q$-valued setting, thus showing that the Definition \ref{Jacobi} of the Jacobi functional agrees with the one given in formula \eqref{Jac0}. This equivalence is recorded in Lemma \ref{lem:equivalence} below. We first need a definition.

\begin{definition}[Normal Dirichlet energy of a section] \label{def:normal_dirichlet}
Let $N \in \Gamma_{Q}^{1,2}(\N\Omega)$. For any point $x \in \Omega$ where $N$ is approximately differentiable, and for any tangent vector field $\xi$, set
\begin{equation}
\nabla^{\perp}_{\xi}N(x) := \sum_{\ell=1}^{Q} \llbracket \p_{x}^{\Sigma\perp} \cdot D_{\xi}N^{\ell}(x) \rrbracket, 
\end{equation}
where $\p_{x}^{\Sigma\perp} = \p_{x}^{\perp} \circ \p_{x}^{\M}$ is the orthogonal projection of $\R^d$ onto $T_{x}^{\perp}\Sigma$. Then, the normal Dirichlet energy of $N$ in $\Omega$ is the quantity
\begin{equation} \label{Normal_Dir}
\Dir^{\N\Sigma}(N,\Omega) := \int_{\Omega} \sum_{i=1}^{m} |\nabla^{\perp}_{\xi_i}N|^{2} \, \dHa^m,
\end{equation}
for any choice of a (local) orthonormal frame $\{ \xi_{i} \}_{i=1}^{m}$ of $\T\Sigma$.
\end{definition}

\begin{lemma}[Equivalence of the definitions of the $\Jac$ functional] \label{lem:equivalence}
For any $N = \sum_{\ell} \llbracket N^{\ell} \rrbracket \in \Gamma_{Q}^{1,2}(\N\Omega)$ it holds
\begin{equation} \label{clean_Q_Jac:eq}
\begin{split}
&\Jac(N,\Omega) = \Dir^{\N\Sigma}(N,\Omega) - \int_{\Omega} \sum_{\ell=1}^{Q} |A \cdot N^{\ell} |^{2} \, \dHa^m - \int_{\Omega} \sum_{\ell=1}^{Q} \Ri(N^\ell,N^\ell) \, \dHa^m \\
&= \int_{\Omega} \sum_{\ell=1}^{Q} \left( \sum_{i=1}^{m} \sum_{\alpha=1}^{k} |\langle D_{\xi_i}N^{\ell}, \nu_{\alpha} \rangle|^{2} - \sum_{i,j=1}^{m} |\langle A(\xi_i,\xi_j), N^{\ell} \rangle|^{2} - \sum_{i=1}^{m} \langle R(\xi_i,N^\ell)N^\ell, \xi_i \rangle \right)\dHa^m\, ,
\end{split} 
\end{equation}
where $\{ \xi_{i} \}_{i=1}^{m}$ and $\{ \nu_{\alpha} \}_{\alpha=1}^{k}$ are (local) orthonormal frames of $\T\Sigma$ and $\N\Sigma$ respectively.
\end{lemma}

\begin{proof}
It is a straightforward consequence of the arguments contained in Remark \ref{rmk:1-valued} combined with the Lipschitz approximation theorem, Proposition \ref{LipSob_app} (cf. also \cite[Lemma 4.5]{DLS14}) and the Lipschitz selection property in Proposition \ref{Lip-select}.
\end{proof}

On the other hand, unlike the single-valued case, the lack of linear structure of $\Gamma_{Q}^{1,2}(\N\Omega)$ in the multivalued case $Q > 1$ does not allow one to associate an operator to the Jacobi functional, nor to characterize multiple valued Jacobi fields as the solutions of a certain (Euler-Lagrange) PDE. Nonetheless, the variational structure of the problem suggests that the minimizers of the Jacobi functional for a given boundary datum have the right to be considered the multivalued counterpart of the classical Jacobi fields. This justifies the definition of Jacobi $Q$-fields given in Definition \ref{Jac_min}. In the next section we explore the first elementary properties of multiple valued Jacobi fields.

\section{Jacobi $Q$-fields} \label{sec:Jac_min}

The goal of this section is to provide the two fundamental tools which will be used in Section \ref{sec:ex} to address the question of the existence of Jacobi $Q$-fields $N$ in $\Omega$ with prescribed boundary value $\left. N \right|_{\partial \Omega} = \left. g \right|_{\partial \Omega}$ for some fixed $g \in \Gamma^{1,2}_{Q}(\N\Omega)$, and ultimately to prove Theorem \ref{cond_ex}. These tools are encoded in Proposition \ref{lsc} and Lemma \ref{str_stab_eq} below. The proof of Theorem \ref{cond_ex} will be obtained by direct methods in the Calculus of Variations, and therefore it is natural to analyze the properties of \emph{lower semi-continuity} and \emph{compactness} of the stability functional. The proof of the weak lower semi-continuity is rather simple, and it is the content of the following proposition.

\begin{proposition} \label{lsc}
The Jacobi functional is lower semi-continuous with respect to the weak topology of $\Gamma_{Q}^{1,2}(\N \Omega)$.
\end{proposition}

Before coming to the proof, it will be useful to make some further considerations about the structure of the Jacobi functional, in order to simplify the notation and to express it as a perturbation of the Dirichlet functional $\Dir(u, \Omega)$.

\begin{remark}
Given any $Q$-valued Lipschitz map $u = \sum_{\ell} \llbracket u^{\ell} \rrbracket$ satisfying $u^{\ell}(x) \in T_{x}^{\perp}\Sigma \subset T_{x}\M$ for all $x \in \Omega$, the orthogonal decomposition
\[
D_{\xi}u^{\ell}(x) = \mathbf{p}^{\M}_{x} \cdot D_{\xi}u^{\ell}(x) + \mathbf{p}^{\M \perp}_{x} \cdot D_{\xi}u^{\ell}(x) = \nabla_{\xi}u^{\ell}(x) + \overline{A}_{x}\left(\xi(x), u^{\ell}(x) \right) 
\]
holds for any tangent vector field $\xi$ at any point $x$ of differentiability for $u$, hence $\Ha^m$-a.e. in $\Omega$. Here, $\overline{A}$ denotes the second fundamental form of the embedding $\M \hookrightarrow \R^d$. Hence, at any such point we may write
\[
|\nabla_{\xi}u^{\ell}|^{2} = |D_{\xi}u^{\ell}|^{2} - |\overline{A}(\xi,u^\ell)|^{2}.
\]
These considerations are extended in the obvious way to all sections $u \in \Gamma_{Q}^{1,2}(\N \Omega)$ at all points of approximate differentiability. Ultimately, we will write
\begin{equation} \label{pert}
\Jac(u, \Omega) = \Dir(u, \Omega) - \mathcal{B}_{\Omega}(u),
\end{equation}
where $\mathcal{B}_{\Omega}$ is the functional on $\Gamma_{Q}^{1,2}(\N \Omega)$ defined by
\begin{equation} \label{B_funct}
\mathcal{B}_{\Omega}(u) := \int_{\Omega} \sum_{\ell=1}^{Q} \left( \sum_{i=1}^{m} |\overline{A}(\xi_i,u^{\ell})|^{2} + 2 \sum_{i,j=1}^{m} |\langle A(\xi_i,\xi_j), u^{\ell} \rangle|^{2} + \sum_{i=1}^{m} \langle R(u^{\ell}, \xi_i)\xi_i, u^{\ell} \rangle \right) \, \dHa^m. 
\end{equation}
Observe that $\mathcal{B}_{\Omega}$ satisfies an estimate of the form
\begin{equation} \label{B_est}
|\mathcal{B}_{\Omega}(u)| \leq C_{0} \| u \|_{L^2}^{2}\,,
\end{equation}
where $C_{0}$ is a \emph{geometric} constant, depending on $\mathbf{A}, \mathbf{\overline{A}}, \mathbf{R}$, where
\begin{align} \label{sup_norms_geometric_A}
\mathbf{A} &= \| A \|_{L^\infty} := \max_{x \in \Sigma} \max\left\lbrace \abs{A_{x}(X, Y)} \, \colon \, X,Y \in \Sf^{m-1} \subset T_{x}\Sigma \right\rbrace \,, \\ \label{sup_norms_geometric_Abar}
\mathbf{\overline{A}} &= \| \overline{A} \|_{L^\infty} := \max_{x \in \Sigma} \max\left\lbrace \abs{\overline{A}_{x}(X,Y)} \, \colon \, X,Y \in \Sf^{m+k-1} \subset T_{x}\M \right\rbrace \,, \\ \label{sup_norms_geometric_R}
\mathbf{R} &= \| R \|_{L^\infty} := \max_{x \in \Sigma} \max\left\lbrace \abs{\p_{x}^{\Sigma \perp} \cdot R_{x}(X,Y)Z} \, \colon \, X \in T_{x}^{\perp}\Sigma, \quad Y,Z \in T_{x}\Sigma, \quad \abs{X} = \abs{Y} = \abs{Z} = 1 \right\rbrace\,.
\end{align}
\end{remark}

\begin{proof}[Proof of Proposition \ref{lsc}]
Consider $Q$-valued sections $u_{h}, u \in \Gamma_{Q}^{1,2}(\N \Omega)$ and assume that $u_{h} \weak u$ weakly in $W^{1,2}(\Omega, \A_{Q})$ as in Definition \ref{weak}. Now, use \eqref{pert} in order to write
\[
\Jac(u_{h},\Omega) = \Dir(u_{h},\Omega) - \mathcal{B}_{\Omega}(u_h).
\]
The weak lower semi-continuity of the Dirichlet energy was proved by De Lellis and Spadaro in \cite[Proof of Theorem 0.8, p.30]{DLS11a}. On the other hand, the condition
\[
\lim_{h \to \infty} \int_{\Omega} \G\left(u_{h},u\right)^{2} \, \dHa^m = 0
\]
is enough to show that in fact
\begin{equation} \label{conv_err}
\lim_{h \to \infty} \mathcal{B}_{\Omega}(u_h) = \mathcal{B}_{\Omega}(u).
\end{equation}
To see this, just write \eqref{conv_err} as
\begin{equation} \label{conv_err2}
\lim_{h} \int_{\Omega} b_h \, \dHa^m = \int_{\Omega} b \, \dHa^m,
\end{equation}
with
\[
b_{h}(x) = \sum_{\ell=1}^{Q} \left( \sum_{i=1}^{m} |\overline{A}(\xi_i,u_{h}^{\ell})|^{2} + 2 \sum_{i,j=1}^{m} |\langle A(\xi_i,\xi_j), u_{h}^{\ell} \rangle|^{2} + \sum_{i=1}^{m} \langle R(u_{h}^{\ell}, \xi_i)\xi_i, u_{h}^{\ell} \rangle \right) 
\]
and
\[
b(x) = \sum_{\ell=1}^{Q} \left( \sum_{i=1}^{m} |\overline{A}(\xi_i,u^{\ell})|^{2} + 2 \sum_{i,j=1}^{m} |\langle A(\xi_i,\xi_j), u^{\ell} \rangle|^{2} + \sum_{i=1}^{m} \langle R(u^{\ell}, \xi_i)\xi_i, u^{\ell} \rangle \right),
\]
and observe that the strong convergence $u_h \to u$ in $L^{2}(\Omega, \A_Q)$ suffices to prove that along a subsequence (not relabeled) $b_{h}(x) \to b(x)$ pointwise $\Ha^m$-a.e. in $\Omega$. Equation \eqref{conv_err2} then follows by standard techniques in integration theory.
\end{proof}

As for compactness, the situation is much more involved. Indeed, as already anticipated, the existence of a solution of the minimum problem for the Jacobi functional with any boundary datum $g$ is strictly related with showing that $N_{0} \equiv Q \llbracket 0 \rrbracket$ is in fact the \emph{only} minimizer under $Q \llbracket 0 \rrbracket$ boundary conditions. 

\begin{remark} \label{Zero BD}
If $N \in \Gamma_{Q}^{1,2}(\N \Omega)$ is a Jacobi $Q$-field with $N|_{\partial \Omega} = Q \llbracket 0 \rrbracket$, then $\Jac(N,\Omega) = 0$. This is a consequence of the homogeneity properties of the functional: in this case, indeed, for any $t \in \R$ the $Q$-field $tN := \sum_{\ell} \llbracket tN^{\ell} \rrbracket$ is a suitable competitor, and
\[
\Jac(tN,\Omega) = t^{2} \Jac(N,\Omega).
\]
Hence, were $\Jac(N,\Omega) < 0$ \footnote{Observe that if $N|_{\partial\Omega} = Q \llbracket 0 \rrbracket$, then the null $Q$-field $N_{0} \equiv Q \llbracket 0 \rrbracket$ is a competitor, whence $\Jac(N,\Omega) \leq \Jac(N_{0},\Omega) = 0$ if $N$ is a minimizer.}, one would obtain that $\lim_{t \to \infty} \Jac(tN,\Omega) = -\infty$, thus contradicting the definition of $N$.

We are then able to conclude that if the minimum problem for the Jacobi functional with $Q \llbracket 0 \rrbracket$ boundary data does admit a solution, then for any minimizer $N$ one has $\Jac(N,\Omega) = 0$. In particular, $N_{0} \equiv Q \llbracket 0 \rrbracket$ is a minimizer.
\end{remark} 

The condition that $N_{0} \equiv Q \llbracket 0 \rrbracket$ is the only minimizer for its boundary value will be referred to as \emph{strict stability condition}, as it characterizes the strictly stable submanifolds in the $Q = 1$ case. In the following lemma we provide an equivalent condition to the strict stability.

\begin{lemma} \label{str_stab_eq}
There exists a unique solution $N_{0} \equiv Q \llbracket 0 \rrbracket$ of the problem
\[
\min \left\lbrace \Jac(u,\Omega) \, \colon \, u \in \Gamma_{Q}^{1,2}(\N \Omega) \mbox{ such that } u|_{\partial \Omega} = Q \llbracket 0 \rrbracket \right\rbrace
\]
if and only if there exists a positive constant $c = c(\Omega)$ such that
\begin{equation} \label{Strict Stability}
Jac(u,\Omega) \geq c \int_{\Omega} |u|^{2} \dHa^m,
\end{equation}
for all $u \in \Gamma_{Q}^{1,2}(\N \Omega)$ with $u|_{\partial \Omega} = Q \llbracket 0 \rrbracket$.
\end{lemma}

\begin{remark}
If $Q = 1$ and $\Sigma$ is strictly stable, then the largest positive constant $c(\Omega)$ such that \eqref{Strict Stability} holds for every $W^{1,2}$ section $u$ of $\N\Omega$ with $u|_{\partial\Omega} = 0$ is the first eigenvalue $\lambda_{1}$ of the Jacobi normal operator $\mathcal{L}$.
\end{remark}

\begin{proof}
Assume first that \eqref{Strict Stability} holds. Then, the Jacobi functional is non-negative on the space of $W^{1,2}$ sections of $\N\Omega$ with vanishing trace at the boundary. It is then clear that $N_{0} \equiv Q \llbracket 0 \rrbracket$ is a minimizer. Moreover, it is the only one. Indeed, if $u$ is a minimizer, then $\Jac(u,\Omega) = 0$, and therefore \eqref{Strict Stability} forces
\[
\int_{\Omega} \G(u,Q\llbracket 0 \rrbracket)^{2} \dHa^{m} = 0.
\]

For the converse, assume that the minimum problem for the Jacobi functional with vanishing boundary condition admits $N_{0} \equiv Q \llbracket 0 \rrbracket$ as the only solution. In particular, this implies that $\Jac(u,\Omega) \geq 0$ for all sections $u \in \Gamma_{Q}^{1,2}(\N \Omega)$ such that $u|_{\partial \Omega} = Q \llbracket 0 \rrbracket$, and in fact that $\Jac(u,\Omega) > 0$ for all such sections except the trivial one $N_{0}$. Now, assume by contradiction that \eqref{Strict Stability} does not hold. Then, for any $h \in \mathbb{N}$ there is a competitor $u_{h} \in \Gamma_{Q}^{1,2}(\N \Omega)$ such that $u_{h}|_{\partial \Omega} = Q \llbracket 0 \rrbracket$, $\| u_{h} \|_{L^{2}} = 1$ and
\[
\Jac(u_{h},\Omega) \leq \frac{1}{h}.
\]
In particular, as a consequence of \eqref{B_est}, we conclude that
\begin{equation}
\int_{\Omega} |Du_{h}|^{2} \, \dHa^{m} \leq C.
\end{equation}
By the compact embedding theorem for $Q$-valued maps, Proposition \ref{embeddings}, and by Proposition \ref{weak_clos}, we infer that there exist a subsequence $\{ u_{h'} \}$ of $\{ u_{h} \}$ and a section $u_{\infty} \in \Gamma_{Q}^{1,2}(\N \Omega)$, $u_{\infty}|_{\partial \Omega} = Q \llbracket 0 \rrbracket$, such that
\[
\lim_{h' \to \infty}\int_{\Omega} \G(u_{h'},u_{\infty})^{2} \, \dHa^{m} = 0,
\]
that is $u_{h'} \weak u_{\infty}$ weakly in $W^{1,2}$. Then, from the semi-continuity of the Jacobi functional follows:
\[
0 \leq \Jac(u_{\infty},\Omega) \leq \liminf_{h' \to \infty} \Jac(u_{h'},\Omega) = 0.
\]
Hence, $\Jac(u_{\infty},\Omega) = 0$, and thus $u_{\infty}$ is a minimizer. By hypothesis, $u_{\infty} \equiv Q \llbracket 0 \rrbracket$, which contradicts $\| u_{\infty} \|_{L^2} = 1$.
\end{proof}

\section{Existence of Jacobi $Q$-fields: the proof of Theorem \ref{cond_ex}} \label{sec:ex}

\subsection{$Q$-valued Luckhaus Lemma and the extension theorem} \label{sec:lemma_luckhaus}

The first step to prove existence of Jacobi $Q$-fields is to derive an extension theorem for $Q$-valued $W^{1,2}$ maps. Such a theorem will be obtained as a simple corollary of the following result, which is interesting per se.
\begin{proposition} \label{Luckhaus}
Let $\N$ be a $d$-dimensional closed Riemannian manifold of class $C^{2}$. Assume $0 < \lambda < 1$ and $f^{1},f^{2} \colon \N \to \A_{Q}(\R^{q})$ are two maps in $W^{1,2}$. Then, there exist a constant $C = C(\N, d, q, Q)$ and a map $h \in W^{1,2}\left( \N \times [0,\lambda] , \A_{Q}(\R^q) \right)$ such that
\begin{equation}\label{interp}
h(\cdot,0) \equiv f^{1} \hspace{0.2cm} \mbox{ and } \hspace{0.2cm} h(\cdot, \lambda) \equiv f^{2} \hspace{0.2cm} \mbox{ in }\N,
\end{equation}
satisfying
\begin{equation} \label{interp_L2}
\int_{\N \times \left[0,\lambda\right]} |h|^{2} \leq C \lambda \int_{\N} \left( |f^{1}|^{2} + |f^{2}|^{2} \right),
\end{equation}
\begin{equation}\label{interp_est}
\int_{\N \times [0,\lambda]} |Dh|^{2} \leq C \lambda \int_{\N} \left( |Df^{1}|^{2} + |Df^{2}|^{2} \right) + \frac{C}{\lambda} \int_{\N} \G(f^{1},f^{2})^{2}.
\end{equation}
\end{proposition}

Such a result adapts to the $Q$-valued setting a classical result by S. Luckhaus, concerning the extension of a Sobolev map defined on the boundary of an annulus $\partial(B_{1} \setminus B_{1-\lambda})$ in its interior with control on the $L^{2}$ norm of the gradient of the extension map (for the precise statement and the proof, see the original paper \cite[Lemma 1]{Luckhaus88}, or the nice presentations given by L. Simon in \cite[Section 2.12.2]{Simon96} or by R. Moser in \cite[Lemma 4.4]{Moser05}).

A version of this result in the $Q$-valued setting was given by C. De Lellis in \cite{DeLellis2013}, where the author interpolates between two functions defined on flat cubes, and by J. Hirsch in \cite[Lemma C.1]{Hir16} in the original Luckhaus setting of functions defined on the boundary of an annulus. The proof of the interpolation between maps defined over a closed Riemannian manifold presented here is based on De Lellis' result and on a decomposition of the manifold that is bi-Lipschitz to a $d$-dimensional cubic complex, following ideas already contained in \cite{White88} and \cite{Hang05}. We will make an extensive use of the following Lemma, which provides the elementary step in the construction of the interpolation.
\begin{lemma} \label{hom_ext}  
There is a constant $C$ depending only on $j$ and $Q$ with the following property. Assume that $0 < \lambda \leq 1$, $L = [0,\lambda]^{j} + z$ is a $j$-dimensional cube of side length $\lambda$, and $\varphi \in W^{1,2}\left(\partial L,\A_{Q}(\R^q)\right)$. Then, there is an extension $\psi$ of $\varphi$ to $L$ such that
\begin{equation} \label{hom_ext_L2}
\int_{L} |\psi|^{2} \, \dHa^j \leq C \lambda \int_{\partial L} |\varphi|^{2} \, \dHa^{j-1}
\end{equation}
and
\begin{equation} \label{hom_ext_dir}
\Dir(\psi,L) \leq C \lambda \Dir(\varphi,\partial L).
\end{equation}
\end{lemma}
\begin{proof}
First observe that, since the inequalities \eqref{hom_ext_L2} and \eqref{hom_ext_dir} are invariant with respect to translations and dilations, it suffices to prove the lemma when $L = \left[0,1\right]^{j}$. Moreover, since $L$ is bi-Lipschitz equivalent to the unit ball, it is enough to show the claim for $L = \overline{B}_{1} \subset \R^j$.

For reasons that will soon become clear, the proof when working in dimension $j=2$ is different with respect to the one in the higher dimensional case ($j \geq 3$): for this reason, we will distinguish between these two scenarios.

\textit{The higher dimensional case ($j \geq 3$)}. This is the easiest situation: indeed, it suffices to define $\psi$ as the zero-degree homogeneous extension of $\varphi$. That is, if $\varphi = \sum_{\ell} \llbracket \varphi^{\ell} \rrbracket$ on $\partial B_{1}$, then one just sets
\begin{equation}
\psi(x) := \sum_{\ell=1}^{Q} \left\llbracket \varphi^{\ell}\left(\frac{x}{|x|}\right) \right\rrbracket \hspace{0.5cm} \mbox{ for } x \in B_{1}\setminus \{0\}.
\end{equation}
A simple computation in radial coordinates shows that both estimates \eqref{hom_ext_L2} and \eqref{hom_ext_dir} hold with $C = \max\{j^{-1}, (j-2)^{-1}\} = (j-2)^{-1}$. Observe that this proof breaks down when $j = 2$, because the zero-degree homogeneous extension of $\varphi$ has infinite Dirichlet energy in two dimensions. 

\textit{The planar case ($j=2$)}. For this proof, it will be useful to introduce a suitable notation. We identify $\R^2$ with the complex plane $\C$, and the unit ball $B_{1} \subset \R^2$ with the disk $\bbD$, as usual defined as
\[
\bbD := \{ z \in \C \, \colon \, |z| < 1 \} = \{ re^{i\theta} \, \colon \, 0 \leq r < 1, \theta \in \R \}.
\]
The boundary of $\bbD$ is the unit circle ${\mathbb S}^{1}$, described by
\[
{\mathbb S}^{1} := \{ z \in \C \, \colon \, |z| = 1 \} = \{ e^{i\theta} \, \colon \, \theta \in \R \}.
\]
Consider now a function $\varphi \in W^{1,2}({\mathbb S}^{1}, \A_{Q})$. By \cite[Proposition 1.5]{DLS11a}, there exists a decomposition $\varphi = \sum_{\ell=1}^{P} \llbracket \varphi_{\ell} \rrbracket$, where each $\varphi_{\ell}$ is an \emph{irreducible} map in $W^{1,2}(\mathbb{S}^{1}, \A_{k_\ell})$. This means that $\sum_{\ell=1}^{P} k_\ell = Q$, and furthermore that for every $\ell = 1,\dots,P$ there exists a $W^{1,2}$ map $\gamma_{\ell} \colon \mathbb{S}^{1} \to \R^q$ such that
\begin{equation}
\varphi_\ell(x) = \sum_{z^{k_\ell} = x} \llbracket \gamma_{\ell}(z) \rrbracket,
\end{equation}
and with $\gamma_{\ell}(z_{1}) \neq \gamma_{\ell}(z_{2})$ if $z_{1} \neq z_{2}$ are two distinct $k_{\ell}$\textsuperscript{th} roots of $x$. In other words, if we identify the point $x = e^{i \theta} \in {\mathbb S}^{1}$ with the phase $\theta \in \left[0, 2 \pi\right)$ we have that
\begin{equation}
\varphi_{\ell}(\theta) = \sum_{m=0}^{k_\ell - 1} \left\llbracket \gamma_{\ell}\left(\frac{\theta + 2 \pi m}{k_\ell}\right) \right\rrbracket,
\end{equation}
with
\[
\gamma_{\ell}\left( \frac{\theta + 2\pi m}{k_{\ell}}\right) \neq \gamma_{\ell}\left( \frac{\theta + 2 \pi \tilde{m}}{k_{\ell}}\right) \quad \mbox{for } m \neq \tilde{m}.
\]

The idea, now, is to consider the \emph{harmonic extension} to the disk $\bbD$ of the function $\gamma_\ell$: if
\begin{equation}
\gamma_{\ell}(\theta) = \frac{a_{\ell,0}}{2} + \sum_{n=1}^{\infty} \left( a_{\ell,n} \cos(n\theta) + b_{\ell,n} \sin(n\theta) \right)
\end{equation}
is the Fourier series of $\gamma_\ell$, we let 
\begin{equation}
\zeta_{\ell}(r,\theta) = \frac{a_{\ell,0}}{2} + \sum_{n=1}^{\infty} r^{n} \left( a_{\ell,n} \cos(n\theta) + b_{\ell,n} \sin(n\theta) \right)
\end{equation}
denote its harmonic extension to the whole disk. Then, for each $\ell = 1,\dots,P$, we consider the $k_\ell$-valued function $\psi_\ell$ obtained ``rolling'' back the $\zeta_\ell$, that is
\begin{equation}
\psi_\ell(x) := \sum_{z^{k_\ell}=x} \llbracket \zeta_\ell(z) \rrbracket
\end{equation}
for $x \in \bbD$, and, finally, we set $\psi := \sum_{\ell=1}^{P} \llbracket \psi_\ell \rrbracket$. We claim now that $\psi$ is a $W^{1,2}(\bbD,\A_Q)$ extension of $\varphi$ satisfying the estimates \eqref{hom_ext_L2} and \eqref{hom_ext_dir}. To see this, fix $\ell \in \{1,\dots,P\}$ and define the following subsets of the unit disk,
\[
\mathcal{D}_m := \left\lbrace r e^{i \theta} \, \colon \, 0 < r < 1, \frac{2 \pi m}{k_\ell} < \theta < \frac{2 \pi (m+1)}{k_\ell} \right\rbrace
\]
for $m = 0,\dots,k_\ell-1$, and
\[
\mathcal{C} := \lbrace r e^{i\theta} \, \colon \, 0 < r < 1, \theta \neq 0 \rbrace.
\]
One immediately sees that $\psi_{\ell}|_{\mathcal{C}} = \sum_{m=0}^{k_\ell - 1} \llbracket \zeta_{\ell} \circ \sigma_{m} \rrbracket$, where $\sigma_m \colon \mathcal{C} \to \mathcal{D}_m$ are the $k_\ell$ determinations of the $k_\ell$\textsuperscript{th} root, that is
\[
\sigma_m(r e^{i\theta}) = r^{\frac{1}{k_\ell}} e^{i \left( \frac{\theta + 2 \pi m}{k_\ell} \right)}.
\]
Similarly, if the arcs $\mathcal{S}_{m}$ are defined by
\[
\mathcal{S}_{m} := \left\lbrace e^{i \theta} \, \colon \, \frac{2 \pi m}{k_\ell} < \theta < \frac{2 \pi (m+1)}{k_\ell} \right\rbrace,
\]
we have that $\varphi_{\ell}|_{\mathbb{S}^{1} \setminus \{1\}} = \sum_{m=0}^{k_\ell - 1} \llbracket \gamma_{\ell} \circ \tau_m \rrbracket$, where $\tau_m \colon \mathbb{S}^{1} \setminus \{1\} \to \mathcal{S}_m$ is given by $\tau_{m} := \sigma_{m}|_{\mathbb{S}^1}$.
Thus, we can immediately compute
\begin{equation}\label{L2_bdry}
\begin{split}
\int_{\mathbb{S}^{1}} |\varphi_\ell|^{2} &= \sum_{m=0}^{k_\ell - 1} \int_{\mathbb{S}^{1}} |\gamma_\ell \circ \tau_m|^{2} \\
&= k_{\ell} \sum_{m=0}^{k_\ell - 1} \int_{\mathcal{S}_m} |\gamma_\ell|^{2} \\
&= k_{\ell} \int_{\mathbb{S}^{1}} |\gamma_\ell|^{2} = k_{\ell} \pi \left( \frac{|a_{\ell,0}|^{2}}{2} + \sum_{n=1}^{\infty} (|a_{\ell,n}|^{2} + |b_{\ell,n}|^{2}) \right)
\end{split}
\end{equation}
by Plancherel's theorem. On the other hand, we can use polar coordinates to compute the integral of the extension $\psi_\ell$ to the disk and see that
\begin{equation}
\begin{split}
\int_{\mathbb{D}} |\psi_\ell|^{2} &= \sum_{m=0}^{k_\ell - 1} \int_{\mathcal{C}} |\zeta_\ell \circ \sigma_m|^{2} \\
&= \sum_{m=0}^{k_\ell - 1} \int_{0}^{1} \left( \int_{0}^{2\pi} \left|\zeta_\ell\left(\rho^{\frac{1}{k_\ell}}, \frac{\alpha + 2 \pi m}{k_\ell}\right)\right|^{2} \, d\alpha \right) \rho \, d\rho \\
&= k_{\ell}^{2} \sum_{m=0}^{k_\ell - 1} \int_{0}^{1} \left( \int_{\frac{2 \pi m}{k_\ell}}^{\frac{2 \pi (m+1)}{k_\ell}} \left|\zeta_\ell(r,\theta)\right|^{2} \, d\theta \right) r^{2k_\ell - 1} \, dr \\
&= k_{\ell}^{2} \int_{0}^{1} \left( \int_{0}^{2\pi} |\zeta_\ell(r,\theta)|^{2} \, d\theta \right) r^{2k_\ell - 1} \, dr \\
&= k_{\ell}^{2} \pi \int_{0}^{1} \left( \frac{|a_{\ell,0}|^{2}}{2} + \sum_{n=1}^{\infty} r^{2n}(|a_{\ell,n}|^{2} + |b_{\ell,n}|^{2}) \right) r^{2k_\ell - 1} \, dr \\
&\overset{\eqref{L2_bdry}}{\leq} k_{\ell} \left( \int_{\mathbb{S}^{1}} |\varphi_\ell|^{2} \right) \left( \int_{0}^{1} r^{2k_\ell - 1} \, dr \right) \\
&= \frac{1}{2} \int_{\mathbb{S}^{1}} |\varphi_\ell|^{2}.
\end{split}
\end{equation}
Summing over $\ell \in \{1,\dots,P\}$, we finally conclude that
\begin{equation}
\int_{\bbD} |\psi|^{2} \leq \frac{1}{2} \int_{\mathbb{S}^{1}} |\varphi|^{2},
\end{equation}
that is, \eqref{hom_ext_L2} holds with $C = \frac{1}{2}$. Concerning \eqref{hom_ext_dir}, we exploit the invariance of the Dirichlet energy under conformal mappings in order to infer that, for any $\ell = 1,\dots,P$,
\begin{equation} \label{Dir1}
\Dir(\psi_\ell, \mathcal{C}) = \sum_{m=0}^{k_\ell - 1} \Dir(\zeta_\ell \circ \sigma_m, \mathcal{C}) = \sum_{m=0}^{k_\ell - 1} \Dir(\zeta_\ell, \mathcal{D}_m) = \int_{\bbD} |D\zeta_\ell|^{2}.
\end{equation}
Now, by a simple computation on planar harmonic functions, it is easy to see that
\begin{equation} \label{Dir2}
\int_{\bbD} |D\zeta_\ell|^{2} \leq \int_{\mathbb{S}^{1}} |\partial_{\theta}\gamma_\ell|^{2},
\end{equation}
where $\partial_{\theta}$ is the tangential derivative along the circle. On the other hand, for every $\ell = 1,\dots,P$,
\begin{equation} \label{Dir3}
\begin{split}
\Dir(\varphi_\ell,\mathbb{S}^{1}) &= \sum_{m=0}^{k_\ell - 1} \int_{\mathbb{S}^{1}} |\partial_{\theta}(\gamma_\ell \circ \tau_m)|^{2} \\
&= \sum_{m=0}^{k_\ell - 1} \int_{\mathbb{S}^{1}} \frac{1}{k_\ell^2} |\partial_{\theta}\gamma_\ell \circ \tau_m|^{2} \\
&= \sum_{m=0}^{k_\ell - 1} \int_{\mathcal{S}_m} \frac{1}{k_\ell} |\partial_{\theta}\gamma_\ell|^{2} \\
&= \frac{1}{k_\ell} \int_{\mathbb{S}^1} |\partial_{\theta}\gamma_\ell|. 
\end{split}
\end{equation}
Finally, summing on $\ell$, the above arguments produce
\begin{equation}
\Dir(\psi,\bbD) \overset{\eqref{Dir1}}{=} \sum_{\ell=1}^{P} \int_{\bbD} |D\zeta_\ell|^{2} \overset{\eqref{Dir2}}{\leq} \sum_{\ell=1}^{P} \int_{\mathbb{S}^1} |\partial_{\theta}\gamma_\ell|^{2} \overset{\eqref{Dir3}}{=} \sum_{\ell=1}^{P} k_{\ell} \Dir(\varphi_\ell, \mathbb{S}^1) \leq Q \Dir(\varphi, \mathbb{S}^1),
\end{equation}
whence \eqref{hom_ext_dir} holds with $C = Q$.
\end{proof}

\begin{proof}[Proof of Proposition \ref{Luckhaus}]
Without loss of generality, we assume that $\N$ is an embedded submanifold of some Euclidean space $\R^N$. We shall divide the proof into steps.

\textit{Step 1.} We first consider a Lipschitz cubic decomposition of the manifold $\N$, that is a pair $\left( \K,\sigma \right)$, where $\K$ is a $d$-dimensional cubic complex, and $\sigma \colon |\K| \to \N$ is a bi-Lipschitz map, $|\K|$ denoting the union of all cells of $\K$. Without loss of generality, we may assume that each cell in $\K$ has unit $d$-dimensional volume. Set $m := \lfloor \frac{1}{\lambda} \rfloor + 1$. Using that $\left[ 0, 1 \right] = \bigcup_{i=1}^{m} \left[ \frac{i-1}{m}, \frac{i}{m} \right]$, we can divide each cell in $\K$ into $m^{d}$ smaller $d$-dimensional cubes, whose side length is at most $\lambda$. We will denote the resulting cubic complex by $\K_{m}$, and regard $\sigma$ as a bi-Lipschitz map $\sigma \colon |\K_{m}| \to \N$: observe that if $L$ is any cell in $\K_{m}$ then the image $\sigma(L)$ is a domain in $\N$ with diameter (computed with respect to the geodesic distance on $\N$) ${\bf diam}(\sigma(L)) \leq \sqrt{d} \, \Lip(\sigma) \lambda$.

For each $j \in \{0,1,\dots,d\}$, $\K_{m}^{j}$ will denote the $j$-skeleton of the complex $\K_{m}$, that is the family of all $j$-dimensional faces, and $|\K_{m}^j|$ will be their union. \\

\textit{Step 2.} Let now $\eta = \eta(\N) > 0$ be so small that the set $\U = \U_{2\eta}(\N) := \lbrace x \in \R^N \colon \dist(x,\N) < 2\eta \rbrace$ is a tubular neighborhood of $\N$, with (unique) differentiable nearest point projection $\Pi \colon \U_{2\eta}(\N) \to \N$. For $i = 1,2$, we extend $f^i$ to a map $F^i \colon \U \to \A_{Q}(\R^q)$ by setting $F^i := f^i \circ \Pi$. One has that
\begin{equation} \label{Retr0}
\int_{\U} |F^i|^{2} \leq c_1 \int_{\N} |f^i|^{2},
\end{equation}
\begin{equation} \label{Retr1}
\int_{\U} |DF^i|^{2} \leq c_{1} \int_{\N} |Df^i|^{2},
\end{equation}
and
\begin{equation}\label{Retr2}
\int_{\U} \G(F^{1},F^{2})^{2} \leq c_{1} \int_{\N} \G(f^1,f^2)^{2},
\end{equation}
where the constant $c_{1}$ depends only on the retraction $\Pi$ (and, thus, on the dimensions $d$ and $N$ and on the width $\eta$ of the tubular neighborhood).

Furthermore, for every $z \in |\K_{m}|$ and for every vector $v \in B^{N}_{\eta}$, we define $\sigma_{v}(z) := \Pi\left( \sigma(z) + v \right)$. Assume that $\eta$ is so small that all $\sigma_{v}$'s are bi-Lipschitz maps $|\K_{m}| \to \N$, and set $f^{i}_{v} := f^{i} \circ \sigma_{v}$, that is $f^{i}_{v}(z) = F^{i}\left( \sigma(z) + v \right)$. By Fubini's theorem, for all $j = 1,\dots,d$ and for a.e. $v \in B^{N}_{\eta}$ one has that $f^{i}_{v} \in W^{1,2} \left( F , \A_{Q}(\R^q) \right)$ for all faces $F \in \K_{m}^j$.

Consider now any non-negative function $\alpha \in L^{1}(\U)$. It is easily seen that there exists a constant $c_2 = c_2(\N,d,N,\eta)$ such that for any $j = 0,\dots,d$ and for every $\theta \in \left( 0,1 \right)$ 
\begin{equation}\label{skeleton}
\int_{|\K_{m}^{j}|} \alpha \left( \sigma(z) + v \right) \, \dHa^{j}(z) \leq c_2 \theta^{-1} \lambda^{j-d} \int_{\U} \alpha
\end{equation}
for all $v \in B^{N}_{\eta}$ with the exception of a set $E$ of $\mathcal{L}^{N}$-measure $|E| \leq \theta |B^{N}_{\eta}|$.
To prove this, first note that
\begin{equation}
\int_{|\K_{m}^j|} \alpha \left( \sigma(z) + v \right) \, \dHa^{j}(z) \leq \frac{1}{\theta |B^{N}_{\eta}|} \int_{B^{N}_{\eta}} \left( \int_{|\K_{m}^j|} \alpha \left( \sigma(z) + v \right) \, \dHa^{j}(z) \right) \, dv
\end{equation}
for all $v \in B^{N}_{\eta} \setminus E$, $|E| \leq \theta |B^{N}_{\eta}|$. Then, conclude by estimating:
\begin{equation}
\begin{split}
\int_{B^{N}_{\eta}} \left( \int_{|\K_{m}^j|} \alpha \left( \sigma(z) + v \right) \, \dHa^{j}(z) \right) \, dv &= \int_{|\K_{m}^{j}|} \left( \int_{B^{N}_{\eta}} \alpha \left( \sigma(z) + v \right) \, dv \right) \, \dHa^{j}(z)\\
&= \int_{|\K_{m}^{j}|} \left( \int_{B^{N}_{\eta}(\sigma(z))} \alpha(w) \, dw \right) \, \dHa^{j}(z) \\
&\leq \Ha^{j}(|K_{m}^{j}|) \int_{\U} \alpha \\
&\leq C m^{d} \lambda^{j} \int_{\U} \alpha \\
&\leq C \lambda^{j - d} \int_{\U} \alpha,
\end{split}
\end{equation}
where the constant $C$ appearing in the last line depends only on the number of cells in the original cubic complex $\K$ and on the dimension $d$.\\

Now, it suffices to apply \eqref{skeleton} with $\alpha = |F^i|^{2}$, $\alpha = |DF^i|^2$ and $\alpha = \G(F^1,F^2)^2$, and, say, $\theta = \frac{1}{2}$, and to plug in equations \eqref{Retr0}, \eqref{Retr1} and \eqref{Retr2} to finally show the following: there exists $v \in B^{N}_{\eta}$ such that for all $j \in \{1,\dots,d\}$ the following inequalities
\begin{equation} \label{skel_est0}
\int_{|\K_{m}^j|} \left( |f^{1}_{v}|^{2} + |f^{2}_{v}|^{2} \right) \leq C \lambda^{j-d} \int_{\N} \left( |f^{1}|^{2} + |f^{2}|^{2} \right),
\end{equation}
and 
\begin{equation} \label{skel_est1}
\int_{|\K_{m}^j|} \left( |Df^{1}_{v}|^{2} + |Df^{2}_{v}|^{2} + \G(f^{1}_{v},f^{2}_{v})^{2} \right) \leq C \lambda^{j-d} \int_{\N} \left( |Df^{1}|^{2} + |Df^{2}|^{2} + \G(f^{1},f^{2})^{2} \right)
\end{equation}
hold true with a constant $C = C(c_{1}, c_{2}, \Lip(\sigma))$. Furthermore, for $j=0$:
\begin{equation} \label{skel_est0b}
\sum_{z \in |\K_{m}^0|} \left( |f^{1}_{v}|^{2}(z) + |f^{2}_{v}|^{2}(z) \right) \leq C \lambda^{-d} \int_{\N} \left( |f^{1}|^{2} + |f^{2}|^{2} \right),
\end{equation}
\begin{equation} \label{skel_est2}
\sum_{z \in |\K_{m}^0|} \G\left(f^{1}_{v}(z),f^{2}_{v}(z)\right)^{2} \leq C \lambda^{-d} \int_{\N} \G\left(f^{1},f^{2}\right)^{2}.
\end{equation}

From now on, we will then assume to have fixed a $v \in B^{N}_{\eta}$ such that the corresponding maps $f^{i}_{v} \colon |\K_{m}| \to \A_{Q}(\R^q)$ satisfy equations \eqref{skel_est0}, \eqref{skel_est1}, \eqref{skel_est0b}, \eqref{skel_est2} and the following condition: for every $j \geq 1$, for each $\tau \in \K_{m}^j$ and for all $\gamma \in \K_{m}^{j-1}$ with $\gamma \subset \tau$, the restrictions $f^{i}_{v}|_{\tau}$ and $f^{i}_{v}|_{\gamma}$ are all $W^{1,2}$, and moreover the trace of $f^{i}_{v}|_{\tau}$ at $\gamma$ is precisely $f^{i}_{v}|_{\gamma}$. \\

\textit{Step 3.} Consider now the $(d+1)$-dimensional cubic complex $\overline{\K} := \K_{m} \times [0,\lambda]$ whose $(d+1)$-dimensional cells are cubes of the form $L \times [0,\lambda]$ for some $L \in \K_{m}^d$. A face $\tau \in \overline{\K}^{j}$, $j < d+1$, is said to be \emph{horizontal} if it is contained in $\K_{m} \times \{0\}$ (lower horizontal) or $\K_{m} \times \{\lambda\}$ (upper horizontal), \emph{vertical} otherwise. The collection of $j$-dimensional faces of $\overline{\K}$ is hence given by
\begin{equation}
\overline{\K}^{j} = \mathscr{L}^{j} \cup \mathscr{U}^{j} \cup \mathscr{V}^{j},
\end{equation}
where $\mathscr{L}^{j}$, $\mathscr{U}^{j}$ and $\mathscr{V}^{j}$ are the lower horizontal, upper horizontal and vertical $j$-dimensional faces respectively. Observe that $\mathscr{V}^{0} = \emptyset$, $\mathscr{L}^{0}$ consists of points $(z,0)$, while $\mathscr{U}^{0}$ consists of points $(z,\lambda)$ with $z \in \K_{m}^{0}$; note, furthermore, that all $(d+1)$-dimensional cells are vertical.

We are now in the position to define a map $\overline{h} \colon |\overline{\K}| \to \A_{Q}(\R^q)$. First of all, we set $\overline{h}|_{\beta} \equiv f^{1}_{v}|_{\beta}$ if $\beta$ is a lower horizontal face, and $\overline{h}|_{\tau} \equiv f^{2}_{v}|_{\tau}$ if $\tau$ is an upper horizontal face. Consider next any vertical segment $\gamma \in \mathscr{V}^{1}$. Its two endpoints are given by $(z,0)$ and $(z,\lambda)$ for some $z \in \K_{m}^0$. Now, if $f^{1}_{v}(z) = \sum_{\ell} \llbracket (f^{1}_{v})_{\ell}(z) \rrbracket$ and $f^{2}_{v}(z) = \sum_{\ell} \llbracket (f^{2}_{v})_{\ell}(z) \rrbracket$ are ordered in such a way that $\G\left( f^{1}_{v}(z), f^{2}_{v}(z) \right)^{2} = \sum_{\ell} |(f^{1}_{v})_{\ell}(z) - (f^{2}_{v})_{\ell}(z)|^{2}$, then a natural extension is obtained by setting
\begin{equation}
\overline{h}(z,\theta) := \sum_{\ell=1}^{Q} \left\llbracket (f^{1}_{v})_{\ell}(z) + \frac{\theta}{\lambda} \left( (f^{2}_{v})_{\ell}(z) - (f^{1}_{v})_{\ell}(z) \right) \right\rrbracket,
\end{equation}
for all $\theta \in [0,\lambda]$. In this way, we obtain the bounds
\begin{equation}
\int_{\gamma} |\overline{h}|^{2} \leq 2 \lambda \left( |f^{1}_{v}|^{2}(z) + |f^{2}_{v}|^{2}(z) \right)
\end{equation}
and
\begin{equation}
\int_{\gamma} |D\overline{h}|^{2} \leq \lambda^{-1} \G\left(f^{1}_{v}(z),f^{2}_{v}(z)\right)^{2}.
\end{equation}
If we carry on this procedure for all vertical segments, we obtain a well defined $Q$-valued map $\overline{h}$ on all the vertical $1$-skeleton $\mathscr{V}^{1}$, which, thanks to \eqref{skel_est0b} and \eqref{skel_est2}, satisfies
\begin{equation}
\int_{|\mathscr{V}^{1}|} |\overline{h}|^{2} \leq C \lambda^{1-d} \int_{\N} \left( |f^1|^{2} + |f^2|^{2} \right)
\end{equation}
and
\begin{equation}
\int_{|\mathscr{V}^{1}|} |D\overline{h}|^{2} \leq C \lambda^{-1-d} \int_{\N} \G\left(f^{1},f^{2}\right)^{2}.
\end{equation}

\begin{figure}[h]
\centering
\includegraphics[width=\textwidth]{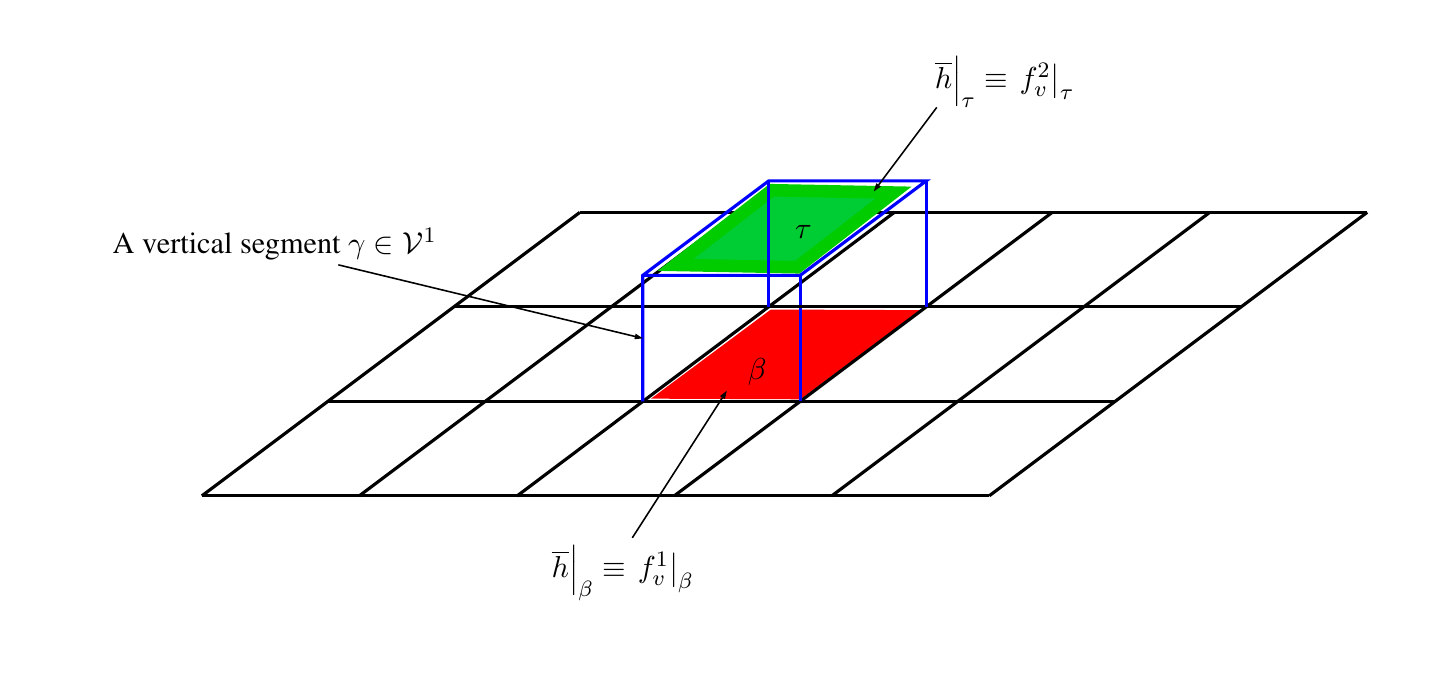}
\caption{The cubic complex $\overline{\K}$ and the first step in the construction of $\overline{h}$.} \label{fig_luckhaus}
\end{figure}

Pick next a vertical $2$-dimensional face $\tau$. Its boundary consists of two horizontal segments $\beta \in \mathscr{L}^{1}$ and $\delta \in \mathscr{U}^{1}$, and two vertical segments joining the points $(z,0)$, $(w,0)$ to the points $(z,\lambda)$, $(w,\lambda)$ respectively. Using our assumptions on $v$, we can conclude that $\overline{h}|_{\partial \tau}$ is in $W^{1,2}$, whence Lemma \ref{hom_ext} yields an extension of $\overline{h}$ to $\tau$ with estimates
\begin{equation}
\int_{\tau} |\overline{h}|^{2} \leq C \lambda \left( \int_{\beta} |f^{1}_{v}|^{2} + \int_{\delta} |f^{2}_{v}|^{2} \right) + C \lambda^2 \left( (|f^{1}_{v}|^{2} + |f^{2}_{v}|^{2})(z) + (|f^{1}_{v}|^{2} + |f^{2}_{v}|^{2})(w) \right)
\end{equation}
and
\begin{equation}
\int_{\tau} |D\overline{h}|^{2} \leq C \lambda \left( \int_{\beta} |Df^{1}_{v}|^{2} + \int_{\delta} |Df^{2}_{v}|^{2} \right) + C \left( \G\left(f^{1}_{v}(z),f^{2}_{v}(z)\right)^{2} + \G\left(f^{1}_{v}(w),f^{2}_{v}(w)\right)^{2}\right).
\end{equation}
Summing over the $2$-skeleton $\mathscr{V}^{2}$, from the estimates \eqref{skel_est0} and \eqref{skel_est0b} we deduce
\begin{equation}
\int_{|\mathscr{V}^{2}|} |\overline{h}|^{2} \leq C \lambda^{2-d} \int_{\N} \left( |f^{1}|^{2} + |f^{2}|^{2} \right),
\end{equation}
whereas \eqref{skel_est1} and \eqref{skel_est2} imply
\begin{equation}
\int_{|\mathscr{V}^{2}|} |D\overline{h}|^{2} \leq C \lambda^{2-d} \int_{\N} \left( |Df^{1}|^{2} + |Df^{2}|^{2} \right) + C \lambda^{-d} \int_{\N} \G\left(f^{1},f^{2}\right)^{2}.
\end{equation}

We then proceed inductively over $\mathscr{V}^{j}$, iteratively applying Lemma \ref{hom_ext} and using the inequalities \eqref{skel_est0} to \eqref{skel_est2} at each step. At the final iteration, namely for $j = d+1$, we construct a map $\overline{h}$ which is $W^{1,2}$ on each $(d+1)$-dimensional cell $L \times [0,\lambda]$, coinciding with $f^{1}_{v}$ on $L \times \{0\}$ and with $f^{2}_{v}$ on $L \times \{\lambda\}$. Furthermore, if two cells $H = L_{1} \times [0,\lambda]$ and $K = L_{2} \times [0,\lambda]$ have a common face $S \in \mathscr{V}^{d}$, the traces of $\overline{h}|_{H}$ and $\overline{h}|_{K}$ at $S$ coincide. Thus, we can regard $\overline{h}$ as a $W^{1,2}$ map defined on the whole cubic complex $\overline{\K}$. Moreover, since $|\overline{\K}| = \bigcup \mathscr{V}^{d+1} = |\mathscr{V}^{d+1}|$, the inductive step provides the following estimates:
\begin{equation}
\int_{|\overline{\K}|} |\overline{h}|^{2} \leq C \lambda \int_{\N} \left( |f^{1}|^{2} + |f^{2}|^{2} \right),
\end{equation}
\begin{equation}
\int_{|\overline{\K}|} |D\overline{h}|^{2} \leq C \lambda \int_{\N} \left( |Df^{1}|^{2} + |Df^{2}|^{2} \right) + \frac{C}{\lambda} \int_{\N} \G\left(f^{1},f^{2}\right)^{2}.
\end{equation}

\textit{Step 4.} Finally, we simply define a map $h \in W^{1,2}\left( \N \times [0,\lambda] , \A_{Q}(\R^q) \right)$ by setting
\begin{equation}
h(x,\theta) := \overline{h}\left( \sigma_{v}^{-1}(x), \theta \right).
\end{equation}
It is immediate to check that such a map indeed satisfies \eqref{interp}, \eqref{interp_L2} and \eqref{interp_est} in the statement.
\end{proof}

\begin{corollary} \label{Extension}
Let $\Sigma^{m} \hookrightarrow \R^{d}$ be a regular compact submanifold, and let $\lambda_{0} := {\rm inj}(\Sigma) > 0$ be the injectivity radius of $\Sigma$. Then, for any $0 < \lambda < \lambda_{0}$, for any $\mathcal{V} \subsetneq \Sigma$ connected, open subset with $C^{2}$ boundary and such that 
\[
\ddist(x, \partial \Sigma) \geq \lambda \quad \mbox{for every }x \in \partial \mathcal{V},
\]
and for any $\tilde{g}_{0} \in W^{1,2}(\partial\mathcal{V}, \A_{Q}(\R^d))$ there exist an open set $\mathcal{V}_{\lambda} \subset \Sigma$ with $\mathcal{V} \Subset \mathcal{V}_{\lambda}$, $\ddist(\mathcal{V}, \partial \mathcal{V}_{\lambda}) \leq \lambda$, and a map $\overline{g}_{\lambda} \in W^{1,2}\left(\mathcal{V}_{\lambda} \setminus \overline{\mathcal{V}},\A_{Q}(\R^{d})\right)$ satisfying:
\begin{equation}
\overline{g}_{\lambda}|_{\partial \mathcal{V}} = \tilde{g}_{0} \hspace{0.3cm} \mbox{ and } \hspace{0.3cm} \overline{g}_{\lambda}|_{\partial \mathcal{V}_{\lambda}} = Q \llbracket 0 \rrbracket,
\end{equation}
\begin{equation}
\int_{\mathcal{V}_{\lambda} \setminus \overline{\mathcal{V}}} |\overline{g}_{\lambda}|^{2} \, \dHa^m \leq C \lambda \int_{\partial \mathcal{V}} |\tilde{g}_{0}|^{2} \, \dHa^{m-1},
\end{equation}
\begin{equation}
\Dir(\overline{g}_{\lambda},\mathcal{V}_{\lambda} \setminus \overline{\mathcal{V}}) \leq C \lambda \Dir(\tilde{g}_{0},\partial \mathcal{V}) + \frac{C}{\lambda} \int_{\partial \mathcal{V}} |\tilde{g}_{0}|^{2} \, \dHa^{m-1},
\end{equation}
for a constant $C = C(\mathcal{V},m,d,Q)$.
\end{corollary}

\begin{proof}
Let $\mathcal{V}$ and $\tilde{g}_{0}$ be as in the statement. Then, by the very definition of injectivity radius, for any $0 < \lambda < \lambda_{0}$ the exponential map, restricted to $\partial \mathcal{V}$, is injective in a ball of radius $\lambda$ around the zero section of the normal bundle of $\partial \mathcal{V}$ in $\Sigma$. In turn, this allows one to define, for any such $\lambda$, a $\lambda$-tubular neighborhood $\U_{\lambda}$ of $\partial \mathcal{V}$ in $\Sigma$ by setting
\begin{equation}
\U_{\lambda} := \lbrace \exp_{\pi}\left( \theta \eta(\pi) \right) \, \colon \, \pi \in \partial \mathcal{V}, \, |\theta| < \lambda \rbrace,
\end{equation}
where for every point $\pi \in \partial \mathcal{V}$ we have denoted $\eta(\pi) \in T_{\pi}\Sigma$ the unit outer co-normal vector to $\partial \mathcal{V}$ at $\pi$.

Note that it is well defined a differentiable parametrization $x \in \U_{\lambda} \mapsto \left( \pi(x),\theta(x) \right) \in \partial \mathcal{V} \times \left( -\lambda,\lambda \right)$ such that $\exp_{\pi(x)} \left( \theta(x) \eta(\pi(x)) \right) = x$ for all $x \in \U_{\lambda}$.
 
Next, the \emph{positive} and \emph{negative} $\lambda$-tubular neighborhoods of $\partial \mathcal{V}$ in $\Sigma$ are respectively defined by
\begin{equation}
\U_{\lambda}^{+} := \lbrace \exp_{\pi}\left( \theta \eta(\pi) \right) \, \colon \, \pi \in \partial \mathcal{V}, \, 0 < \theta < \lambda \rbrace,
\end{equation}
\begin{equation}
\U_{\lambda}^{-} := \lbrace \exp_{\pi}\left( \theta \eta(\pi) \right) \, \colon \, \pi \in \partial \mathcal{V}, \, - \lambda < \theta < 0 \rbrace.
\end{equation}

We set $\mathcal{V}_{\lambda} := \mathcal{V} \cup \U_{\lambda}$. The claimed result is then simply obtained by applying Proposition \ref{Luckhaus} with $\N = \partial \mathcal{V}$, $f^{1} = \tilde{g}_{0}$, $f^{2} = Q \llbracket 0 \rrbracket$ and setting $\overline{g}_{\lambda}(x) := h\left( \pi(x),\theta(x) \right)$ for $x \in \mathcal{V}_{\lambda} \setminus \overline{\mathcal{V}} = \U_{\lambda}^{+}$.
\end{proof}

\subsection{The compactness theorem}
Back to our original setting, we let $\Omega$ be an open and connected subset of $\Sigma \hookrightarrow \M$ in which we wish to solve the minimum problem for the $\Jac$ functional. We will assume $C^{2}$ regularity for $\partial \Omega$. Let $\lambda_{0} := {\rm inj}(\Sigma)$. For $0 < \lambda < \lambda_{0}$, set $\mathcal{V} := \lbrace x \in \Omega \, \colon \, \ddist(x, \partial\Omega) > \lambda \rbrace$, so that $\Omega$ coincides with the set $\mathcal{V}_{\lambda} = \mathcal{V} \cup \U_{\lambda}$ which was obtained in the proof of Corollary \ref{Extension}. Using the same notations introduced in the proof of Corollary \ref{Extension}, we parametrize $\U_{\lambda}$ with coordinates $\left( \pi, \theta \right) \in \partial\mathcal{V} \times \left( -\lambda, \lambda \right)$. 

Let us now define $\Phi_{\lambda} \colon \mathcal{V} \to \Omega$ to be the diffeomorphism given by:
\begin{equation}
\Phi_{\lambda}(x) :=
\begin{cases}
\exp_{\pi(x)}\left( \varphi_{\lambda}(\theta(x)) \eta(\pi(x)) \right) &\mbox{ if } x \in \U_{\lambda}^{-} \\
x &\mbox{ otherwise },
\end{cases}
\end{equation}
where $\varphi_{\lambda}$ is any monotone increasing diffeomorphism $\varphi_{\lambda} \colon \left( -\lambda,0 \right) \to \left( -\lambda,\lambda \right)$ such that $\varphi_{\lambda}(\theta) =
 \theta$ for $\theta \in \left( -\lambda, -\frac{\lambda}{2} \right)$. From this moment on, we will assume that such a family of diffeomorphisms $\varphi_{\lambda}$ has been fixed, and satisfies a bound of the form
\begin{equation} \label{bound}
c^{-1} \leq |\varphi_{\lambda}'| \leq c
\end{equation}
for a positive constant $c$ which does not depend on $\lambda$.

Furthermore, if $u = \sum_{\ell} \llbracket u_{\ell} \rrbracket$ is any map in $W^{1,2}\left( \Omega,\A_{Q}(\R^d) \right)$, we set:
\begin{equation} \label{orth_proj}
u^{\perp}(x) := \sum_{\ell=1}^{Q} \llbracket \p_{x}^{\Sigma\perp} \cdot u_{\ell}(x) \rrbracket,
\end{equation}
where $\p^{\Sigma\perp}$ is the normal bundle projection defined in Definition \ref{def:normal_dirichlet}. Observe that $u^{\perp} \in \Gamma_{Q}^{1,2}(\N \Omega)$. The following Lemma yields a useful formula to relate the Dirchlet energy of $u$ with the Dirichlet energy of $u^{\perp}$. 
\begin{lemma} \label{Projection}
For every $\varepsilon > 0$ there exists a positive constant $C_{\varepsilon}$ such that the following estimate holds true:
\begin{equation} \label{proj}
\Dir(u^{\perp},\Omega) \leq (1 + \varepsilon) \Dir(u,\Omega) + C_{\varepsilon} \int_{\Omega} |u|^{2} \, \dHa^{m}.
\end{equation}
\end{lemma}

\begin{proof}
Write $\mathbf{p}^{\Sigma \perp}(x,v) := \mathbf{p}^{\Sigma \perp}_{x} \cdot v$ for $x \in \Omega$, $v \in \R^d$. If $v = v(x)$ is a (single valued) Lipschitz map defined in $\Omega$, then for any tangent vector field $\xi$ one has
\[
D_{\xi}\mathbf{p}^{\Sigma \perp}(x,v(x)) = \partial_{x}\mathbf{p}^{\Sigma \perp}(x,v(x)) \cdot \xi(x) + \partial_{v}\mathbf{p}^{\Sigma \perp}(x,v(x)) \cdot D_{\xi}v(x).
\] 

Since, for fixed $x$, the map $v \in \R^{d} \mapsto \mathbf{p}^{\Sigma \perp}(x,v) \in T_{x}^{\perp}\Sigma$ is linear with Lipschitz constant not larger than $1$, we conclude that for any $v \colon \Omega \to \R^{d}$ Lipschitz one has
\[
\Dir(v^{\perp},\Omega) \leq \Dir(v,\Omega) + C \int_{\Omega} |v|^{2} \, \dHa^m + C \int_{\Omega} |v||Dv| \, \dHa^m,
\]
where $C$ is a constant depending on $\max_{\overline{\Omega} \times \mathbb{S}^{d-1}} |\partial_{x}\mathbf{p}^{\Sigma\perp}|$.

Formula \eqref{proj} is then a consequence of Young's inequality. The formula is then extended to Lipschitz $Q$-valued maps via decomposition into $Q$ Lipschitz functions (Proposition \ref{Lip-select}), and finally to Sobolev $Q$-maps via approximation (Proposition \ref{LipSob_app}).
\end{proof}

We are now ready to state and prove the proposition that will provide the key towards Theorem \ref{cond_ex}.
\begin{proposition}
Let $\Omega \subset \Sigma$ be open, connected with $C^{2}$ boundary. Assume the strict stability condition \eqref{Strict Stability} holds for every $u \in \Gamma_{Q}^{1,2}(\N \Omega)$ such that $u|_{\partial \Omega} = Q \llbracket 0 \rrbracket$. Then, if $g \in \Gamma_{Q}^{1,2}(\N\Omega)$ has boundary trace $g_{0} := g|_{\partial\Omega} \in W^{1,2}(\partial\Omega, \A_{Q}(\R^d))$, the following estimate
\begin{equation} \label{Str_bound_stab}
\Jac(N, \Omega) \geq c(\Omega) \int_{\Omega} |N|^{2} \, \dHa^m - C(\Omega,g_{0})
\end{equation}
holds true for any $N \in \Gamma_{Q}^{1,2}(\N \Omega)$ such that $N|_{\partial \Omega} = g_{0}$.
\end{proposition}

\begin{proof}
Fix $\lambda < \lambda_{0}$ to be chosen, and let $\mathcal{V} \Subset \Omega$ be such that $\Omega = \mathcal{V}_{\lambda}$ as above. For any $N \in \Gamma_{Q}^{1,2}(\N \Omega)$ such that $N|_{\partial \Omega} = g_{0}$, consider the map $\tilde{N} := N \circ \Phi_{\lambda} \in W^{1,2}(\mathcal{V}, \A_{Q}(\R^d))$, and observe that $\tilde{N}|_{\partial\mathcal{V}} = \tilde{g}_{0}$, where $\tilde{g}_{0}(\pi) = g_{0}(\exp_{\pi}(\lambda \eta(\pi)))$ for $\pi \in \partial\mathcal{V}$.

Now, apply Corollary \ref{Extension} with this choice of $\mathcal{V}$, $\tilde{g}_{0}$ and $\lambda$ in order to extend $\tilde{N}$ to the map $u \in W^{1,2}(\Omega, \A_{Q}(\R^d))$ given by
\begin{equation}
u(x) :=
\begin{cases}
\tilde{N}(x) &\mbox{ if } x \in \mathcal{V} \\
\overline{g}_{\lambda}(x) &\mbox{ if } x \in \Omega \setminus \mathcal{V} = \U_{\lambda}^{+},
\end{cases}
\end{equation}
Observe that the normal bundle projection $u^{\perp}$ satisfies $u^{\perp} \in \Gamma_{Q}^{1,2}(\N \Omega)$ and the boundary condition $u^{\perp}|_{\partial \Omega} = Q \llbracket 0 \rrbracket$. From the hypothesis, we are therefore able to conclude that
\begin{equation}\label{sbs1}
\Jac(u^{\perp},\Omega) \geq c(\Omega) \int_{\Omega} |u^{\perp}|^{2} \, \dHa^m.
\end{equation}

Now, note that $u^{\perp} \equiv N$ in $\Omega \setminus \U_{\lambda}$. Combining this observation with \eqref{sbs1}, we trivially deduce
\begin{equation}\label{sbs2}
\Jac\left(N,\Omega \setminus \U_{\lambda}\right) + \Jac\left(u^{\perp},\U_{\lambda}\right) \geq c(\Omega) \left( \int_{\Omega \setminus \U_{\lambda}} |N|^2 \, \dHa^m + \int_{\U_{\lambda}} |u^{\perp}|^{2} \, \dHa^m \right).
\end{equation} 

In order to prove our result, we then clearly have to provide suitable estimates for $\Jac\left(u^{\perp}, \U_{\lambda}\right)$ and $\int_{\U_{\lambda}} |u^{\perp}|^{2}$.

We observe first that
\begin{equation} \label{discard}
\int_{\U_{\lambda}} |u^{\perp}|^{2} = \int_{\U_{\lambda}^{-}} |u^{\perp}|^{2} + \int_{\U_{\lambda}^{+}} |u^{\perp}|^{2} \geq \int_{\U_{\lambda}^{-}} |u^{\perp}|^{2}.
\end{equation}
Recall that
\[
u^{\perp}|_{\U_{\lambda}^{-}} = \sum_{\ell=1}^{Q} \llbracket \p^{\Sigma\perp} \cdot (N_{\ell} \circ \Phi_{\lambda}) \rrbracket = \sum_{\ell=1}^{Q} \left\llbracket \sum_{\beta=1}^{k} \langle N_{\ell} \circ \Phi_{\lambda}, \nu_{\beta} \rangle \nu_{\beta} \right\rrbracket,
\]
whence
\begin{equation}
\int_{\U_{\lambda}^{-}} |u^{\perp}|^{2} = \int_{\U_{\lambda}^{-}} \sum_{\ell=1}^{Q} \sum_{\beta=1}^{k} |\langle N_{\ell}\left( \Phi_{\lambda}(x) \right), \nu_{\beta}(x) \rangle|^{2} \, \dHa^{m}(x).
\end{equation}
Now, changing variable $y = \Phi_{\lambda}(x)$, integrating along geodesics and using \eqref{bound} one easily proves that from this follows
\begin{equation} \label{error1}
\int_{\U_{\lambda}^{-}} |u^{\perp}|^{2} \geq C \left( \int_{\U_{\lambda}} |N|^{2} \, \dHa^m - \mathcal{E}_{\lambda}^{(1)} \right),
\end{equation}
where the error term $\mathcal{E}_{\lambda}^{(1)}$ satisfies the estimate
\begin{equation}\label{error1:bd}
|\mathcal{E}_{\lambda}^{(1)}| \leq \frac{1}{2} \int_{\U_{\lambda}} |N|^{2} \, \dHa^m,
\end{equation}
\emph{provided} $\lambda$ satisfies some suitable smallness conditions which are \emph{not} depending on $N$. Combining \eqref{sbs2}, \eqref{discard}, \eqref{error1} and \eqref{error1:bd}, we conclude that for suitably small $\lambda$
\begin{equation}\label{sbs3}
\Jac\left(N, \Omega \setminus \U_{\lambda}\right) + \Jac\left( u^{\perp}, \U_{\lambda} \right) \geq c(\Omega) \int_{\Omega} |N|^{2} \, \dHa^m,
\end{equation}
up to possibly changing the value of $c(\Omega)$.

Now, we work on $\Jac\left( u^{\perp}, \U_{\lambda} \right)$. As before, decompose
\begin{equation}
\Jac\left( u^{\perp}, \U_{\lambda} \right) = \Jac\left( u^{\perp}, \U_{\lambda}^{-} \right) + \Jac\left( u^{\perp}, \U_{\lambda}^{+} \right).
\end{equation}
Concerning the first addendum, one shows that
\begin{equation}\label{error2}
\Jac\left( u^{\perp}, \U_{\lambda}^{-} \right) \leq C \Jac\left( N, \U_{\lambda} \right) + \mathcal{E}_{\lambda}^{(2)},
\end{equation}
where the error $\mathcal{E}_{\lambda}^{(2)}$ satisfies
\begin{equation}\label{error2:bd}
|\mathcal{E}_{\lambda}^{(2)}| \leq \varepsilon \left( \Dir\left( N, \U_{\lambda} \right) + \int_{\U_{\lambda}} |N|^{2} \, \dHa^m \right)
\end{equation}
for any choice of $\varepsilon > 0$, provided $\lambda$ is smaller than some $\lambda_{*}$ depending on $\varepsilon$ and on the geometry of the problem, but, again, not on the map $N$. In particular, this allows to absorb the error term and conclude, under the previously considered smallness assumptions on $\lambda$, that
\begin{equation}\label{sbs4}
\Jac\left( N, \Omega \right) \geq c(\Omega) \left( \int_{\Omega} |N|^{2} \, \dHa^m - \Jac\left( u^{\perp}, \U_{\lambda}^{+} \right) \right)
\end{equation}
after possibly having redefined $c(\Omega)$.

Now we are able to conclude: following the same strategy as before, it is not difficult to estimate
\begin{equation}
|\Jac\left( u^{\perp}, \U_{\lambda}^{+} \right)| \leq C \left( \Dir\left( \overline{g}_{\lambda}, \U_{\lambda}^{+} \right) + \int_{\U_{\lambda}^{+}} |\overline{g}_{\lambda}|^{2} \, \dHa^m \right),
\end{equation}
where $\lambda$ is small, but fixed, and does not depend on $N$. Our result, equation \eqref{Str_bound_stab}, is then an immediate consequence of Corollary \ref{Extension} and of the definition of $\tilde{g}_{0}$.
\end{proof}

We are now ready to prove the Conditional Existence Theorem \ref{cond_ex}.

\begin{proof}[Proof of Theorem \ref{cond_ex}]
The proof is an application of the direct methods in the Calculus of Variations. Fix any $g \in \Gamma_{Q}^{1,2}(\N\Omega)$ with boundary trace $g_{0} := g|_{\partial\Omega} \in W^{1,2}(\partial\Omega, \A_{Q}(\R^d))$. Then, the inequality \eqref{Str_bound_stab} implies that for any $N \in \Gamma_{Q}^{1,2}(\N \Omega)$ with $N|_{\partial \Omega} = g_{0}$ one has
\[
\Jac(N,\Omega) \geq -C(\Omega,g_{0}),
\]
thus the Jacobi functional is bounded from below in the class of competitors.

Set
\[
\Lambda := \inf \lbrace \Jac(N,\Omega) \, \colon \, N \in \Gamma_{Q}^{1,2}(\N \Omega), N|_{\partial \Omega} = g_{0} \rbrace > -\infty,
\]
and consider a minimizing sequence $\{N_{h}\}_{h=1}^{\infty} \subset \Gamma_{Q}^{1,2}(\N \Omega)$, $N_{h}|_{\partial \Omega} = g_{0}$, $\lim_{h \to \infty} \Jac(N_h,\Omega) = \Lambda$. Then, for $h \geq h_{0}$ sufficiently large, one has
\[
\Jac(N_{h},\Omega) \leq \Lambda+1,
\]
from which we deduce
\[
\Dir(N_{h},\Omega) \leq C \int_{\Omega} |N_{h}|^{2} \, \dHa^m + |\Lambda| + 1.
\]

On the other hand, \eqref{Str_bound_stab} immediately implies that
\[
\int_{\Omega} |N_{h}|^{2} \, \dHa^m \leq C(|\Lambda|,\Omega,g_{0}).
\]

Putting all together, we conclude that
\begin{equation}
\Dir(N_{h},\Omega) + \int_{\Omega} |N_{h}|^{2} \, \dHa^m \leq C,
\end{equation}
where $C$ is a constant depending only on $|\Lambda|$, $\Omega$, $g_{0}$ and the geometry of the embeddings $\Sigma \hookrightarrow \M \hookrightarrow \R^d$. Hence, up to extracting a subsequence, $N_{h}$ converges weakly in $W^{1,2}$, strongly in $L^2$, to a map $\overline{N} \in \Gamma_{Q}^{1,2}(\N \Omega)$ with $\overline{N}|_{\partial \Omega} = g_{0}$. The lower semi-continuity of the Jacobi functional with respect to weak convergence, Proposition \ref{lsc}, allows us to conclude that $\overline{N}$ is the desired minimizer. 
\end{proof}

\section{H\"older regularity of Jacobi $Q$-fields} \label{section:reg}

In this section, we present a proof of the following quantitative version of Theorem \ref{qual_reg_Holder}. As usual, $\Omega$ is an open subset of the compact $m$-dimensional manifold $\Sigma$ minimally embedded in $\M$.
\begin{theorem}[H\"older regularity of Jacobi multi-fields] \label{reg_Holder}
There exist universal constants $\alpha = \alpha(m,Q) \in \left( 0, 1 \right)$ and $\Lambda = \Lambda(m,Q) > 0$ and a radius $0 < r_{0} = r_{0}(m,Q) < {\rm inj}(\Sigma)$ with the following property. If $N \in \Gamma_{Q}^{1,2}(\N\Omega)$ is $\Jac$-minimizing, then for every $0 < \theta < 1$ there exists a constant $C = C(m,d,Q,\Sigma,\theta)$ such that
\begin{equation} \label{reg_Holder:est}
\begin{split}
\left[ N \right]_{C^{0,\alpha}(\overline{\B}_{\theta r}(p))} :&= \sup_{x_{1}, x_{2} \in \overline{\B}_{\theta r}(p)} \frac{\G(N(x_1), N(x_2))}{\mathbf{d}(x_1,x_2)^{\alpha}} \\
&\leq C \left( r^{2-m-2\alpha} \left( \Dir(N,\B_{r}(p)) + \Lambda \int_{\B_{r}(p)} |N|^{2} \, \dHa^{m} \right) \right)^{\sfrac{1}{2}}
\end{split}
\end{equation}
for every $p \in \Omega$ and for every $r \leq \min\lbrace r_{0}, \ddist(p, \partial\Omega)\rbrace$. In particular, $N \in C^{0,\alpha}_{loc}(\Omega, \A_{Q}(\R^d))$.
\end{theorem}

The proof of Theorem \ref{reg_Holder} is a fairly easy consequence of Proposition \ref{Prop Holder} below. Before stating it, we need to introduce some further notation. 

Let us fix a point $p \in \Omega$, and a radius $r < \min\lbrace {\rm inj}(\Sigma), \ddist(p, \partial\Omega)\rbrace$, in such a way that the exponential map $\exp_{p}$ defines a diffeomorphism
\[
\exp_{p} \colon B_{r}(0) \subset T_{p}\Sigma \to \mathbf{B}_{r}(p) \subset \Omega.
\] 

Denote by $y = (y^{1},\dots,y^{m})$ coordinates in $T_{p}\Sigma$ corresponding to the choice of an orthonormal basis $(e_{1},\dots,e_{m})$, and set $u := N \circ \exp_{p}$. Observe that for any $y \in B_{r}$ the differential $d(\exp_{p})|_{y}$ realizes a linear isomorphism between $T_{p}\Sigma$ and $T_{\exp_{p}(y)}\Sigma$. Fix an orthonormal frame $(\xi_{1},\dots,\xi_{m})$ of the tangent bundle $\T \Sigma|_{{\bf B}_{r}(p)}$ extending $(e_{1},\dots,e_{m})$ (i.e. such that $\xi_{i}|_{p} = e_{i}$ for $i = 1,\dots,m$), and define, for $y \in B_{r}$,
\begin{equation}
\varepsilon_{i}(y) := \left( d(\exp_{p})|_{y} \right)^{-1} \cdot \xi_{i}(\exp_{p}(y)).
\end{equation}
Then, an elementary computation shows that
\begin{equation}
\int_{\mathbf{B}_{r}(p)} |N(x)|^{2} \, \dHa^{m}(x) = \int_{B_{r}} |u(y)|^{2} \mathbf{J}\exp_{p}(y) \, {\rm d}y
\end{equation}
and
\begin{equation}
\Dir(N,\mathbf{B}_{r}(p)) = \int_{B_{r}} \sum_{i=1}^{m} |D_{\varepsilon_{i}}u(y)|^{2} \mathbf{J}\exp_{p}(y) \, {\rm d}y, 
\end{equation}
where $\mathbf{J}\exp_{p}$ is the Jacobian determinant of the exponential map. From this it is immediate to deduce that the following asymptotic behaviors are satisfied for $r \to 0$ uniformly in $p$:
\begin{equation} \label{ab1}
\int_{\B_{r}(p)} |N(x)|^{2} \, \dHa^{m}(x) = (1 + O(r)) \int_{B_r} |u(y)|^{2} \, {\rm d}y,
\end{equation}
\begin{equation} \label{ab2}
\Dir(N, \B_{r}(p)) = (1 + O(r))\int_{B_r} \sum_{i=1}^{m} |D_{e_i}u(y)|^{2} \, {\rm d}y = (1+O(r)) \Dir(u,B_r).
\end{equation}

We can now state the key result from which we will conclude the H\"older regularity of Jacobi $Q$-fields.

\begin{proposition} \label{Prop Holder}
There exist a universal positive constant $\Lambda = \Lambda(m,Q)$ and a radius $0 < r_0 = r_0(m,Q) < {\rm inj}(\Sigma)$ with the following property. Let $N \in \Gamma_{Q}^{1,2}(\N \Omega)$ be $\Jac$-minimizing and $p \in \Omega$. Then, for a.e. radius $r \leq \min\{r_{0}, \ddist(p, \partial \Omega)\}$ one has
\begin{equation} \label{Diff Ineq 1}
\Dir(u,B_r) + \Lambda \int_{B_r} |u|^{2} \, {\rm d}y \leq C(m) r \left( \Dir(u, \partial B_r) + \Lambda \int_{\partial B_r} |u|^{2} \, \dHa^{m-1} \right),  
\end{equation} 
where $u := N \circ \exp_{p}|_{B_r} \in W^{1,2}\left(B_r , \A_{Q}(\R^{d}) \right)$ and $C(m) < (m-2)^{-1}$ when $m \geq 3$.
\end{proposition}

In order to prove Proposition \ref{Prop Holder}, we will need the following simple result on classical Sobolev functions in the Euclidean space.

\begin{lemma} \label{Poincare}
For every $\varepsilon > 0$ there exists a constant $C = C_{\varepsilon} > 0$ such that the inequality
\begin{equation}\label{Poincare:eq}
\int_{B_r} |g|^{2} \, {\rm d}y \leq \left( \frac{1}{m} + \varepsilon \right) r \int_{\partial B_r} |g|^{2} \, \dHa^{m-1} + C_{\varepsilon} r^{2} \int_{B_r} |Dg|^{2} \, {\rm d}y
\end{equation}
holds for any function $g \in W^{1,2}(B_r^{m})$.
\end{lemma}

\begin{proof}
First observe that, by a simple scaling argument, it is enough to prove the lemma for $r = 1$. Assume the lemma is false: suppose, by contradiction, that there exists $\varepsilon_0 > 0$ such that for any $h \in \mathbb{N}$ there is $g_{h} \in W^{1,2}(B_1^{m})$, with $\| g_{h} \|_{L^2} = 1$, such that
\begin{equation} \label{contradiction}
1 >  \left( \frac{1}{m} + \varepsilon_0 \right) \int_{\partial B_1} |g_{h}|^{2} \, \dHa^{m-1} + h \int_{B_1} |Dg_{h}|^{2} \, {\rm d}y.
\end{equation}
The inequality \eqref{contradiction} readily implies that 
\begin{equation}
\lim_{h \to \infty} \int_{B_{1}} |Dg_{h}|^{2} \, {\rm d}y = 0,
\end{equation}
whence, by Rellich's compactness theorem, the sequence $g_{h}$ converges up to a subsequence (not relabeled) weakly in $W^{1,2}$, strongly in $L^2$, to a constant function $g \equiv c$. The condition $\| g \|_{L^2} = 1$ forces the constant to satisfy $|c|^{2} = \omega_{m}^{-1}$, where $\omega_{m}$ is the volume of the unit ball in $\R^{m}$. Hence, it suffices to pass to the limit the inequality
\begin{equation}
1 > \left( \frac{1}{m} + \varepsilon_{0} \right) \int_{\partial B_{1}} |g_{h}|^{2} \, \dHa^{m-1}
\end{equation}
to obtain the desired contradiction:
\begin{equation}
1 > \left( \frac{1}{m} + \varepsilon_{0} \right) m.
\end{equation}
\end{proof}

\begin{corollary} \label{cor Q-Poincare}
For every $\varepsilon > 0$ there exists a constant $C_{\varepsilon} > 0$ such that for any function $v \in W^{1,2}\left( B_r , \A_{Q}(\R^{d}) \right)$ one has:
\begin{equation} \label{Q-Poincare}
\int_{B_r} |v|^{2} \, {\rm d}y \leq \left( \frac{1}{m} + \varepsilon \right) r \int_{\partial B_r} |v|^{2} \, \dHa^{m-1} + C_{\varepsilon} r^{2} \Dir(v,B_r).
\end{equation}
\end{corollary}

\begin{proof} 
Fix $\varepsilon > 0$ and $v \in W^{1,2}(B_{r} , \A_{Q}(\R^d))$, and apply Lemma \ref{Poincare} to the function $g = |v| = \G(v, Q \llbracket 0 \rrbracket) \in W^{1,2}(B_{r})$. The inequality \eqref{Q-Poincare} then follows immediately, because $g|_{\partial B_{r}} = \left| v|_{\partial B_{r}} \right|$ and $|\partial_{j} g| \leq |\partial_{j}v|$ for every $j = 1,\dots,m$.
\end{proof}
\,

\begin{proof}[Proof of Proposition \ref{Prop Holder}]
Let $N \in \Gamma_{Q}^{1,2}(\N \Omega)$ be $\Jac$-minimizing, and fix any point $p \in \Omega$. For every radius $r < \min\lbrace {\rm inj}(\Sigma), \ddist(p, \partial\Omega) \rbrace$ the exponential map $\exp_{p}$ maps the Euclidean ball $B_{r}(0) \subset T_{p}\Sigma$ diffeomorphically onto the geodesic ball $\B_{r}(p) \subset \Sigma$, and the composition $u := N \circ \exp_{p}$ is a $W^{1,2}$ $Q$-valued map defined in $B_{r}$.

Let now $f \in W^{1,2}\left(B_r,\A_{Q}(\R^{d})\right)$ be $\Dir$-minimizing in $B_r$ such that $f|_{\partial B_r} = u|_{\partial B_r}$ \footnote{Recall that the existence of such a map $f$ is guaranteed by Theorem \ref{ex_reg}.}, and set $h := f \circ \exp_{p}^{-1}$. Then, the normal bundle projection $h^{\perp} \in \Gamma_{Q}^{1,2}(\N \mathbf{B}_{r}(p))$ satisfies $h^{\perp}|_{\partial \mathbf{B}_{r}(p)} = N|_{\partial \mathbf{B}_{r}(p)}$ and is therefore a competitor for the Jacobi functional. Hence, using minimality, the definition of the Jacobi functional and \eqref{B_est}, we deduce:
\begin{equation}
\Jac(N,\mathbf{B}_{r}(p)) \leq \Jac(h^{\perp},\mathbf{B}_{r}(p)) \leq \Dir(h^{\perp},\mathbf{B}_{r}(p)) + C_{0} \int_{\B_{r}(p)} |h|^{2} \, \dHa^m,
\end{equation}
which in turn produces
\begin{equation}
\Dir(N, \B_{r}(p)) \leq \Dir(h^{\perp}, \B_{r}(p)) + C_{0} \left(\int_{\B_{r}(p)} |h|^{2} \, \dHa^m + \int_{\B_{r}(p)} |N|^{2} \, \dHa^m\right).
\end{equation}

Hence, combining Lemma \ref{Projection} with \eqref{ab1} and \eqref{ab2}, we can conclude that for any $\varepsilon_{1} \in \left( 0,1 \right)$ there exists a radius $0 < r_{\varepsilon_{1}} < {\rm inj}(\Sigma)$ such that the estimate
\begin{equation} \label{Key Est_0}
\Dir(u,B_r) \leq \left( 1 + \varepsilon_{1} \right) \Dir(f,B_r) + C_{\varepsilon_{1}} \left( \int_{B_r} |f|^2 \, {\rm d}y + \int_{B_r} |u|^{2} \, {\rm d}y \right),
\end{equation}
holds true whenever $r \leq r_{\varepsilon_{1}}$.

Now we apply \cite[Proposition 3.10]{DLS11a}: since $f$ is $\Dir$-minimizing in $B_r$, the estimate
\begin{equation} \label{Key Est}
\Dir(f,B_r) \leq C(m) r \Dir(u,\partial B_r)
\end{equation}
holds with constants $C(2) = Q$ and $C(m) < (m-2)^{-1}$ for $m \geq 3$ whenever $\Dir(u, \partial B_{r})$ is finite, and thus for a.e. $r$.
Combining \eqref{Key Est_0} with \eqref{Key Est}, we deduce that we can choose $\varepsilon_{1} = \varepsilon_{1}(m,Q)$ so small that the inequality
\begin{equation} \label{Reg1}
\Dir(u,B_r) \leq C(m) r \Dir(u,\partial B_r) +  C \left( \int_{B_r} |f|^2 \, {\rm d}y + \int_{B_r} |u|^{2} \, {\rm d}y \right)
\end{equation}
holds with, say, $C(2) = 2Q$ and again $C(m) < (m-2)^{-1}$ when $m \geq 3$ for a.e. $r \leq r_{m,Q}$.

Now, fix $\varepsilon > 0$ and apply the result of Corollary \ref{cor Q-Poincare} first with $v = f$ and then with $v = u$, and plug the resulting inequalities in \eqref{Reg1}. Using the fact that $f$ and $u$ have the same boundary value and that $\Dir(f,B_r) \leq \Dir(u,B_r)$, we obtain the following key inequality:
\begin{equation}
\Dir(u,B_r) \leq C(m) r \Dir(u,\partial B_r) + C \left( \frac{1}{m} + \varepsilon \right) r \int_{\partial B_r} |u|^{2} \dHa^{m-1} + C_{\varepsilon} r^{2} \Dir(u,B_r).
\end{equation}

This implies the following: for every $\Lambda > 0$ one has
\begin{equation}
\begin{split}
\Dir(u,B_r) + \Lambda \int_{B_r} |u|^2 \, {\rm d}y &\leq C(m) r \Dir(u,\partial B_r) \\
                                             & + \left( C + \Lambda \right) \left( \frac{1}{m} + \varepsilon \right) r \int_{\partial B_r} |u|^2 \, \dHa^{m-1} \\
                                             & + C_{\varepsilon,\Lambda} r^{2} \Dir(u,B_r).
\end{split}
\end{equation}

For a suitable choice of $\Lambda = \Lambda_{m,\varepsilon} \gg 1$ this yields:
\begin{equation}
\left( 1 - C_{m,\varepsilon} r^{2} \right) \Dir(u,B_r) + \Lambda \int_{B_r} |u|^2 \, {\rm d}y \leq  C(m) r \Dir(u,\partial B_r) + \Lambda \left( \frac{1}{m} + 2\varepsilon \right) r \int_{\partial B_r} |u|^2 \, \dHa^{m-1}.
\end{equation}

Finally, we divide the whole inequality by $1 - C_{m,\varepsilon} r^{2}$ and conclude that if $r$ is sufficiently small, say $r \leq r_{m,\varepsilon,Q}$ then the inequality
\begin{equation}
\Dir(u,B_r) + \Lambda \int_{B_r} |u|^2 \, {\rm d}y \leq C(m) r \Dir(u,\partial B_r) + \Lambda \left( \frac{1}{m} + 4\varepsilon \right) r \int_{\partial B_r} |u|^2 \, \dHa^{m-1}
\end{equation}
holds with a possible new choice of $C(m)$, say $C(2) = 4Q$ and still $C(m) < (m-2)^{-1}$ for $m \geq 3$. The conclusion immediately follows, by choosing $\varepsilon = \varepsilon(m, Q)$ in such a way that $\frac{1}{m} + 4 \varepsilon < 4Q$ when $m = 2$ and $\frac{1}{m} + 4 \varepsilon < \frac{1}{m-2}$ when $m \geq 3$.
\end{proof}

We have now all the ingredients that are needed to prove Theorem \ref{reg_Holder}: as announced at the beginning of the section, the proof can be easily achieved by combining our Proposition \ref{Prop Holder} with the Campanato-Morrey estimates \ref{CM}. 

\begin{proof}[Proof of Theorem \ref{reg_Holder}]
Let $r_{0}$ be the radius given in Proposition \ref{Prop Holder}. Fix any point $p \in \Omega$, and assume without loss of generality that $\B_{r_{0}}(p) \Subset \Omega$. Consider the corresponding exponential map $\exp_{p} \colon B_{r_{0}}(0) \subset T_{p}\Sigma \to \B_{r_{0}}(p) \subset \Sigma$, and set $u := N \circ \exp_{p}$. By Proposition \ref{Prop Holder}, for a.e. radius $r \leq r_0$ the inequality \eqref{Diff Ineq 1} is satisfied with universal constants $\Lambda$ and $C(m)$, with $C(m) < (m-2)^{-1}$ when $m \geq 3$. We set: 
\begin{equation}
\gamma(m) :=
\begin{cases}
C(m)^{-1} &\mbox{ if } m=2 \\
C(m)^{-1} - m + 2 &\mbox{ if } m \geq 3,
\end{cases}
\end{equation}
and we denote by $\phi = \phi(r)$ the absolutely continuous function
\begin{equation}
\phi(r) := \Dir(u,B_r) + \Lambda \int_{B_r} |u|^{2} \, {\rm d}y
\end{equation}
for $r \in \left(0, r_0\right]$. By \eqref{Diff Ineq 1}, $\phi$ satisfies the differential inequality
\begin{equation} \label{Diff Ineq 2}
(m - 2 + \gamma) \phi \leq r \phi'
\end{equation}
almost everywhere in the interval $\left( 0, r_{0} \right]$. Integrating \eqref{Diff Ineq 2} we obtain:
\begin{equation} \label{Dir_decay}
\Dir(u,B_r) \leq \phi(r) \leq \frac{\phi(r_0)}{r_{0}^{m-2+\gamma}} r^{m-2+\gamma} =: A r^{m-2+\gamma}.
\end{equation}
As a consequence of the Campanato - Morrey estimates, Proposition \ref{CM}, we conclude that $u$ is H\"older continuous with exponent $\alpha := \frac{\gamma}{2}$, with
\begin{equation} \label{Holder_sn}
[u]_{C^{0,\alpha}\left(\overline{B}_{\theta r_0}\right)} := \sup_{y_{1}, y_{2} \in \overline{B}_{\theta r_0}} \frac{\G\left( u(y_1), u(y_2)\right)}{|y_{1} - y_{2}|^{\alpha}} \leq C \sqrt{A},
\end{equation}
for any $0 < \theta < 1$ and for a constant $C = C(m,d,Q,\theta)$.

The estimate \eqref{reg_Holder:est} is an immediate consequence of \eqref{Holder_sn} and the properties of the exponential map.
\end{proof}

\section{First variation formulae and the analysis of the frequency function} \label{sec:frequency}

In this section we start the machinery that will eventually lead us, in Section \ref{sec:sing}, to the proof of Theorem \ref{sing:thm}, according to which a Jacobi $Q$-field $N$ is in fact the ``superposition'' of $Q$ classical Jacobi fields in a neighborhood of all points of the domain $\Omega$ \emph{with the exception} of those belonging to a \emph{singular set} of small Hausdorff dimension.

The first step towards this result consists of deriving some Euler-Lagrange conditions for $\Jac$-minimizing multivalued maps. Throughout the whole section, we will assume, as usual,  that $N$ is a $Q$-valued section of the normal bundle of $\Sigma$ in $\M$ defined in an open set $\Omega$, where it minimizes the Jacobi functional as specified in Definition \ref{Jac_min}.

\subsection{First variations}

Suppose that for some $\delta > 0$ we have a 1-parameter family $\lbrace N_{s} \rbrace_{s \in \left(-\delta, \delta \right)} \subset \Gamma_{Q}^{1,2}(\N\Omega)$ such that $N_{0} = N$ and $N_{s} \equiv N$ in a neighborhood of $\partial\Omega$ for all $s$. Then, the minimization property of $N$ implies that the map $s \mapsto \Jac(N_{s}, \Omega)$ takes its minimum at $s = 0$, and thus
\begin{equation}\label{EL}
\left.\frac{d}{ds} \Jac(N_{s}, \Omega) \right|_{s=0}= 0
\end{equation} 
whenever the derivative on the left exists. The family $\lbrace N_{s} \rbrace$ is called an (admissible) 1-parameter family of variations of $N$ in $\Omega$, and formula \eqref{EL} is the first variation formula corresponding to the given variation.

We will consider two natural types of variations in order to perturb the map $N$. The \emph{inner variations} are generated by right compositions with diffeomorphisms of the domain and by a suitable ``orthogonalization procedure''; the \emph{outer variations} correspond instead to ``left compositions''. The relevant definition is the following.
\begin{definition} \label{in_out_var}
Let $N = \sum_{\ell} \llbracket N^{\ell} \rrbracket \in \Gamma_{Q}^{1,2}(\N\Omega)$ be $\Jac$-minimizing in $\Omega$.
\begin{itemize}
\item[(OV)] Given $\psi \in C^{1}( \Omega \times \R^d, \R^d )$ such that $\spt(\psi) \subset \Omega' \times \R^d$ for some $\Omega' \Subset \Omega$ and $\psi(x,u) \in T_{x}^{\perp}\Sigma \subset T_{x}\M$ for all $(x,u) \in \Omega \times T_{x}^{\perp}\Sigma$, an admissible variation of $N$ in $\Omega$ can be defined by $N_{s}(x) := \sum_{\ell=1}^{Q} \llbracket N^{\ell}(x) + s \psi( x, N^{\ell}(x) ) \rrbracket$. Such a family is called outer variation (OV);

\item[(IV)] Given a $C^1$ vector field $X$ of $\T\Sigma$ compactly supported in $\Omega$, for $s$ sufficiently small the map $x \mapsto \Phi_{s}(x) := \exp_{x}\left( s X(x) \right)$ is a diffeomorphism of $\Omega$ which leaves $\partial \Omega$ fixed. As a consequence, the family $\lbrace N_{s} \rbrace$ defined by $N_{s} := \left( N \circ \Phi_{s} \right)^{\perp}$ is an admissible variation of $N$ in $\Omega$, which we call inner variation (IV).
\end{itemize}

\end{definition}

In the next proposition, we obtain an explicit formulation of \eqref{EL} in the case of outer variations induced by maps $\psi$ as above. Consistently with the notation introduced for multi-fields in Definition \ref{def:normal_dirichlet}, given $(x,u) \in \Omega \times \R^d$ we will denote by $\nabla^{\perp}\psi(x,u)$ the linear operator $T_{x}\Sigma \to T_{x}^{\perp}\Sigma$ obtained by projecting $D_{x}\psi(x,u)$ onto $T_{x}^{\perp}\Sigma$ at every $x$. Also, recall the definitions of $A \cdot u$, $u$ being a (single-valued) section of $\N\Omega$, and of the quadratic form $\Ri$. The symbol $\langle L \, \colon \, M \rangle$ will be used to denote the usual Hilbert-Schmidt scalar product of two matrices $L$ and $M$. 

\begin{proposition}[Outer variation formula] \label{OV}
Let $\psi$ be as in \textrm{(OV)} and such that
\begin{equation} \label{OV:cond}
|D_{u}\psi| \leq C < \infty \hspace{0.2cm} \mbox{ and } \hspace{0.2cm} |\psi| + |D_{x}\psi| \leq C(1 + |u|).
\end{equation}
Then, the first variation formula corresponding to the outer variation $N_{s}$ defined by $\psi$ is
\begin{equation}\label{OV:eq}
\int_{\Omega} \sum_{\ell=1}^{Q} \langle \nabla^{\perp}N^{\ell}(x) \, \colon \left( \nabla^{\perp}\psi(x,N^{\ell}(x)) + D_{u}\psi(x,N^{\ell}(x)) \cdot DN^{\ell}(x) \right) \rangle \, \dHa^m(x) = \mathcal{E}_{{\rm OV}}(\psi),
\end{equation}
where
\begin{equation} \label{OV:error}
\mathcal{E}_{{\rm OV}}(\psi) := \int_{\Omega} \sum_{\ell=1}^{Q} \left( \langle A \cdot N^{\ell}(x) \, \colon A \cdot \psi(x, N^{\ell}(x)) \rangle + \Ri(N^{\ell}(x), \psi(x, N^{\ell}(x))) \right) \, \dHa^m(x).
\end{equation}
\end{proposition}

\begin{proof}
The proof is straightforward: using \eqref{clean_Q_Jac:eq} with $N_s$ in place of $u$, it suffices to differentiate in $s$ and recall that $\Ri$ is a symmetric quadratic form on the normal bundle of $\Sigma$ in $\M$ (the hypotheses in \eqref{OV:cond} ensure the summability of the various integrands involved in the computation). 
\end{proof}

An explicit formula for \eqref{EL} in the case of inner variations induced by vector fields $X$ as in Definition \ref{in_out_var} is the content of the following proposition. Recall that $\overline{A}$ denotes the second fundamental form of the embedding $\M \hookrightarrow \R^{d}$.

\begin{proposition}[Inner variation formula] \label{IV}
Let $X$ be as in \textrm{(IV)}. Then, the first variation formula corresponding to the inner variation $N_{s}$ defined by the family $\Phi_{s}$ of diffeomorphisms induced by $X$ is
\begin{equation} \label{IV:eq}
- \int_{\Omega} |\nabla^{\perp}N|^{2} \di_{\Sigma}(X) \, \dHa^m + 2 \int_{\Omega} \sum_{\ell=1}^{Q} \langle \nabla^{\perp}N^{\ell} \, \colon \, \nabla^{\perp}N^\ell \cdot \nabla^{\Sigma}X \rangle \, \dHa^m = \mathcal{E}_{{\rm IV}}(X),
\end{equation}
where
\[
\mathcal{E}_{{\rm IV}}(X) = \mathcal{E}_{{\rm IV}}^{(1)}(X) + \mathcal{E}_{{\rm IV}}^{(2)}(X) + \mathcal{E}_{{\rm IV}}^{(3)}(X)
\]
is defined by
\begin{equation} \label{IV:error1}
\mathcal{E}_{{\rm IV}}^{(1)}(X) := 2 \int_{\Omega} \sum_{\ell=1}^{Q} \left( \tr_{\Sigma}(\langle \overline{A}(\cdot, N^{\ell}), \overline{A}(X, \nabla^{\perp}_{(\cdot)}N^{\ell}) \rangle) - \tr_{\Sigma}(\langle \overline{A}(X, N^\ell), \overline{A}(\cdot, \nabla^{\perp}_{(\cdot)}N^\ell) \rangle) \right) \, \dHa^{m},
\end{equation}
\begin{equation} \label{IV:error2}
\mathcal{E}_{{\rm IV}}^{(2)}(X) := 2 \int_{\Omega} \sum_{\ell=1}^{Q} \langle A \cdot N^{\ell} \, \colon \, A \cdot \nabla^{\perp}_{X}N^{\ell} \rangle \, \dHa^m,
\end{equation}
and
\begin{equation} \label{IV:error3}
\mathcal{E}_{{\rm IV}}^{(3)}(X) := 2 \int_{\Omega} \sum_{\ell=1}^{Q} \Ri(N^\ell, \nabla^{\perp}_{X}N^\ell) \, \dHa^m.
\end{equation}
\end{proposition}

\begin{proof}
Fix the vector field $X$, and consider the associated variation $\lbrace N_{s} \rbrace$, with $N_{s} = \sum_{\ell} \llbracket N_{s}^{\ell} \rrbracket$ defined by
\begin{equation} \label{Coord_repr}
N_{s}^{\ell}(x) = (N^{\ell} \circ \Phi_{s})^{\perp}(x) = \sum_{\beta = 1}^{k} \langle N^{\ell}(\Phi_{s}(x)), \nu_{\beta}(x) \rangle \nu_{\beta}(x),
\end{equation}
$\left( \nu_{\beta} \right)_{\beta=1}^{k}$ being a (local) orthonormal frame of $\N\Omega$. Recall that $\Phi_{0} = \mathrm{Id}_{\Omega}$ and that $\left.\frac{\partial}{\partial s} \Phi_{s}(x) \right|_{s=0} = X(x)$. Now, using \eqref{clean_Q_Jac:eq}, we write the first variation formula as 
\begin{equation} \label{IV_0}
\begin{split}
0 &= \left.\frac{d}{ds} \Jac(N_{s}, \Omega)\right|_{s=0} \\
   &= \underbrace{\frac{d}{ds}\bigg|_{s=0} \Dir^{\N\Sigma}(N_{s}, \Omega)}_{=: I_1} + \underbrace{\left( - \frac{d}{ds}\bigg|_{s=0} \int_{\Omega} \sum_{\ell=1}^{Q} | A \cdot N_{s}^{\ell} |^{2} \, \dHa^m \right)}_{=: I_2} + \underbrace{\left( - \frac{d}{ds}\bigg|_{s=0} \int_{\Omega} \sum_{\ell=1}^{Q} \Ri(N_{s}^{\ell}, N_{s}^{\ell}) \, \dHa^m\right)}_{=:I_3},
\end{split}
\end{equation}
and we will work on the three terms separately.

\textit{Step 1: computing $I_{1}$.} Write $I_{1} = \sum_{\ell} I_{1}^{\ell}$, where
\begin{equation} \label{I1}
I_{1}^{\ell} = \frac{d}{ds}\bigg|_{s=0} \int_{\Omega} \sum_{i=1}^{m} \sum_{\alpha=1}^{k} |\langle D_{\xi_i}N_{s}^{\ell}, \nu_{\alpha} \rangle|^{2} \, \dHa^m.
\end{equation}
Using the representation formula \eqref{Coord_repr}, one immediately computes
\begin{equation} \label{I1_2}
\begin{split}
\langle D_{\xi_i}N_{s}^{\ell}, \nu_{\alpha} \rangle (x) &= \langle DN^{\ell}|_{\Phi_{s}(x)} \cdot D\Phi_{s}|_{x} \cdot \xi_{i}(x), \nu_{\alpha}(x) \rangle \\
&+ \langle N^{\ell}(\Phi_{s}(x)), D_{\xi_i}\nu_{\alpha}(x) \rangle \\
&+ \sum_{\beta=1}^{k} \langle N^{\ell}(\Phi_{s}(x)), \nu_{\beta}(x) \rangle \langle D_{\xi_i}\nu_{\beta}(x), \nu_{\alpha}(x) \rangle.
\end{split} 
\end{equation}
Now, since $\langle \nu_{\alpha}, \nu_{\beta} \rangle = \delta_{\alpha\beta}$, we have that $\langle D_{\xi_i}\nu_{\beta}, \nu_{\alpha} \rangle = - \langle \nu_{\beta}, D_{\xi_i} \nu_{\alpha} \rangle$, so that the last term in formula \eqref{I1_2} becomes
\begin{equation}
- \sum_{\beta=1}^{k} \langle N^{\ell}(\Phi_{s}(x)), \nu_{\beta}(x) \rangle \langle D_{\xi_i}\nu_{\alpha}(x), \nu_{\beta}(x) \rangle = - \langle N^{\ell}(\Phi_{s}(x)), \nabla^{\perp}_{\xi_i} \nu_{\alpha}(x) \rangle,
\end{equation}
and we can write
\begin{equation} \label{I1_3}
\begin{split}
\langle D_{\xi_i}N_{s}^{\ell}, \nu_{\alpha} \rangle(x) &= \langle DN^{\ell}|_{\Phi_{s}(x)} \cdot D\Phi_{s}|_{x} \cdot \xi_{i}(x), \nu_{\alpha}(x) \rangle \\
&+ \langle N^{\ell}(\Phi_{s}(x)), (D_{\xi_i}\nu_{\alpha} - \nabla^{\perp}_{\xi_i} \nu_{\alpha})(x) \rangle.
\end{split}
\end{equation}
For small values of the parameter $s$, the map $\Phi_{s}$ is a diffeomorphism of $\Omega$, and we will denote by $\Phi_{s}^{-1}$ its inverse. Then, we can change variable $x = \Phi_{s}^{-1}(y)$ in the integral, and finally write
\begin{equation} \label{I1_4}
I_{1}^{\ell} = \frac{d}{ds}\bigg|_{s=0} \int_{\Omega} \sum_{i=1}^{m} \sum_{\alpha=1}^{k} |g^{\ell}_{i\alpha}(s,y)|^{2} \mathbf{J}\Phi_{s}^{-1}(y) \, \dHa^{m}(y),
\end{equation}
where $\mathbf{J}\Phi_{s}^{-1}$ is the Jacobian determinant of $D\Phi_{s}^{-1}$ and 
\begin{equation}
g^{\ell}_{i\alpha}(s,y) = \langle DN^{\ell}|_{y} \cdot \zeta_{i}(s,y), \nu_{\alpha}(\Phi_{s}^{-1}(y)) \rangle + \langle N^{\ell}(y), (D_{\xi_i}\nu_{\alpha} - \nabla^{\perp}_{\xi_i}\nu_{\alpha})(\Phi_{s}^{-1}(y)) \rangle,
\end{equation}
with $\zeta_{i}(s,y) := D\Phi_{s}|_{\Phi_{s}^{-1}(y)} \cdot \xi_{i}(\Phi_{s}^{-1}(y))$. Hence, we have:
\begin{equation} \label{I1_5}
\begin{split}
I_{1}^{\ell} &= - \int_{\Omega} \sum_{i=1}^{m} \sum_{\alpha=1}^{k} |g^{\ell}_{i\alpha}(0,y)|^{2} \di_{\Sigma}(X) \, \dHa^m + 2 \int_{\Omega} \sum_{i=1}^{m} \sum_{\alpha=1}^{k} g^{\ell}_{i\alpha}(0,y) \partial_{s}g^{\ell}_{i\alpha}(0,y) \, \dHa^m \\
&= - \int_{\Omega} \sum_{i=1}^{m} \sum_{\alpha=1}^{k} |g^{\ell}_{i\alpha}(0,y)|^{2} \di_{\Sigma}(X) \, \dHa^m + \int_{\Omega} \frac{\partial}{\partial s} \left( \sum_{i=1}^{m} \sum_{\alpha=1}^{k} |g^{\ell}_{i\alpha}(s,y)|^{2} \right)\bigg|_{s=0} \, \dHa^{m}(y).
\end{split}
\end{equation}
Now, since 
\[
\sum_{i=1}^{m} \sum_{\alpha=1}^{k} |g^{\ell}_{i\alpha}(s,y)|^{2} = |\nabla^{\perp}N^{\ell}_{s}|^{2}(\Phi_{s}^{-1}(y)),
\]
its value is independent of the orthonormal frame chosen: thus, having fixed a point $y \in \Omega$, we can impose $\nabla \xi_{i} = \nabla \nu_{\alpha} = 0$ at $y$.

We can now proceed computing explicitly \eqref{I1_5}. Clearly, $g^{\ell}_{i\alpha}(0,y) = \langle D_{\xi_i}N^{\ell}, \nu_{\alpha} \rangle(y)$, so we are only left with the computation of $\partial_{s}g^{\ell}_{i\alpha}(0,y)$. We start observing that
\begin{equation}
\partial_{s}\zeta_{i}(0,y) = ( D_{\xi_i}X - D_{X} \xi_i )(y) = - \left[ X, \xi_i \right](y),
\end{equation}
from which we easily deduce
\begin{equation}
\begin{split}
\partial_{s}|_{s=0} \left( \langle DN^{\ell}|_{y} \cdot \zeta_{i}(s,y), \nu_{\alpha}(\Phi_{s}^{-1}(y)) \rangle \right) &= - \langle D_{\left[ X, \xi_i \right]}N^{\ell}, \nu_{\alpha} \rangle - \langle D_{\xi_i}N^{\ell}, D_{X}\nu_{\alpha} \rangle \\
&= \langle D_{\nabla^{\Sigma}_{\xi_i}X}N^{\ell}, \nu_{\alpha} \rangle - \langle \overline{A}(\xi_i, N^\ell), \overline{A}(X, \nu_{\alpha}) \rangle,
\end{split}
\end{equation}
where we have used that $\nabla_{X}\xi_i = \nabla_{X}\nu_{\alpha} = 0$ at $y$ (and, therefore, $\left[ X, \xi_i \right](y) = - \nabla^{\Sigma}_{\xi_i}X(y)$ and $D_{X}\nu_{\alpha} = \overline{A}(X, \nu_{\alpha})$).

On the other hand, 
\begin{equation}
\begin{split}
\partial_{s}|_{s=0} \left( \langle N^{\ell}(y), (D_{\xi_i}\nu_{\alpha} - \nabla^{\perp}_{\xi_i}\nu_{\alpha})(\Phi_{s}^{-1}(y)) \rangle \right) &= - \langle N^{\ell}, D_{X}(D_{\xi_i}\nu_{\alpha} - \nabla^{\perp}_{\xi_i}\nu_{\alpha}) \rangle \\
&= \langle D_{X}N^{\ell}, D_{\xi_i}\nu_{\alpha} - \nabla^{\perp}_{\xi_i}\nu_{\alpha} \rangle \\
&= \langle \overline{A}(X, N^\ell), \overline{A}(\xi_i, \nu_{\alpha}) \rangle
\end{split}
\end{equation}
because the fields $N^\ell$ and $D_{\xi_i}\nu_{\alpha} - \nabla^{\perp}_{\xi_i}\nu_{\alpha}$ are mutually orthogonal and, again, because $\nabla \nu_{\alpha} = 0$ at $y$.

This allows to conclude:
\begin{equation} \label{I1_fin}
\begin{split}
I_{1} = &- \int_{\Omega} |\nabla^{\perp}N|^{2} \di_{\Sigma}(X) \, \dHa^m + 2 \int_{\Omega} \sum_{\ell=1}^{Q} \langle \nabla^{\perp}N^{\ell} \, \colon \, \nabla^{\perp}N^\ell \cdot \nabla^{\Sigma}X \rangle \, \dHa^m \\
&+ 2 \int_{\Omega} \sum_{\ell=1}^{Q} \left( \tr_{\Sigma}(\langle \overline{A}(X, N^\ell), \overline{A}(\cdot, \nabla^{\perp}_{(\cdot)}N^\ell) \rangle) - \tr_{\Sigma}(\langle \overline{A}(\cdot, N^\ell), \overline{A}(X, \nabla^{\perp}_{(\cdot)}N^\ell) \rangle) \right) \, \dHa^m \\
&= - \int_{\Omega} |\nabla^{\perp}N|^{2} \di_{\Sigma}(X) \, \dHa^m + 2 \int_{\Omega} \sum_{\ell=1}^{Q} \langle \nabla^{\perp}N^{\ell} \, \colon \, \nabla^{\perp}N^\ell \cdot \nabla^{\Sigma}X \rangle \, \dHa^m - \mathcal{E}_{{\rm IV}}^{(1)}(X).
\end{split}
\end{equation}

\textit{Step 2: computing $I_{2}$.} Write $I_{2} = \sum_{\ell} I_{2}^{\ell}$, where 
\begin{equation}
I_{2}^{\ell} = - \frac{d}{ds}\bigg|_{s=0} \int_{\Omega} | A \cdot N_{s}^{\ell} |^{2}\, \dHa^m.
\end{equation}
Since the tensor $A$ takes values in the normal bundle of $\Sigma$, clearly $A \cdot N_{s}^{\ell} = A \cdot (N^{\ell} \circ \Phi_{s})$, whence
\begin{equation}\label{I2_2}
\int_{\Omega} | A \cdot N_{s}^{\ell} |^{2}\, \dHa^m = \int_{\Omega} \sum_{i,j=1}^{m} |\langle A_{x}(\xi_{i}(x), \xi_{j}(x)), N^{\ell}(\Phi_{s}(x))|^{2} \, \dHa^m(x).
\end{equation}
We can now differentiate in $s$ and evaluate for $s = 0$ in formula \eqref{I2_2} to obtain:
\begin{equation} \label{I2_3}
I_{2}^{\ell} = - 2 \int_{\Omega} \sum_{i,j=1}^{m} \langle A_{x}(\xi_i(x), \xi_j(x)), N^{\ell}(x) \rangle \langle A_{x}(\xi_i(x), \xi_j(x)), D_{X}N^{\ell}(x) \rangle,
\end{equation}
which readily yields
\begin{equation} \label{I2_fin}
I_{2} = - 2 \int_{\Omega} \sum_{\ell=1}^{Q} \langle A \cdot N^{\ell} \, \colon \, A \cdot \nabla^{\perp}_{X}N^{\ell} \rangle \, \dHa^m = - \mathcal{E}_{{\rm IV}}^{(2)}(X).
\end{equation}

\textit{Step 3: computing $I_{3}$.} As before, write $I_{3} = \sum_{\ell} I_{3}^{\ell}$, where
\begin{equation}
I_{3}^{\ell} = - \frac{d}{ds}\bigg|_{s=0} \int_{\Omega} \Ri(N_{s}^{\ell}, N_{s}^{\ell}) \, \dHa^m.
\end{equation}

Now, it suffices to differentiate in $s$ and evaluate at $s=0$ inside the integral keeping in mind that $\Ri$ is a symmetric $2$-tensor to get
\begin{equation} \label{I3_fin}
I_{3} = - 2 \int_{\Omega} \sum_{\ell=1}^{Q} \Ri(N^\ell, \nabla^{\perp}_{X}N^\ell) \, \dHa^m = - \mathcal{E}_{{\rm IV}}^{(3)}(X).
\end{equation}  

\textit{Conclusion.} The statement, formula \eqref{IV:eq}, is immediately obtained by plugging equations \eqref{I1_fin}, \eqref{I2_fin} and \eqref{I3_fin} into \eqref{IV_0}.

\end{proof}

The first variation formulae \eqref{OV:eq} and \eqref{IV:eq} will play a fundamental role in the next section to discuss the almost monotonicity properties of the frequency function. Before proceeding, we apply the outer variation formula to show that minimizers of the $\Jac$ functional enjoy a Caccioppoli type inequality. 

\begin{proposition}[Caccioppoli inequality] \label{rpis:thm}
There exists a geometric constant $C > 0$ such that for any $\Jac$-minimizing $Q$-valued map $N = \sum_{\ell} \llbracket N^{\ell} \rrbracket \in \Gamma_{Q}^{1,2}(\N\Omega)$ the inequality
\begin{equation} \label{rpis:eq1}
\int_{\Omega} \eta(x)^{2} |\nabla^{\perp}N(x)|^{2} \, \dHa^{m}(x) \leq 4 \int_{\Omega} |D\eta(x)|^{2} |N(x)|^{2} \, \dHa^{m}(x) + C \int_{\Omega} \eta(x)^{2} |N(x)|^{2} \, \dHa^{m}(x) 
\end{equation} 
holds for any choice of $\eta \in C^{1}_{c}(\Omega)$. In particular, for every $p \in \Omega$ and for every $r < \min\left\lbrace {\rm inj}(\Sigma), \ddist(p, \partial \Omega) \right\rbrace$ one has
\begin{equation} \label{rpis:eq2}
\int_{\B_{\frac{r}{2}}(p)} |\nabla^{\perp}N|^{2} \, \dHa^{m} \leq \frac{C}{r^{2}} \int_{\B_{r}(p)} |N|^{2} \, \dHa^{m}.
\end{equation} 
\end{proposition}

\begin{proof}
Fix $N$ and $\eta$ as in the statement, and apply the outer variation formula \eqref{OV:eq} with $\psi(x,u) := \eta(x)^{2} u$. Since $D_{x}\psi(x,u) = 2 \eta(x) u \otimes D\eta(x)$ and $D_{u}\psi(x,u) = \eta(x)^{2} {\rm Id}$, for this choice of $\psi$ the outer variation formula reads
\begin{equation} \label{rpis:1}
\int_{\Omega} \eta^{2} |\nabla^{\perp}N|^{2} + 2 \sum_{\ell=1}^{Q} \langle \eta \nabla^{\perp}N^{\ell} \, \colon \, N^{\ell} \otimes D\eta \rangle \, \dHa^{m} = \int_{\Omega} \eta^{2} \sum_{\ell=1}^{Q} \left( |A \cdot N^{\ell}|^{2} + \Ri(N^{\ell}, N^{\ell}) \right) \, \dHa^{m}.
\end{equation}
Applying Young's inequality we immediately deduce that for any $\delta > 0$ one has
\begin{equation} \label{rpis:2}
\int_{\Omega} \eta^{2} |\nabla^{\perp}N|^{2} \, \dHa^{m} \leq \delta \int_{\Omega} \eta^{2} |\nabla^{\perp}N|^{2} \, \dHa^{m} + \frac{1}{\delta} \int_{\Omega} |D \eta|^{2} |N|^{2} \, \dHa^{m} + C \int_{\Omega} \eta^{2} |N|^{2} \, \dHa^{m},
\end{equation}
for a constant $C = C({\bf A}, {\bf R})$, where, we recall, ${\bf A} = \|A\|_{L^{\infty}}$ and ${\bf R} = \| R \|_{L^{\infty}}$ are defined in \eqref{sup_norms_geometric_A} and \eqref{sup_norms_geometric_R}. Choose $\delta = \frac{1}{2}$ to obtain \eqref{rpis:eq1}. In order to deduce \eqref{rpis:eq2}, apply \eqref{rpis:eq1} with $\eta(x) := \phi\left( \frac{d(x)}{r} \right)$, where $d(x) := {\bf d}(x,p)$ and $\phi$ is a cut-off function $0 \leq \phi \leq 1$ such that $\phi(t) = 1$ for $0 \leq t \leq \frac{1}{2}$, $\phi(t) = 0$ for $t \geq 1$ and $|\phi'| \leq 2$. 
\end{proof}

\subsection{Almost monotonicity of the frequency function and its consequences}

The next step towards the proof of Theorem \ref{sing:thm} consists of a careful asymptotic analysis of the celebrated \emph{frequency function}. 

\begin{definition}[Frequency function]
Fix any point $p \in \Omega$. For any radius $0 < r < \min\lbrace {\rm inj}(\Sigma), \ddist(p, \partial\Omega)\rbrace$, define the \textit{energy function}
\begin{equation} \label{D_energy}
{\bf D}_{N,p}(r) := \int_{\B_{r}(p)} |\nabla^{\perp}N|^{2}(x) \, \dHa^{m}(x)
\end{equation}
and the \textit{height function}
\begin{equation} \label{H_height}
{\bf H}_{N,p}(r) := \int_{\partial \B_{r}(p)} |N|^{2}(x) \, \dHa^{m-1}(x).
\end{equation}
The \textit{frequency function} is then defined by
\begin{equation} \label{freq:eq}
{\bf I}_{N,p}(r) := \frac{r {\bf D}_{N,p}(r)}{{\bf H}_{N,p}(r)}
\end{equation}
for all $r$ such that ${\bf H}_{N,p}(r) > 0$. When the $Q$-field $N$ and the point $p$ are fixed and there is no ambiguity, we will drop the subscripts and simply write ${\bf D}(r)$, ${\bf H}(r)$ and ${\bf I}(r)$.
\end{definition}

\begin{remark}
Observe that ${\bf D}(r) = \Dir^{\N\Sigma}(N, \B_{r}(p))$; ${\bf D}$ is an absolutely continuous function with derivative 
\[
{\bf D}'(r) = \int_{\partial \B_{r}(p)} |\nabla^{\perp}N|^{2} \, \dHa^{m-1} 
\]
almost everywhere. As for ${\bf H}(r)$, note that $|N|$ is the composition of $N$ with the Lipschitz function $\G(\cdot, Q\llbracket 0 \rrbracket)$, thus it belongs to $W^{1,2}$. Hence, $|N|^{2}$ is a $W^{1,1}$ function, and also ${\bf H} \in W^{1,1}$.
\end{remark}

\begin{remark} \label{well_posedness}
It is an easy consequence of the H\"older regularity of $N$ that the frequency function ${\bf I}(r)$ is well defined and bounded for suitably small radii at any point $p \in \Omega$ such that $N(p) \neq Q \llbracket 0 \rrbracket$. Indeed, if such assumption is satisfied then 
\[
\lim_{r \to 0^{+}} \frac{1}{\Ha^{m-1}(\partial \B_{r}(p))} {\bf H}(r) = |N|^{2}(p) = \G(N(p), Q \llbracket 0 \rrbracket)^{2} > 0,
\]
which in turn implies that ${\bf H}(r) > 0$ for small values of $r$. Furthermore, from the proof of Theorem \ref{reg_Holder} (cf. in particular formula \eqref{Dir_decay}) we can also infer that if $r$ is sufficiently small then
\[
{\bf D}(r) \leq C r^{m-2+2\alpha},
\]
where $\alpha$ is the H\"older exponent of $N$. In particular, from this one immediately concludes that there exists the limit
\[
\lim_{r \to 0} {\bf I}(r) = 0
\]
at every point $p$ such that $N(p) \neq Q \llbracket 0 \rrbracket$. 
\\
As we shall see, we will obtain as a byproduct of the improved regularity theory developed in this section that the frequency function is well defined and bounded also in a suitable neighborhood of $r = 0$ at every point $p$ such that $N(p) = Q\llbracket 0 \rrbracket$, and that also at such points the limit $\lim_{r \to 0^{+}} {\bf I}(r)$ exists, but it is strictly positive. 
\end{remark}

The main analytic feature of the frequency function is the following \emph{almost monotonicity} property.

\begin{theorem}[Almost monotonicity of the frequency] \label{freq:am}
There exist a geometric constant $C_{0}$ and a radius $0 < r_{0} < \min\lbrace {\rm inj}(\Sigma), \ddist(p, \partial \Omega)\rbrace$ such that for all $0 < s < t \leq r_{0}$ with ${\bf H}\big|_{\left[s,t\right]} > 0$ one has
\begin{equation} \label{am:eq}
{\bf I}(s) \leq C_{0} \left( 1 + {\bf I}(t) \right).
\end{equation}

\end{theorem}

Propositions \ref{dicotomia} and \ref{am_impr_thm} below contain the most relevant consequences of Theorem \ref{freq:am}. Both these results will be derived under the additional assumption that $p \in \Omega$ has been fixed in such a way that $N(p) = Q \llbracket 0 \rrbracket$. As already observed in Remark \ref{well_posedness} above, these are exactly the points where we lack a precise description of the behavior of the frequency function. The arguments contained in the next sections will illustrate the reason why an analysis of the Jacobi multi-field $N$ in a neighborhood of such a point is indeed crucial in order to obtain the proof of Theorem \ref{sing:thm}.

The first result we are interested in is the following dichotomy: if $N(p) = Q \llbracket 0 \rrbracket$, then either there exists a neighborhood of $p$ where the map $N$ is identically vanishing, and thus where the frequency function is not defined at all, or, conversely, there is a neighborhood of $p$ where the frequency function is well defined everywhere and bounded.

\begin{proposition} \label{dicotomia}
Let $N \in \Gamma_{Q}^{1,2}(\N\Omega)$ be $\Jac$-minimizing. Assume $p \in \Omega$ is such that $N(p) = Q\llbracket 0 \rrbracket$. Then, the following dichotomy holds:
\begin{itemize}
\item[$(i)$] either $N \equiv Q\llbracket 0 \rrbracket$ in a neighborhood of $p$;
\item[$(ii)$] or there exists a radius $r_{0} > 0$ such that 
\[
{\bf H}(r) > 0 \mbox{ for all } r \in \left( 0, r_{0} \right] \hspace{0.5cm} \mbox{ and } \hspace{0.5cm} \limsup_{r \to 0} {\bf I}(r) < \infty.
\]
\end{itemize}
\end{proposition}

As it is natural, the most interesting situation is when condition $(ii)$ in the above Proposition \ref{dicotomia} is observed. As a first remark, we observe that the fact that the frequency function is bounded in a neighborhood of a point $p$ allows to improve the almost monotonicity property itself.

\begin{proposition}[Improved almost monotonicity of the frequency] \label{am_impr_thm}
Let $N \in \Gamma^{1,2}_{Q}(\N \Omega)$ be $\Jac$-minimizing. Assume $p \in \Omega$ is such that $N(p) = Q \llbracket 0 \rrbracket$ but $N$ does not vanish in a neighborhood of $p$. Then, there exist $r_{0} > 0$ and a constant $\lambda = \lambda({\bf I}(r_{0})) > 0$ such that the function
\begin{equation} \label{am_impr_eq:1}
r \in \left( 0, r_{0} \right] \mapsto e^{\lambda r} {\bf I}(r)
\end{equation}
is monotone non-decreasing. The limit
\begin{equation} \label{am_impr_eq:2}
\lim_{r \to 0} {\bf I}(r) =: I_0(p)
\end{equation}
exists and is strictly positive.
\end{proposition}

The rest of the section will be devoted to the proofs of Theorem \ref{freq:am}, Proposition \ref{dicotomia} and Proposition \ref{am_impr_thm}. 

\subsection{First variation estimates and the proof of Theorem \ref{freq:am}}

The proof of Theorem \ref{freq:am} is a consequence of some estimates involving the functions ${\bf D}$ and ${\bf H}$ and their derivatives, which in turn can be obtained by testing the first variations formulae \eqref{OV:eq} and \eqref{IV:eq} with a suitable choice of the maps $\psi$ and $X$. The derivation of these estimates is the content of Lemma \ref{FVE} below. We need to define the following auxiliary functions.

\begin{definition} \label{aux}
We denote by $\frac{\partial}{\partial \hat{r}}$ the vector field which is tangent to geodesic arcs parametrized by arc length and emanating from $p$. We will set $\nabla_{\hat{r}} := \nabla_{\frac{\partial}{\partial \hat{r}}}$, the directional derivative along $\frac{\partial}{\partial \hat{r}}$, and we will let $\nabla^{\perp}_{\hat{r}}$ be its projection onto the normal bundle of $\Sigma$ in $\M$. We set:
\begin{equation}
{\bf E}(r) = {\bf E}_{N,p}(r) := \int_{\partial\B_{r}(p)} \sum_{\ell=1}^{Q} \langle N^{\ell}(x), \nabla^{\perp}_{\hat{r}} N^{\ell}(x) \rangle \, \dHa^{m-1}(x),
\end{equation}
\begin{equation}
{\bf G}(r) = {\bf G}_{N,p}(r) := \int_{\partial \B_{r}(p)} |\nabla^{\perp}_{\hat{r}}N|^{2}(x) \, \dHa^{m-1}(x),
\end{equation}
and
\begin{equation}
{\bf F}(r) = {\bf F}_{N,p}(r) := \int_{\B_{r}(p)} |N|^{2}(x) \, \dHa^{m}(x).
\end{equation}
\end{definition}

\begin{remark} 
Note that ${\bf F}(r) = \| N \|_{L^{2}(\B_{r}(p))}^{2}$ is an absolutely continuous function, and for a.e. $r$
\begin{equation} \label{F'_eq}
{\bf F}'(r) = \int_{\partial \B_{r}(p)} |N|^{2} \, \dHa^{m-1} = {\bf H}(r).
\end{equation}

\end{remark}

\begin{lemma}[First variation estimates] \label{FVE}
There exist a geometric constant $C_{0} > 0$ and a radius $0 < r_{0} < \min\lbrace {\rm inj}(\Sigma), \ddist(p, \partial\Omega)\rbrace$ such that the following inequalities hold true for a.e. $0 < r \leq r_{0}$:

\begin{equation} \label{OV:est}
|{\bf D}(r) - {\bf E}(r)| \leq C_{0} {\bf F}(r),
\end{equation} 

\begin{equation} \label{IV:est}
|{\bf D}'(r) - 2 {\bf G}(r) - \frac{m-2}{r} {\bf D}(r)| \leq C_{0} r {\bf D}(r) + C_{0}({\bf D}(r){\bf F}(r))^{\sfrac{1}{2}},
\end{equation}

\begin{equation} \label{H'_est}
|{\bf H}'(r) - \frac{m-1}{r}{\bf H}(r) - 2{\bf E}(r)| \leq C_{0} r {\bf H}(r).
\end{equation}
Furthermore, if ${\bf I}(r) \geq 1$ then
\begin{equation} \label{OV:est_fin}
|{\bf D}(r) - {\bf E}(r)| \leq C_{0} r^{2} {\bf D}(r),
\end{equation}
and
\begin{equation} \label{IV:est_fin}
|{\bf D}'(r) - 2 {\bf G}(r) - \frac{m-2}{r} {\bf D}(r)| \leq C_{0} r {\bf D}(r).
\end{equation}

\end{lemma}

\begin{proof}

\textit{Step 1: proof of \eqref{OV:est}.} We test the outer variation formula \eqref{OV:eq} with the map $\psi$ given by  
\begin{equation} \label{test_OV}
\psi(x,u) := \phi\left( \frac{\mathrm{d}(x)}{r} \right) u,
\end{equation}
where $\mathrm{d}(\cdot) := {\bf d}(\cdot, p)$, and $\phi = \phi(t) \in C^{\infty}(\left[ 0, \infty \right))$ is a cut-off function such that:
\begin{equation} \label{cut_off:properties}
0 \leq \phi \leq 1, \quad \phi \equiv 1 \mbox{ in a neighborhood of } t = 0, \quad \phi \equiv 0 \mbox{ for } t \geq 1.
\end{equation}
Observe first that this choice of $\psi$ induces an admissible family of outer variations: indeed, one clearly sees that $\spt(\psi) \subset {\bf B}_{r}(p)$, the geodesic ball centered at $p$ and of radius $r$, which is compactly supported in $\Omega$, and that the orthogonality conditions and the assumptions in \eqref{OV:cond} are satisfied. We compute:
\[
D_{x}\psi(x,u) = r^{-1} \phi'\left( \frac{\mathrm{d}(x)}{r} \right) u \otimes \nabla \mathrm{d},
\]
which yields
\begin{equation} \label{OV:est:1}
\langle \nabla^{\perp}N^{\ell}(x) \, \colon \, \nabla^{\perp}\psi(x, N^{\ell}(x)) \rangle = r^{-1} \phi'\left( \frac{\mathrm{d}(x)}{r} \right) \langle \nabla^{\perp}_{\hat{r}}N^{\ell}(x), N^{\ell}(x) \rangle.
\end{equation}
On the other hand, $D_{u}\psi(x,u) = \phi\left( \frac{\mathrm{d}(x)}{r} \right) {\rm Id}$, whence
\begin{equation} \label{OV:est:2}
\langle \nabla^{\perp}N^{\ell}(x) \, \colon \, D_{u}\psi(x,N^{\ell}(x)) \cdot DN^{\ell}(x) \rangle = \phi\left( \frac{\mathrm{d}(x)}{r} \right) |\nabla^{\perp}N^{\ell}|^{2}(x).
\end{equation}

Analogously, we can compute explicitly the right-hand side of \eqref{OV:eq} corresponding to our choice of $\psi$ and get:
\begin{equation} \label{OV:est:3}
\mathcal{E}_{{\rm OV}}(\psi) = \int_{\Sigma} \phi\left( \frac{\mathrm{d}(x)}{r} \right) \sum_{\ell=1}^{Q} \left( | A \cdot N^{\ell} |^{2}(x) + \Ri(N^{\ell}(x), N^{\ell}(x)) \right) \, \dHa^{m}(x).
\end{equation}

By a standard approximation procedure, the details of which are left to the reader, it is easy to see that we can test with 
\begin{equation} \label{OV:est:test}
\phi(t) = \phi_{h}(t) = 
\begin{cases}
1 & \mbox{ for } 0 \leq t \leq 1 - \frac{1}{h} \\
h(1 - t) & \mbox{ for } 1 - \frac{1}{h} \leq t \leq 1 \\
0 & \mbox{ for } t \geq 1.
\end{cases}
\end{equation}

Inserting into \eqref{OV:est:1}, \eqref{OV:est:2} and \eqref{OV:est:3}, the outer variation formula \eqref{OV:eq} becomes
\begin{equation} \label{OV:est:4}
\begin{split}
- \frac{h}{r} \int_{\B_{r}(p) \setminus \B_{r - \frac{r}{h}}(p)} &\sum_{\ell=1}^{Q} \langle \nabla^{\perp}_{\hat{r}} N^{\ell}(x), N^{\ell}(x) \rangle \, \dHa^{m}(x) + \int_{\Sigma} \phi_{h}\left( \frac{\mathrm{d}(x)}{r} \right) |\nabla^{\perp}N(x)|^{2} \, \dHa^{m}(x) \\
& = \int_{\Sigma} \phi_{h}\left( \frac{\mathrm{d}(x)}{r} \right) \sum_{\ell=1}^{Q} \left( | A \cdot N^{\ell} |^{2}(x) + \Ri(N^{\ell}(x), N^{\ell}(x)) \right) \, \dHa^{m}(x).
\end{split}
\end{equation}
Now, let $h \uparrow \infty$. The left-hand side of \eqref{OV:est:4} converges to ${\bf D}(r) - {\bf E}(r)$, whereas the right-hand side converges to 
\[
\int_{\B_{r}(p)} \sum_{\ell=1}^{Q} \left( | A \cdot N^{\ell} |^{2} + \Ri(N^{\ell}, N^{\ell}) \right) \, \dHa^{m}.
\]
In particular, the inequality \eqref{OV:est} readily follows with a constant $C_0$ depending on ${\bf A} = \|A\|_{L^{\infty}}$ and ${\bf R} = \| R \|_{L^{\infty}}$.\\

\textit{Step 2: proof of \eqref{IV:est}.} We test now the inner variation formula \eqref{IV:eq} with the vector field $X$ defined by
\begin{equation} \label{test_IV}
\begin{split}
X(x) &:= \frac{\mathrm{d}(x)}{r} \phi\left( \frac{\mathrm{d}(x)}{r} \right) \frac{\partial}{\partial \hat{r}} \\
&= \phi\left( \frac{\mathrm{d}(x)}{r} \right) \frac{1}{2r} \nabla(\mathrm{d}(x)^{2}),
\end{split}
\end{equation}
with $\phi$ as in \eqref{cut_off:properties}.

Standard computations lead to
\[
\begin{split}
\nabla^{\Sigma}X(x) &= \phi'\left( \frac{\mathrm{d}(x)}{r} \right) \frac{\mathrm{d}(x)}{r^2} \frac{\partial}{\partial \hat{r}} \otimes \frac{\partial}{\partial \hat{r}} + \phi\left( \frac{\mathrm{d}(x)}{r} \right) \frac{1}{2r} {\rm Hess}^{\Sigma}(\mathrm{d}(x)^2) \\
& = \phi'\left( \frac{\mathrm{d}(x)}{r} \right) \frac{\mathrm{d}(x)}{r^2} \frac{\partial}{\partial \hat{r}} \otimes \frac{\partial}{\partial \hat{r}} + \phi\left( \frac{\mathrm{d}(x)}{r} \right) \left( \frac{{\rm Id}}{r} + O(r) \right)
\end{split}
\]
for $r \to 0$, and consequently
\[
\begin{split}
\di_{\Sigma}X(x) &= \phi'\left( \frac{\mathrm{d}(x)}{r} \right) \frac{\mathrm{d}(x)}{r^2} + \phi\left( \frac{\mathrm{d}(x)}{r} \right) \frac{1}{2r} \Delta_{\Sigma}(\mathrm{d}(x)^2) \\
&= \phi'\left( \frac{\mathrm{d}(x)}{r} \right) \frac{\mathrm{d}(x)}{r^2} + \phi\left( \frac{\mathrm{d}(x)}{r} \right) \left( \frac{m}{r} + O(r) \right).
\end{split}
\]

Choosing again tests of the form $\phi = \phi_{h}$ as in \eqref{OV:est:test}, plugging into \eqref{IV:eq} and taking the limit $h \uparrow \infty$, we see that the left-hand side of the inner variation formula reads
\begin{equation} \label{IV:est1}
{\bf D}'(r) - 2 {\bf G}(r) - \frac{m-2}{r} {\bf D}(r) + O(r) {\bf D}(r)
\end{equation}
for $r \to 0$.

We proceed with the analysis of the error term $\mathcal{E}_{{\rm IV}}(X)$. Straightforward computations imply the following estimates:
\[
|\mathcal{E}_{{\rm IV}}^{(1)}| \leq C_{1} \int_{\B_{r}(p)} |N(x)| |\nabla^{\perp}N(x)| \, \dHa^m(x),
\]

\[
|\mathcal{E}_{{\rm IV}}^{(2)}| + |\mathcal{E}_{{\rm IV}}^{(3)}| \leq C_{2,3} \int_{\B_{r}(p)} |N(x)| |\nabla^{\perp}_{\hat{r}}N(x)| \, \dHa^{m}(x),
\]
where $C_{1}$ is a geometric constant depending on ${\bf{\overline{A}}} = \| \overline{A} \|_{L^{\infty}}$, and $C_{2,3}$ depends on ${\bf A}$ and ${\bf R}$. Applying the Cauchy-Schwarz inequality we conclude
\begin{equation} \label{IV:est2}
|\mathcal{E}_{{\rm IV}}(X)| \leq C_{0} \left( {\bf D}(r) {\bf F}(r) \right)^{\sfrac{1}{2}}.
\end{equation} 

Combining \eqref{IV:est1} and \eqref{IV:est2}, we deduce the inequality \eqref{IV:est} whenever $r$ is small enough. \\

\textit{Step 3: proof of \eqref{H'_est}.} Let $\exp_{p} \colon \mathcal{V} \subset T_{p}\Sigma \to \Sigma$ be the exponential map with pole $p$. Since $B_{r}(0) \Subset \mathcal{V}$ for every $r < {\rm inj}(\Sigma)$, we can use the change of coordinates $x = \exp_{p}(y)$ to write:
\[
\begin{split}
{\bf H}(r) &= \int_{\partial B_{r}} |N|^{2}(\exp_{p}(y)) \, {\bf J}\exp_{p}(y) \, \dHa^{m-1}(y) \\
&= r^{m-1} \int_{\partial B_{1}} |N|^{2}(\exp_{p}(rz)) \, {\bf J}\exp_{p}(rz) \, \dHa^{m-1}(z).
\end{split}
\]
Thus, we differentiate under the integral sign and compute
\[
\begin{split}
{\bf H}'(r) &= (m-1)r^{m-2} \int_{\partial B_{1}} |N|^{2}(\exp_{p}(rz)) \, {\bf J}\exp_{p}(rz) \, \dHa^{m-1}(z) \\
&+2r^{m-1} \int_{\partial B_{1}} \sum_{\ell=1}^{Q} \langle N^{\ell}(\exp_{p}(rz)), \nabla^{\perp}_{\hat{r}}N^{\ell}(\exp_{p}(rz)) \rangle \, {\bf J}\exp_{p}(rz) \, \dHa^{m-1}(z) \\
&+r^{m-1}\int_{\partial B_{1}} |N|^{2}(\exp_{p}(rz)) \frac{d}{dr}\left({\bf J}\exp_{p}(rz)\right) \, \dHa^{m-1}(z). 
\end{split}
\]
Since $\frac{d}{dr}\left( {\bf J}\exp_{p}(rz)\right) = O(r)$ for $r \to 0$, we are able to conclude
\begin{equation} 
{\bf H}'(r) = \frac{m-1}{r} {\bf H}(r) + 2 {\bf E}(r) + O(r) {\bf H}(r),
\end{equation}
from which \eqref{H'_est} readily follows. \\

\textit{Step 4: proof of \eqref{OV:est_fin} and \eqref{IV:est_fin}.} It suffices to exploit the inequality
\begin{equation} \label{freq1regime:poinc}
{\bf F}(r) \leq C_{0} r {\bf H}(r) + C_{0} r^{2} {\bf D}(r),
\end{equation}
which can be easily deduced from the Poincar\'e inequality (note also that the same inequality has been already proved in the Euclidean setting earlier on, cf. Corollary \ref{Q-Poincare}). In the regime ${\bf I}(r) \geq 1$, that is ${\bf H}(r) \leq r {\bf D}(r)$, \eqref{freq1regime:poinc} simply reads
\begin{equation} \label{freq1regime:poinc2}
{\bf F}(r) \leq C_{0} r^{2} {\bf D}(r).
\end{equation}
Then, \eqref{OV:est_fin} and \eqref{IV:est_fin} are an immediate consequence of \eqref{OV:est} and \eqref{IV:est} respectively.
\end{proof}

We can now proceed with the proof of the almost monotonicity property of the frequency.

\begin{proof}[Proof of Theorem \ref{freq:am}]
Set $\boldsymbol{\Omega}(r) := \log(\max\{{\bf I}(r),1\})$. In order to prove the theorem, it suffices to show that
\begin{equation} \label{am:eq2}
\boldsymbol{\Omega}(s) \leq C + \boldsymbol{\Omega}(t)
\end{equation}
for some positive geometric constant $C$. If $\boldsymbol{\Omega}(s) = 0$ there is nothing to prove. Thus, we assume that $\boldsymbol{\Omega}(s) > 0$. Define
\[
\tau := \sup\{r \in \left( s, t \right] \, \colon \, \boldsymbol{\Omega}(r) > 0 \mbox{\ on } \left(s, r \right) \}.
\]
If $\tau < t$, then by continuity it must be $\boldsymbol{\Omega}(\tau) = 0$: hence, in this case we would have $\boldsymbol{\Omega}(\tau) = 0 \leq \boldsymbol{\Omega}(t)$, and therefore proving that $\boldsymbol{\Omega}(s) \leq C + \boldsymbol{\Omega}(\tau)$ would imply \eqref{am:eq2}. Thus, we can assume without loss of generality that $\boldsymbol{\Omega}(r) > 0$ in $\left( s, t \right)$: in other words, ${\bf I}(r) > 1$, and $\boldsymbol{\Omega}(r) = \log({\bf I}(r))$. Then, as a consequence of \eqref{OV:est_fin}, if $r_{0}$ is taken small enough one has

\begin{equation} \label{D_E_comp}
\frac{{\bf D}(r)}{2} \leq {\bf E}(r) \leq 2 {\bf D}(r),
\end{equation}
that is the quantity ${\bf E}(r)$ is positive and comparable to ${\bf D}(r)$ at small scales. 

Guided by this principle, we compute:
\begin{equation} \label{am0}
\begin{split}
- \frac{d}{dr}(\log {\bf I}(r)) &= \frac{{\bf H}'(r)}{{\bf H}(r)} - \frac{{\bf D}'(r)}{{\bf D}(r)} - \frac{1}{r} \\
&= \frac{{\bf H}'(r)}{{\bf H}(r)} -  \frac{{\bf D}'(r)}{{\bf E}(r)} - {\bf D'}(r) {\bf Z}(r) - \frac{1}{r},
\end{split}
\end{equation}
where ${\bf Z}(r) := \displaystyle\frac{1}{{\bf D}(r)} - \displaystyle\frac{1}{{\bf E}(r)}$ satisfies 

\begin{equation} \label{err_est}
|{\bf Z}(r)| = \frac{|{\bf D}(r) - {\bf E}(r)|}{{\bf D}(r){\bf E}(r)} \overset{\eqref{D_E_comp}}{\le} 2 \frac{|{\bf D}(r) - {\bf E}(r)|}{{\bf D}(r)^2} \overset{\eqref{OV:est}}{\le} C_{0} \frac{{\bf F}(r)}{{\bf D}(r)^{2}}.
\end{equation}

Now, by \eqref{H'_est} one has that
\begin{equation} \label{am1}
\frac{{\bf H}'(r)}{{\bf H}(r)} \leq Cr + \frac{m-1}{r} + 2 \frac{{\bf E}(r)}{{\bf H}(r)},
\end{equation}
whereas the inner variation formula \eqref{IV:est_fin} yields
\begin{equation} \label{am2}
\begin{split}
- \frac{{\bf D}'(r)}{{\bf E}(r)} &\leq  C r \frac{{\bf D}(r)}{{\bf E}(r)} - 2 \frac{{\bf G}(r)}{{\bf E}(r)} - \frac{m-2}{r} \frac{{\bf D}(r)}{{\bf E}(r)}\\
&\overset{\eqref{D_E_comp}}{\leq}  Cr - 2 \frac{{\bf G}(r)}{{\bf E}(r)} - \frac{m-2}{r} \left( 1 - {\bf D}(r) {\bf Z}(r) \right) \\
&\overset{\eqref{err_est}}{\leq}  Cr - 2 \frac{{\bf G}(r)}{{\bf E}(r)} - \frac{m-2}{r} + C r^{-1} \frac{{\bf F}(r)}{{\bf D}(r)} \\
&\leq Cr - 2 \frac{{\bf G}(r)}{{\bf E}(r)} - \frac{m-2}{r}
\end{split}
\end{equation}
because of \eqref{freq1regime:poinc2}.

Plugging \eqref{am1} and \eqref{am2} into \eqref{am0}, and using the estimate on the error term ${\bf Z}(r)$ given by \eqref{err_est}, we obtain the following:

\begin{equation} \label{am:key}
- \frac{d}{dr} (\log{\bf I}(r)) \leq Cr + 2 \left( \frac{{\bf E}(r)}{{\bf H}(r)} - \frac{{\bf G}(r)}{{\bf E}(r)} \right) + C \frac{{\bf D}'(r)}{{\bf D}(r)^{2}} {\bf F}(r).
\end{equation}

Now, by the Cauchy-Schwarz inequality one has
\[
{\bf E}(r)^{2} \leq {\bf G}(r){\bf H}(r),
\]
whence the term $\displaystyle \frac{{\bf E}(r)}{{\bf H}(r)} - \displaystyle \frac{{\bf G}(r)}{{\bf E}(r)}$ is non-positive and \eqref{am:key} yields
\begin{equation} \label{am:key2}
- \frac{d}{dr} (\log{\bf I}(r)) \leq Cr + C \frac{{\bf D}'(r)}{{\bf D}(r)^{2}} {\bf F}(r).
\end{equation}

Integrating for $r \in \left( s,t \right)$, we obtain
\begin{equation} \label{am:int}
\boldsymbol{\Omega}(s) - \boldsymbol{\Omega}(t) \leq C + C\left( \frac{{\bf F}(s)}{{\bf D}(s)} - \frac{{\bf F}(t)}{{\bf D}(t)} \right) + C \int_{s}^{t} \frac{{\bf F}'(r)}{{\bf D}(r)} \, dr \leq C,
\end{equation}
where the last inequality follows from the above observation that, in the regime ${\bf I} \geq 1$, the inequalities 
\[
{\bf F}(r) \leq C_{0} r^{2} {\bf D}(r), \hspace{0.5cm} {\bf F}'(r) = {\bf H}(r) \leq r {\bf D}(r)
\]
hold almost everywhere. This completes the proof.
\end{proof}

\subsection{Proof of Propositions \ref{dicotomia} and \ref{am_impr_thm}}

We will need the following version of the Poincar\'e inequality.

\begin{lemma} \label{Poinc_type:thm}
There exist a radius $0 < r_{0} = r_{0}(m,Q) < {\rm inj}(\Sigma)$ and a geometric constant $C > 0$ with the following property. Let $N \in \Gamma_{Q}^{1,2}(\N\Omega)$ be a multiple valued section of $\N\Sigma$ $\Jac$-minimizing in $\Omega$. Assume $p \in \Omega$ is such that $N(p) = Q \llbracket 0 \rrbracket$. Then, the inequality
\begin{equation} \label{Poinc_type:eq}
\| N \|_{L^2(\B_{r}(p))}^{2} \leq C r^2 \Dir^{\N\Sigma}(N,\B_{r}(p))
\end{equation}
holds true for every $0 < r \leq \min\lbrace r_{0}, \ddist(p, \partial\Omega) \rbrace$.
\end{lemma}

\begin{proof}
Let $r_{0} = r_{0}(m,Q)$ be the radius given by Theorem \ref{reg_Holder}, and let $r \leq \min\lbrace r_{0}, \ddist(p, \partial \Omega) \rbrace$ be arbitrary. Let $\rho \in \left( 0, \frac{r}{2} \right]$ be a radius to be chosen later and split $\| N \|_{L^{2}(\B_{r}(p))}^{2}$ into the sum
\begin{equation} \label{P_t:split}
\int_{\B_{r}(p)} |N|^{2} \, \dHa^{m} = \int_{\B_{\rho}(p)} |N|^{2} \, \dHa^{m} + \int_{\B_{r}(p) \setminus \B_{\rho}(p)} |N|^{2} \, \dHa^{m}.
\end{equation}
In order to estimate the first term in the sum, we recall that $|N|^{2}(x) = \G(N(x), Q\llbracket 0 \rrbracket)^{2} = \G(N(x), N(p))^{2}$ and exploit the $\alpha$-H\"older continuity of $N$ to conclude
\begin{equation} \label{P_t:1st}
\begin{split}
\int_{\B_{\rho}(p)} |N|^{2} \, \dHa^{m} &\leq \rho^{2\alpha} \left[ N \right]_{C^{0,\alpha}(\overline{\B}_{\rho}(p))}^{2} \Ha^{m}(\B_{\rho}(p)) \\
&\overset{\eqref{reg_Holder:est}}{\leq} C \rho^{2} \left( \Dir(N, \B_{2\rho}(p)) + \Lambda \int_{\B_{2\rho}(p)} |N|^{2} \, \dHa^{m} \right) \\
&\leq C \rho^{2} \Dir(N, \B_{r}(p)) + C \Lambda \rho^{2} \int_{\B_{r}(p)} |N|^{2} \, \dHa^{m} \\
&\leq \underbrace{C r^{2} \Dir^{\N\Sigma}(N, \B_{r}(p))}_{=: I_{1}} + \underbrace{C(\Lambda + C_{0}) \rho^{2} \int_{\B_{r}(p)} |N|^{2} \, \dHa^{m}}_{=: I_{2}},
\end{split}
\end{equation}
where $C_{0}$ depends on $\mathbf{A}$ and $\overline{\mathbf{A}}$.

As for the second addendum in \eqref{P_t:split}, we integrate in normal polar coordinates with pole $p$ to write
\begin{equation} \label{P_t:2nd}
\int_{\B_{r}(p) \setminus \B_{\rho}(p)} |N|^{2} \, \dHa^{m} = \int_{\rho}^{r} \left( \int_{\partial \B_{\tau}(p)} |N|^{2} \, \dHa^{m-1} \right) \, \mathrm{d}\tau.
\end{equation}
Now, fix any $\tau \in \left( \rho, r \right)$, and for every $x \in \partial \B_{\tau}(p)$ let $\gamma_{x} = \gamma_{x}(s)$, $s \in \left[ 0, \tau \right]$, be the unique geodesic parametrized by arclength joining $p$ to $x$. Also denote by $\overline{x}$ the point where $\gamma_{x}$ intersects $\partial \B_{\rho}(p)$. Then, the fundamental theorem of calculus immediately yields
\begin{equation} \label{P_t:arc}
|N|^{2}(x) \leq |N|^{2}(\overline{x}) + 2 \int_{\rho}^{\tau} (|N| |\nabla^{\perp}N|)(\gamma_{x}(s)) \, ds.
\end{equation}
Integrate the above inequality in $x \in \partial \B_{\tau}(p)$ to get
\begin{equation} \label{P_t:sphere}
\int_{\partial \B_{\tau}(p)} |N|^{2} \, \dHa^{m-1} \leq C \left( \frac{\tau}{\rho} \right)^{m-1} \left( \int_{\partial\B_{\rho}(p)} |N|^{2} \, \dHa^{m-1} + 2 \int_{\B_{\tau}(p) \setminus \B_{\rho}(p)} |N| |\nabla^{\perp}N| \, \dHa^{m} \right).
\end{equation}
Using once again the H\"older estimate \eqref{reg_Holder:est} and recalling that $\rho \leq \frac{r}{2}$, we are able to control
\begin{equation} \label{1term}
\begin{split}
\left( \frac{\tau}{\rho} \right)^{m-1} \int_{\partial \B_{\rho}(p)} |N|^{2} \, \dHa^{m-1} &\leq \left( \frac{\tau}{\rho} \right)^{m-1} \rho^{2\alpha} \left[ N \right]_{C^{0,\alpha}(\overline{\B}_{\frac{r}{2}}(p))}^{2} \Ha^{m-1}(\partial \B_{\rho}(p)) \\
&\leq C \tau^{m-1} \rho^{2\alpha} r^{2-m-2\alpha} \left( \Dir(N, \B_{r}(p)) + \Lambda \int_{\B_{r}(p)} |N|^{2} \, \dHa^{m} \right).
\end{split}
\end{equation}
We can now integrate in $\tau \in \left( \rho, r \right)$, so that using the estimates in \eqref{P_t:sphere} and \eqref{1term} we can easily deduce from \eqref{P_t:2nd} the following inequality:
\begin{equation} \label{P_t:2nd_fin}
\begin{split}
\int_{\B_{r}(p) \setminus \B_{\rho}(p)} |N|^{2} \, \dHa^{m} &\leq \underbrace{C \left(\frac{\rho}{r}\right)^{2\alpha} r^{2} \Dir^{\N\Sigma}(N, \B_{r}(p))}_{=: J_{1}} \\
&+ \underbrace{C (\Lambda + C_{0}) \rho^{2\alpha} \int_{\B_{r}(p)} |N|^{2} \, \dHa^{m}}_{=: J_{2}} \\
&+ \underbrace{C \left( \frac{r}{\rho} \right)^{m-1} r \int_{\B_{r}(p)} |N| |\nabla^{\perp}N| \, \dHa^{m}}_{=: J_3},
\end{split}
\end{equation}
where $C$ and $C_{0}$ are geometric constants. Now we can sum up the contributions coming from the ball $\B_{\rho}(p)$ and from the annulus $\B_{r}(p) \setminus \B_{\rho}(p)$ and choose $\rho = \rho(\Lambda, C, C_{0})$ so small that the terms $I_{2}$ and $J_{2}$ are absorbed in the left-hand side of the equation, thus ultimately providing
\begin{equation}\label{P_t:fin}
\int_{\B_{r}(p)} |N|^{2} \, \dHa^{m} \leq C r^{2} \Dir^{\N\Sigma}(N, \B_{r}(p)) + C r \int_{\B_{r}(p)} |N| |\nabla^{\perp}N| \, \dHa^{m}.
\end{equation}

Finally, use Young's inequality: for any choice of the parameter $\eta > 0$, \eqref{P_t:fin} implies that
\begin{equation} \label{P_t:Young}
\int_{\B_{r}(p)} |N|^{2} \, \dHa^{m} \leq Cr^{2} \Dir^{\N\Sigma}(N, \B_{r}(p)) + C r \left( \eta \int_{\B_{r}(p)} |N|^{2} \, \dHa^{m} + \frac{1}{\eta} \Dir^{\N\Sigma}(N, \B_{r}(p)) \right).
\end{equation}
The conclusion immediately follows by choosing $\eta$ such that $Cr\eta = \frac{1}{2}$.
\end{proof}

\begin{proof}[Proof of Proposition \ref{dicotomia}]
First observe that if $N$ does not vanish identically in a neighborhood of $p$, then there exists $r_{0} > 0$ such that ${\bf H}(r_0) > 0$. Clearly, without loss of generality we can suppose that \eqref{Poinc_type:eq} holds for every $0 < r \leq r_{0}$, and also that \eqref{am:eq} holds in any interval $\left[ s, t \right] \subset \left( 0, r_{0} \right]$ such that ${\bf H}\big|_{\left[ s, t \right]} > 0$. We claim that in fact ${\bf H}(r) > 0$ for all $0 < r \leq r_0$. Indeed, if this is not true, let $\rho > 0$ be given by $\rho := \sup\lbrace r \in \left(0, r_0 \right] \, \colon \, {\bf H}(r) = 0 \rbrace$. By definition ${\bf H}(r) > 0$ for $\rho < r \leq r_0$, whence for such $r$'s we can take advantage of Theorem \ref{freq:am} and write
\[
{\bf I}(r) \leq C_{0}(1 + {\bf I}(r_0)).
\]
By letting $r \downarrow \rho$ we conclude
\[
\rho {\bf D}(\rho) \leq C_{0}(1 + {\bf I}(r_0)) {\bf H}(\rho) = 0,
\]
which in turn produces $\Dir^{\N\Sigma}(N, \B_{\rho}(p)) = 0$. Then, by Lemma \ref{Poinc_type:thm}, $N$ vanishes identically in $\B_{\rho}(p)$, contradiction.

It is now a simple consequence of Theorem \ref{freq:am} that
\[
\limsup_{r \to 0} {\bf I}(r) \leq C_{0}(1 + {\bf I}(r_0)),
\]
which completes the proof.
\end{proof}

\begin{proof}[Proof of Proposition \ref{am_impr_thm}]
Under the assumptions in the statement, case $(ii)$ in Proposition \ref{dicotomia} must hold, and thus the frequency function is well defined and bounded in an interval $\left( 0, r_0 \right]$. Moreover, the Poincaré inequality \eqref{Poinc_type:eq} implies that, modulo possibly taking a smaller value of $r_0$, the first variation estimates of Lemma \ref{FVE} can be again re-written as in \eqref{OV:est_fin} and \eqref{IV:est_fin}, and that \eqref{D_E_comp} holds. Thus, we can compute:
\begin{equation} \label{am_impr:1}
\begin{split}
{\bf I}'(r) &= \frac{{\bf D}(r)}{{\bf H}(r)} + \frac{r {\bf D}'(r)}{{\bf H}(r)} - \frac{r {\bf D}(r) {\bf H}'(r)}{{\bf H}(r)^{2}} \\
&= \frac{{\bf D}(r)}{{\bf H}(r)} + \frac{r}{{\bf H}(r)} \left( 2 {\bf G}(r) + \frac{m-2}{r} {\bf D}(r) + \mathcal{E}_{1}(r) \right) - \frac{r {\bf D}(r)}{{\bf H}(r)^{2}} \left( \frac{m-1}{r} {\bf H}(r) + 2 {\bf E}(r) + \mathcal{E}_{2}(r) \right) \\
&= \frac{2 r}{{\bf H}(r)^{2}} \left( {\bf G}(r) {\bf H}(r) - {\bf E}(r)^{2} \right) + \frac{r}{{\bf H}(r)} \, \mathcal{E}_{1}(r) - \frac{r {\bf D}(r)}{{\bf H}(r)^{2}} \, \mathcal{E}_{2}(r) + \mathcal{E}_{3}(r), 
\end{split}
\end{equation}
where
\begin{equation} \label{am_impr:2}
|\mathcal{E}_{1}(r)| \overset{\eqref{IV:est_fin}}{\leq} C_{0} r {\bf D}(r),
\end{equation}
\begin{equation} \label{am_impr:3}
|\mathcal{E}_{2}(r)| \overset{\eqref{H'_est}}{\leq} C_{0} r {\bf H}(r),
\end{equation}
and
\begin{equation} \label{am_impr:4}
|\mathcal{E}_{3}(r)| = \frac{2 r {\bf E}(r)}{{\bf H}(r)^{2}} |{\bf E}(r) - {\bf D}(r)| \overset{\eqref{OV:est_fin}, \eqref{D_E_comp}}{\leq} C_{0} \frac{r^{3} {\bf D}(r)^{2}}{{\bf H}(r)^{2}},
\end{equation}
if $r_0$ is chosen small enough. Since ${\bf G}(r) {\bf H}(r) - {\bf E}(r)^{2} \geq 0$ by the Cauchy-Schwartz inequality, the above arguments show the existence of a radius $r_{0} > 0$ and a geometric constant $C_0 > 0$ such that 
\begin{equation} \label{am_impr:5}
{\bf I}'(r) \geq - C_{0} r {\bf I}(r) - C_{0} r {\bf I}(r)^{2} 
\end{equation}
for all $r \in \left( 0, r_0 \right]$. On the other hand, for such $r$'s one has ${\bf I}(r) \leq C_{0}(1 + {\bf I}(r_{0}))$ by Theorem \ref{freq:am}. Thus, this allows to conclude that
\begin{equation} \label{am_impr:6}
{\bf I}'(r) \geq - \lambda {\bf I}(r),
\end{equation}
for some positive $\lambda$ depending only on $r_0$ and ${\bf I}(r_{0})$. The monotonicity of the function $r \mapsto e^{\lambda r} {\bf I}(r)$ is now a simple consequence of \eqref{am_impr:6}.

Next, we conclude the proof showing that the limit
\begin{equation} \label{am_impr:6bis}
I_0 := \lim_{r \to 0} e^{\lambda r} {\bf I}(r) = \lim_{r \to 0} {\bf I}(r)
\end{equation}
is positive. To see this, we show that the Poincaré inequality \eqref{Poinc_type:eq} allows to bound the frequency function from below with a positive constant. Indeed, arguing as in the proof of Lemma \ref{Poinc_type:thm} (cf. in particular the equations \eqref{P_t:sphere} and \eqref{1term}), it is easily seen that one can estimate
\begin{equation} \label{am_impr:7}
\begin{split}
{\bf H}(r) &= \int_{\partial \B_{r}(p)} |N|^{2} \, \dHa^{m-1} \leq C \left( \int_{\partial \B_{\frac{r}{2}}(p)} |N|^{2} \, \dHa^{m-1} + \int_{\B_{r}(p) \setminus \B_{\frac{r}{2}}(p)} |N| |\nabla^{\perp}N| \, \dHa^{m} \right)\\
&\leq Cr \left( \Dir(N, \B_{r}(p)) + \Lambda \int_{\B_{r}(p)} |N|^{2} \, \dHa^{m} \right) + C \int_{\B_{r}(p)} |N| |\nabla^{\perp}N| \, \dHa^{m}\\
&\leq Cr {\bf D}(r) + Cr {\bf F}(r) + C \int_{\B_{r}(p)} |N| |\nabla^{\perp}N| \, \dHa^{m}.
\end{split}
\end{equation} 

In turn, applying Young's inequality to the last addendum in the right-hand side of \eqref{am_impr:7} yields
\begin{equation} \label{am_impr:8}
\int_{\B_{r}(p)} |N| |\nabla^{\perp}N| \, \dHa^{m} \leq \frac{r}{2} {\bf D}(r) + \frac{1}{2r} {\bf F}(r).
\end{equation}

Plugging \eqref{am_impr:8} in \eqref{am_impr:7} and using the Poincaré inequality \eqref{Poinc_type:eq} finally gives
\begin{equation} \label{am_impr:9}
{\bf H}(r) \leq C (1 + r^{2}) r {\bf D}(r) \leq C (1 + r_{0}^{2}) r {\bf D}(r),
\end{equation}
thus completing the proof.
\end{proof}

\section{Reverse Poincar\'e and analysis of blow-ups for the top stratum} \label{sec:blow-up}

The final goal of this section is to perform the key step in the proof of Theorem \ref{sing:thm}, namely the blow-up procedure, see Theorem \ref{blow-up:thm} below. In doing this, we will clarify the importance of the results obtained in the previous paragraph.

\subsection{Reverse Poincar\'e inequalities}

The proof of the blow-up theorem will heavily rely on an important technical tool, a reverse Poincar\'e inequality for $\Jac$-minimizers. In Proposition \ref{rpis:thm}, we have already shown that $\Jac$-minimizers enjoy a Caccioppoli type inequality: the $L^{2}$-norm of a Jacobi $Q$-field $N$ in a ball $\B_{r}(p)$ controls the Dirichlet energy in the ball with half the radius. As an immediate consequence of the boundedness of the frequency function, one can actually show that the Dirichlet energy in $\B_{\frac{r}{2}}(p)$ can be controlled with the $L^{2}$-norm of $N$ in the annulus $\B_{r}(p) \setminus \B_{\frac{r}{2}}(p)$, provided that we allow the constant to depend on the value of the frequency at a given scale.

\begin{proposition} \label{Caccioppoli}
Let $N \in \Gamma_{Q}^{1,2}(\N\Omega)$ be $\Jac$-minimizing. Then, there exists $r_{0} > 0$ such that for any $r \in \left( 0, r_{0} \right]$ one has
\begin{equation} \label{Caccioppoli:eq}
\Dir^{\N\Sigma}(N, \B_{\frac{r}{2}}(p)) \leq \frac{C}{r^2} \int_{\B_{r}(p) \setminus \B_{\frac{r}{2}}(p)} |N|^{2} \, \dHa^{m}
\end{equation}
for some positive $C = C({\bf I}(r_{0}))$.
\end{proposition}

\begin{proof}
If $N$ is vanishing identically in a neighborhood of $p$ there is nothing to prove. Therefore, we can assume that either $N(p) \neq Q \llbracket 0 \rrbracket$ or $N(p) = Q \llbracket 0 \rrbracket$ but $N$ does not vanish identically in any neighborhood of $p$. In any of the two cases, either by the arguments contained in Remark \ref{well_posedness} or by Proposition \ref{dicotomia}, there exists a positive radius $r_{0}$ such that the frequency function is well defined and bounded for all $r \in \left( 0, r_{0} \right]$. Thus, there exists a positive constant $C = C({\bf I}(r_{0}))$ such that, for fixed $r \leq r_{0}$, $\tau {\bf D}(\tau) \leq C {\bf H}(\tau)$ for $\tau$ in the interval $\left[ \frac{r}{2}, r \right]$. Integrate with respect to $\tau$ to get \eqref{Caccioppoli:eq}:
\[
\begin{split}
\frac{3}{8} r^{2} \Dir^{\N\Sigma}(N, \B_{\frac{r}{2}}(p)) = \frac{3}{8} r^{2} {\bf D}\left( \frac{r}{2} \right) &\leq \int_{\frac{r}{2}}^{r} \tau {\bf D}(\tau) \, \mathrm{d}\tau \\ 
&\leq C \int_{\frac{r}{2}}^{r} {\bf H}(\tau) \, \mathrm{d}\tau = C \int_{\B_{r}(p) \setminus \B_{\frac{r}{2}}(p)} |N|^{2} \, \dHa^{m}.
\end{split}
\]
\end{proof}

The Caccioppoli inequality can in fact be improved further under the assumption that $N(p) = Q\llbracket 0 \rrbracket$: indeed, at small scales the inequality \eqref{rpis:eq2} holds without having to increase the support of the ball on the right-hand side. Again, for this to be true we need to allow the constant to depend on the value of the frequency at scale $r_{0}$.

\begin{proposition}[Reverse Poincaré Inequality] \label{reverse_Poinc:thm}
Let $N \in \Gamma_{Q}^{1,2}(\N\Omega)$ be $\Jac$-minimizing. Assume $N(p) = Q \llbracket 0 \rrbracket$. Then, there exists $r_{0} > 0$ such that for any $r \in \left( 0, r_{0} \right]$ the following inequality
\begin{equation} \label{reverse_Poinc:eq}
\Dir^{\N\Sigma}(N, \B_{r}(p)) \leq \frac{C}{r^2} \int_{\B_{r}(p)} |N|^{2} \, \dHa^m
\end{equation}
holds for some positive $C = C({\bf I}(r_{0}))$.
\end{proposition}

\begin{proof}
Once again, we observe that \eqref{reverse_Poinc:eq} is trivial when $N \equiv Q \llbracket 0 \rrbracket$ in a neighborhood of $p$. We assume then that case $(ii)$ in Proposition \ref{dicotomia} holds, and we let $r_{0}$ be the radius given in there. Since the frequency function is well defined and bounded in $\left( 0, r_{0} \right]$, there exists $C = C({\bf I}(r_{0})) > 0$ such that 
\begin{equation} \label{r_P:1}
\int_{\B_{r}(p)} |\nabla^{\perp}N|^{2} \, \dHa^m \leq \frac{C}{r} \int_{\partial \B_{r}(p)} |N|^{2} \, \dHa^{m-1},
\end{equation}
for all $r$'s in the above interval. Arguing once again as in the proof of Lemma \ref{Poinc_type:thm}, we have that for every $\rho \in \left( 0, \frac{r}{2} \right]$ it holds
\begin{equation} \label{r_P:2}
\int_{\partial \B_{r}(p)} |N|^{2} \, \dHa^{m-1} \leq C \left( \frac{r}{\rho} \right)^{m-1} \left( \int_{\partial \B_{\rho}(p)} |N|^{2} \, \dHa^{m-1} + 2 \int_{\B_{r}(p) \setminus \B_{\rho}(p)} |N| |\nabla^{\perp}N| \, \dHa^{m} \right).
\end{equation} 

Furthermore, by the H\"older continuity of $N$ and since $\rho \leq \frac{r}{2}$ we also have
\begin{equation} \label{r_P:3}
\left( \frac{r}{\rho} \right)^{m-1} \int_{\partial \B_{\rho}(p)} |N|^{2} \, \dHa^{m-1} \overset{\eqref{Poinc_type:eq}}{\leq} C \left( \frac{\rho}{r} \right)^{2\alpha} r \left( \int_{\B_{r}(p)} |\nabla^{\perp}N|^{2} \, \dHa^{m} + (C_{0} + \Lambda) \int_{\B_{r}(p)} |N|^{2}\, \dHa^{m} \right).
\end{equation}

Combining \eqref{r_P:1}, \eqref{r_P:2} and \eqref{r_P:3} gives
\begin{equation} \label{r_P:4}
\begin{split}
\int_{\B_{r}(p)} |\nabla^{\perp}N|^{2} \, \dHa^m &\leq C \left( \frac{\rho}{r} \right)^{2\alpha} \left( \int_{\B_{r}(p)} |\nabla^{\perp}N|^{2} \, \dHa^{m} + \int_{\B_{r}(p)} |N|^{2} \, \dHa^{m} \right) \\
&+ \frac{C}{r} \left( \frac{r}{\rho} \right)^{m-1} \int_{\B_{r}(p)} |N| |\nabla^{\perp}N| \, \dHa^{m}. 
\end{split}
\end{equation}

Now, if we choose $\rho$ so small that $C \left( \frac{\rho}{r} \right)^{2\alpha} \leq \frac{1}{2}$ then from \eqref{r_P:4} follows
\begin{equation} \label{r_P:5}
\begin{split}
\int_{\B_{r}(p)} |\nabla^{\perp}N|^{2} \, \dHa^{m} &\leq \int_{\B_{r}(p)} |N|^{2} \, \dHa^{m} + \frac{C}{r} \int_{\B_{r}(p)} |N| |\nabla^{\perp}N| \, \dHa^{m} \\
&\leq \int_{\B_{r}(p)} |N|^{2} \, \dHa^{m} + \frac{C}{2r} \eta \int_{\B_{r}(p)} |\nabla^{\perp}N|^{2} \, \dHa^{m} + \frac{C}{2r\eta} \int_{\B_{r}(p)} |N|^{2} \, \dHa^{m}, 
\end{split}
\end{equation}
by the Young's inequality. Choose $\eta = \frac{r}{C}$ to obtain
\begin{equation} \label{r_P:6}
\int_{\B_{r}(p)} |\nabla^{\perp}N|^{2} \, \dHa^{m} \leq \left( \frac{C}{r^{2}} + 2 \right) \int_{\B_{r}(p)} |N|^{2},
\end{equation}
which immediately implies \eqref{reverse_Poinc:eq}.
\end{proof}

Now that we have the Reverse Poincar\'e inequality at our disposal, we can enter the core of the blow-up scheme.

\subsection{The top-multiplicity singular stratum. Blow-up}

The main difficulty in the proof of Theorem \ref{sing:thm} consists of estimating the Hausdorff dimension of the set of singular points with multiplicity exactly equal to $Q$. The proof of the general result then follows in a relatively easy way by an induction argument on $Q$. Therefore, it is fundamental to study the structure of the \emph{top-multiplicity singular stratum} of $N$, denoted $\sing_{Q}(N)$ and defined as follows.

\begin{definition}[Top-multiplicity points] \label{TMP}
Let $N \in \Gamma_{Q}^{1,2}(\N\Omega)$ be $\Jac$-minimizing. A point $p \in \Omega$ has multiplicity $Q$, or simply is a $Q$-point for $N$, and we will write $p \in {\rm D}_{Q}(N)$, if there exists $v \in T^{\perp}_{p}\Sigma$ such that $N(p) = Q \llbracket v \rrbracket$. We will define the top-multiplicity regular and singular strata of $N$ by 
\[
\reg_{Q}(N) := \reg(N) \cap {\rm D}_{Q}(N), \hspace{0.5cm} \sing_{Q}(N) := \sing(N) \cap {\rm D}_{Q}(N),
\]
respectively.
\end{definition} 

From this point onward, we will assume to have fixed a point $p \in {\rm D}_{Q}(N)$. The first step is to show that without loss of generality we can always assume that $N(p) = Q \llbracket 0 \rrbracket$. Recall the definition of the map $\bfeta$ given in \eqref{eta}.

\begin{lemma} \label{zero_average}
Let $N \in \Gamma_{Q}^{1,2}(\N\Omega)$ be $\Jac$-minimizing. Then:
\begin{itemize}
\item[$(i)$] the map $\bfeta \circ N \colon \Omega \to \R^d$ is a classical Jacobi field;
\item[$(ii)$] if $\zeta \colon \Omega \to \R^d$ is a classical Jacobi field, then the $Q$-valued map $u := \sum_{\ell} \llbracket N^{\ell} + \zeta \rrbracket$ is $\Jac$-minimizing.
\end{itemize}
\end{lemma}

\begin{proof}
Recall (cf. Remark \ref{rmk:1-valued} and the notation therein) that a normal vector field $\zeta \in \Gamma^{1,2}(\N\Omega) := \Gamma^{1,2}_{1}(\N\Sigma)$  is a Jacobi field if it is a weak solution of the linear elliptic PDE on the normal bundle $\N\Sigma$
\[
\left( -\Delta_{\Sigma}^{\perp} - \mathscr{A} - \mathscr{R} \right) \zeta = 0,
\]
that is, if the identity
\begin{equation} \label{Jac_PDE}
\int_{\Omega} \left( \langle \nabla^{\perp}\zeta \, \colon \, \nabla^{\perp}\phi \rangle - \langle A \cdot \zeta \, \colon \, A \cdot \phi \rangle - \Ri(\zeta,\phi) \right) \, \dHa^{m} = 0 
\end{equation} 
holds for all test functions $\phi \in C^{1}\left( \Omega , \R^d \right)$ with $\spt(\phi) \subset \Omega' \Subset \Omega$ and $\phi(x) \in T_{x}^{\perp}\Sigma \subset T_{x}\M$ for every $x \in \Omega$.

In order to prove $(i)$, first observe that the map $\bfeta$ preserves the fibers of the normal bundle, so that $\bfeta \circ N(x) \in T_{x}^{\perp}\Sigma$ for a.e. $x \in \Omega$ and thus $\bfeta \circ N \in \Gamma^{1,2}(\N\Omega) = \Gamma_{1}^{1,2}(\N\Omega)$. Now, fix any vector field $\phi$ as above. It is immediate to see that we can test the outer variation formula \eqref{OV:eq} with $\psi(x,u) := \phi(x)$, and that the resulting equation is precisely
\[
\int_{\Omega} \left( \langle \nabla^{\perp}(\bfeta \circ N) \, \colon \, \nabla^{\perp} \phi \rangle - \langle A \cdot (\bfeta \circ N) \, \colon \, A \cdot \phi \rangle - \Ri(\bfeta \circ N, \phi) \right) \, \dHa^{m} = 0,
\] 
that is $\bfeta \circ N$ solves \eqref{Jac_PDE} and the proof of $(i)$ is complete.

In order to prove $(ii)$, we take any $h \in \Gamma^{1,2}_{Q}(\N\Omega)$ such that $h|_{\partial\Omega} = N|_{\partial \Omega}$ and we show that
\[
\Jac(u, \Omega) \leq \Jac(\tilde{h}, \Omega),
\]
with $\tilde{h} = \sum_{\ell} \llbracket h^{\ell} + \zeta \rrbracket$. We compute:
\[
\begin{split}
\Jac(u, \Omega) &= \int_{\Omega} \sum_{\ell=1}^{Q} \left( |\nabla^{\perp}(N^{\ell} + \zeta)|^{2} - |A \cdot (N^{\ell} + \zeta)|^{2} - \Ri(N^{\ell} + \zeta, N^{\ell} + \zeta) \right) \, \dHa^m \\
&= \Jac(N, \Omega) + Q \left( \int_{\Omega} (|\nabla^{\perp}\zeta|^{2} - |A \cdot \zeta|^{2} - \Ri(\zeta, \zeta)) \, \dHa^{m} \right) \\
&+ 2Q \left( \int_{\Omega} (\langle \nabla^{\perp}(\bfeta \circ N) \, \colon \, \nabla^{\perp}\zeta \rangle - \langle A \cdot (\bfeta \circ N) \, \colon \, A \cdot \zeta \rangle - \Ri(\bfeta \circ N, \zeta)) \, \dHa^m \right).
\end{split}
\]
Using that $\Jac(N, \Omega) \leq \Jac(h, \Omega)$ and recalling the definition of $\tilde{h}$, we see that
\[
\begin{split}
\Jac(u, \Omega) - \Jac(\tilde{h}, \Omega) &\leq 2Q \int_{\Omega} \langle \nabla^{\perp}(\bfeta \circ N - \bfeta \circ h) \, \colon \, \nabla^{\perp}\zeta \rangle \\
&- 2Q \int_{\Omega} \langle A \cdot (\bfeta \circ N - \bfeta \circ h) \, \colon \, A \cdot \zeta \rangle \\
&- 2Q \int_{\Omega} \Ri(\bfeta \circ N - \bfeta \circ h, \zeta) = 0,
\end{split}
\]
because $\zeta$ is a Jacobi field and the function $\phi = \bfeta \circ N - \bfeta \circ h$ is a $W^{1,2}$ section of the normal bundle vanishing at $\partial \Omega$. This completes the proof.
\end{proof}

\begin{remark}
As a simple corollary of Lemma \ref{zero_average}, if $N$ is $\Jac$-minimizing then the $Q$-valued map $\tilde{N} = \sum_{\ell} \llbracket N^{\ell} - \bfeta \circ N \rrbracket$ is a $\Jac$-minimizer with $\bfeta \circ \tilde{N} \equiv 0$ and the same singular set as $N$. Therefore, there is no loss of generality in assuming that $\bfeta \circ N \equiv 0$, and, thus, that every $Q$-point $p$ satisfies $N(p) = Q \llbracket 0 \rrbracket$. In particular, when $p \in {\rm D}_{Q}(N)$ we can apply all the results of the previous section that were proved under the above assumption. Furthermore, the content of Proposition \ref{dicotomia} becomes more apparent in this context. Indeed, the dichotomy stated in there discriminates perfectly between regular and singular top-multiplicity points: $p \in \reg_{Q}(N)$ if and only if the condition $(i)$ is observed; on the other hand, $p \in \sing_{Q}(N)$ if and only if the frequency function is well defined and bounded in a neighborhood of $p$. 
\end{remark}

In view of the above remark, we assume from this point onwards that $N$ is $\Jac$-minimizing and such that $\bfeta \circ N = 0$. We fix a point $p \in \sing_{Q}(N)$, and an orthonormal basis \[\left( e_{1}, \dots, e_{m}, e_{m+1}, \dots, e_{m+k}, e_{m+k+1}, \dots, e_{d} \right)\] of the euclidean space $\R^d$ with the property that $T_{p}\Sigma = {\rm span}(e_{1}, \dots, e_{m})$ and $T^{\perp}_{p}\Sigma = {\rm span}(e_{m+1}, \dots, e_{m+k})$. Choose local orthonormal frames $\left( \xi_i \right)_{i=1}^{m}$ and $\left( \nu_{\alpha} \right)_{\alpha=1}^{k}$ of $\T\Sigma$ and $\N\Sigma$ respectively which extend the basis at $p$, that is, such that $\xi_{i}(p) = e_{i}$ for $i = 1,\dots,m$ and $\nu_{\alpha}(p) = e_{m+\alpha}$ for $\alpha = 1,\dots,k$. With a slight abuse of notation, we will sometimes denote the linear subspace $\R^{m} \times \{0\} \times \{0\}$ by $\R^m$ and $\{0\} \times \R^k \times \{0\}$ by $\R^k$.

Let $r_{0} > 0$ be such that all the conclusions from the previous paragraphs hold. For every $r \in \left( 0, r_{0} \right]$, translate and rescale the manifolds $\M$ and $\Sigma$, setting
\[
\M_{r} := \frac{\M - p}{r}, \hspace{0.5cm} \Sigma_{r} := \frac{\Sigma - p}{r},
\]
that is $\M_{r} = \iota_{p,r}(\M)$ and $\Sigma_{r} = \iota_{p,r}(\Sigma)$, where
\[
\iota_{p,r}(x) := \frac{x - p}{r}.
\]
The manifolds $\M_{r}$ and $\Sigma_{r}$ will be regarded as Riemannian submanifolds of $\R^d$ with the induced euclidean metric. We will let
\[
{\bf ex}_{r} \colon B_{1} \subset T_{0}\Sigma_{r} \simeq \R^m \to \Sigma_r
\]
be the exponential map, and we will use the symbol $\psi_{p,r}$ to denote the map
\[
\psi_{p,r} := \iota_{p,r}^{-1} \circ {\bf ex}_{r}.
\]
Observe that $\psi_{p,r}$ maps the euclidean ball $B_{1}(0) \subset \R^m$ diffeomorphically onto the geodesic ball $\B_{r}(p) \subset \Sigma$.

\begin{remark} \label{trivial_geom}
Observe that, since $T_{0}\M_{r} = T_{p}\M = {\rm span}(e_1,\dots,e_{m+k})$ for every $r$, the ambient manifolds $\M_{r}$ converge, as $r \downarrow 0$, to $\R^{m+k} \times \{0\}$ in $C^{3,\beta}$. For the same reason, the $\Sigma_{r}$'s converge to $\R^{m} \times \{0\} \times \{0\}$ in $C^{3,\beta}$ and the exponential maps ${\bf ex}_{r}$ converge in $C^{2,\beta}$ to the identity map of the ball $B_{1} \subset \R^m$ (cf. \cite[Proposition A.4]{DLS13c}).
\end{remark}

\begin{definition}
We define the \emph{blow-ups} of $N$ at $p$ as the one-parameter family of maps $N_{p,r} \colon B_{1} \subset T_{0}\Sigma_{r} \to \A_{Q}(\R^d)$ indexed by $r \in \left( 0, r_{0} \right]$ and given by
\begin{equation} \label{blow-ups}
N_{p,r}(y) := \frac{r^{\frac{m}{2}} N(\psi_{p,r}(y))}{\|N\|_{L^{2}(\B_{r}(p))}} = \frac{r^{\frac{m}{2}} \, N(p+r{\bf ex}_{r}(y))}{\| N \|_{L^{2}(\B_{r}(p))}}.
\end{equation}
\end{definition}

Observe that the maps $N_{p,r}$ are well defined because $N$ is not vanishing in any ball $\B_{r}(p)$ with $0 < r \leq r_{0}$.

The next theorem is the anticipated convergence result for the blow-up maps.

\begin{theorem} \label{blow-up:thm}
Let $N \in \Gamma_{Q}^{1,2}(\N\Omega)$ be $\Jac$-minimizing with $\bfeta \circ N = 0$. Assume $p \in \sing_{Q}(N)$. Then, for any sequence $N_{p,r_j}$ with $r_j \downarrow 0$, a subsequence, not relabeled, converges weakly in $W^{1,2}$, strongly in $L^2$ and locally uniformly to a continuous $Q$-valued function $\mathscr{N}_{p} \colon B_{1} \subset \R^m \to \A_{Q}(\R^k)$ such that:
\begin{itemize}
\item[$(a)$] $\mathscr{N}_{p}(0) = Q \llbracket 0 \rrbracket$ and $\bfeta \circ \mathscr{N}_{p} \equiv 0$, but $\| \mathscr{N}_{p} \|_{L^{2}(B_1)} = 1$, and thus, in particular, $\mathscr{N}_{p}$ is non-trivial;
\item[$(b)$] $\mathscr{N}_{p}$ is locally $\Dir$-minimizing in $B_1$;
\item[$(c)$] $\mathscr{N}_{p}$ is $\mu$-homogeneous, with $\mu = I_{0}(p)$, the frequency of $N$ at $p$ defined in \eqref{am_impr_eq:2}.
\end{itemize} 
\end{theorem}

\begin{remark} \label{blow-up:rmk}
Any map $\mathscr{N}_{p}$ which is the limit of a blow-up sequence $N_{p,r_{j}}$ in the sense specified above will be called a \emph{tangent map} to $N$ at $p$.
\end{remark}

\begin{proof}
Let $N$ and $p$ be as in the statement. For any sequence $r_{j} \downarrow 0$, let us denote $\M_{j} := \M_{r_j}$, $\Sigma_j := \Sigma_{r_j}$, ${\bf ex}_j := {\bf ex}_{r_j}$, $\psi_{j} := \psi_{p, r_j}$ and $N_{j} := N_{p,r_j}$ in order to simplify the notation. We will divide the proof into steps.

\textit{Step 1: boundedness in $W^{1,2}$.} Assume for the moment that $j \in \mathbb{N}$ is fixed. We start estimating $\| N_{j} \|_{L^{2}(B_1)}$. Changing coordinates $x = \psi_{j}(y)$ in the integral, we compute explicitly
\begin{equation}  \label{BU:-1}
\begin{split}
\| N \|_{L^2(\B_{r_j}(p))}^{2} &= \int_{\B_{r_j}(p)} |N(x)|^{2} \, \dHa^m(x) \\
&= \| N \|_{L^{2}(\B_{r_j}(p))}^{2} \int_{B_1} |N_{j}(y)|^{2} \, {\bf J}{\bf ex}_{j}(y) \, {\rm d}y, 
\end{split}
\end{equation}
and thus
\begin{equation} \label{BU:0}
\int_{B_{1}} |N_{j}(y)|^{2} \, {\bf J}{\bf ex}_{j}(y) \, {\rm d}y = 1
\end{equation}
for every $j$. By the considerations in Remark \ref{trivial_geom}, we can deduce that necessarily
 
\begin{equation} \label{BU:1}
\frac{1}{2} \leq \| N_{j} \|_{L^{2}(B_1)}^{2} \leq 2
\end{equation}
when $j$ is large enough.

Next, we bound the Dirichlet energy of the blow-up maps in $B_1$. For any $y \in B_{1} \subset T_{0}\Sigma_j$, and for all $i = 1,\dots,m$, let $\varepsilon_{i} = \varepsilon_{i}(y)$ be the vector in $T_{0}\Sigma_{j}$ defined by
\[
\mathrm{d}({\bf ex}_{j})|_{y} \cdot \varepsilon_{i}(y) = \xi_{i}(p+r_{j}{\bf ex}_{j}(y)).
\]
We note that, when $j \uparrow \infty$, $\varepsilon_{i}$ converges to $e_{i} = \xi_{i}(p)$ uniformly in $B_{1}$.

Again by changing variable $x = \psi_{j}(y)$ in the integral, we compute:
\begin{equation} \label{BU:2}
\begin{split}
\Dir(N,\B_{r_j}(p)) &= \int_{\B_{r_j}(p)} \sum_{i=1}^{m} |D_{\xi_i}N(x)|^{2} \, \dHa^{m}(x) \\
&= \frac{\| N \|_{L^{2}(\B_{r_j}(p))}^{2}}{r_{j}^2} \int_{B_{1}} \sum_{i=1}^{m} |D_{\varepsilon_i}N_{j}(y)|^{2} \, {\bf J}{\bf ex}_{j}(y) \, {\rm d}y. 
\end{split}
\end{equation}
On the other hand, using that $N$ takes values in the normal bundle, we have the usual estimate
\begin{equation} \label{BU:3}
\Dir(N, \B_{r_j}(p)) \leq \Dir^{\N\Sigma}(N, \B_{r_j}(p)) + C_{0} \| N \|_{L^2(\B_{r_j}(p))}^{2},
\end{equation}
for some positive geometric constant $C_{0} = C_{0}({\bf A}, {\bf \overline{A}})$. From \eqref{BU:2} and \eqref{BU:3} we conclude that for any $j$
\begin{equation} \label{BU:4}
\int_{B_{1}} \sum_{i=1}^{m} |D_{\varepsilon_i}N_{j}(y)|^{2} \, {\bf J}{\bf ex}_{j}(y) \, {\rm d}y \leq \frac{r_{j}^{2} \, \Dir^{\N\Sigma}(N, \B_{r_j}(p))}{\| N \|_{L^{2}(\B_{r_j}(p))}^{2}} + C_{0} r_{j}^{2} \leq C(1 + r_{j}^{2}),
\end{equation} 
because of the reverse Poincaré inequality \eqref{reverse_Poinc:eq}. Thus, we conclude that the Dirichlet energy of the blow-up maps in $B_{1}$ is bounded:
\begin{equation} \label{BU:5}
\Dir(N_{j}, B_{1}) := \int_{B_{1}} \sum_{i=1}^{m} |D_{e_i}N_j|^{2} \, {\rm d}y \leq C.
\end{equation}

\textit{Step 2: convergence.} The $W^{1,2}$ bounds given by estimates \eqref{BU:1} and \eqref{BU:5}, together with Proposition \ref{embeddings}, clearly imply the $W^{1,2}$-weak and $L^2$-strong convergence of a subsequence in $B_{1}$. We claim now that the $N_{j}$'s are locally H\"older equi-continuous. This is an easy consequence of the H\"older estimate in \eqref{reg_Holder:est} and of the reverse Poincar\'e inequality. Indeed, for any $0 < \theta < 1$ and for any points $y_{1}, y_{2} \in \overline{B}_{\theta}$ one has the following:
\[
\begin{split}
\G(N_{j}(y_{1}), N_{j}(y_{2})) &= \frac{r_{j}^{\frac{m}{2}}}{\| N \|_{L^{2}(\B_{r_j}(p))}} \G(N(p + r_{j} {\bf ex}_{j}(y_1)), N(p + r_{j} {\bf ex}_{j}(y_2))) \\
&\leq \frac{r_{j}^{\frac{m}{2}}}{\| N \|_{L^{2}(\B_{r_j}(p))}} \left[ N \right]_{C^{0,\alpha}(\overline{\B}_{\theta r_j}(p))} {\bf d}(p + r_{j} {\bf ex}_{j}(y_1), p + r_{j} {\bf ex}_{j}(y_2))^{\alpha} \\
&\overset{\eqref{reg_Holder:est}}{\leq} C \frac{r_{j}}{\| N \|_{L^{2}(\B_{r_j}(p))}} \left( \Dir^{\N\Sigma}(N, \B_{r_j}(p)) + (\Lambda + C_{0}) | N \|_{L^{2}(\B_{r_j}(p))}^{2} \right)^{\sfrac{1}{2}} |y_{1} - y_{2}|^{\alpha} \\
&\overset{\eqref{reverse_Poinc:eq}}{\leq} C(1 + r_{j}) |y_{1} - y_{2}|^{\alpha}.
\end{split}
\]
Hence, for every $0 < \theta < 1$ there exists $C = C(\theta) > 0$ such that
\begin{equation} \label{unif_Holder}
\left[ N_{j} \right]_{C^{0,\alpha}(\overline{B}_{\theta})} := \sup_{y_{1}, y_{2} \in \overline{B}_{\theta}} \frac{\G(N_{j}(y_1), N_j(y_2))}{|y_1 - y_2|^{\alpha}} \leq C 
\end{equation}
for all $j$. Since $N_{j}(0) = Q\llbracket 0 \rrbracket$, the $N_{j}$'s are also locally uniformly bounded, hence the Ascoli-Arzel\`a theorem implies that, up to extracting another subsequence if necessary, the convergence is locally uniform, and the limit is a continuous $Q$-valued function $\mathscr{N}_{p} \colon B_{1} \subset \R^m \to \A_{Q}(\R^d)$.

\textit{Step 3: properties of the limit: proof of $(a)$.} It is immediate to see that $\bfeta \circ \mathscr{N}_{p} \equiv 0$. Indeed, from the assumption that $\bfeta \circ N \equiv 0$ and the definition of the blow-up maps, we deduce that $\bfeta \circ N_{j} \equiv 0$ for every $j$. Now, the $N_{j}$'s converge to $\mathscr{N}_{p}$ locally uniformly, and thus
\[
\bfeta \circ \mathscr{N}_{p}(y) = \lim_{j \to \infty} \bfeta \circ N_{j}(y) = 0
\]
for all $y \in B_1$. With the same argument, using the pointwise convergence of $N_{j}$ to $\mathscr{N}_{p}$ and the fact that $N_{j}(0) = Q \llbracket 0 \rrbracket$ for every $j$ we conclude that $\mathscr{N}_{p}(0) = Q \llbracket 0 \rrbracket$.

Nonetheless, the map $\mathscr{N}_{p}$ is non-trivial. Indeed, since $N_{j} \to \mathscr{N}_{p}$ strongly in $L^2$, estimate \eqref{BU:0} guarantees that:
\begin{equation}
\| \mathscr{N}_{p} \|_{L^{2}(B_1)}^{2} = \lim_{j \to \infty} \int_{B_1} |N_{j}(y)|^{2} \, {\bf J}{\bf ex}_{j}(y) \, {\rm d}y = 1.
\end{equation}

Next, we see that $\mathscr{N}_{p}(y) \in \A_{Q}(\R^k)$ for every $y \in B_1$. Indeed, considering the projection $\mathscr{N}_{p}^{(1)}$ of $\mathscr{N}_{p}$ onto the subspace $\R^m \times \{0\} \times \{0\}$ we easily infer that
\[
\begin{split}
\int_{B_{1}} |\mathscr{N}_{p}^{(1)}(y)|^{2} \, {\rm d}y &= \int_{B_1} \sum_{\ell=1}^{Q} \sum_{i=1}^{m} |\langle \mathscr{N}_{p}^{\ell}(y), e_i \rangle|^{2} \, {\rm d}y \\
&= \lim_{j \to \infty} \int_{B_1} \sum_{\ell=1}^{Q} \sum_{i=1}^{m} | \langle N_{j}^{\ell}(y), \xi_{i}(p+r_j {\bf ex}_{j}(y)) \rangle|^{2} \, {\rm d}y = 0, 
\end{split}
\]
because of the definition of $N_j$. Analogously, the projection $\mathscr{N}_{p}^{(3)}$ onto $\{0\} \times \{0\} \times \R^K$ satisfies
\[
\begin{split}
\int_{B_1} |\mathscr{N}_{p}^{(3)}(y)|^{2} \, {\rm d}y &= \int_{B_1} \sum_{\ell=1}^{Q} \sum_{\beta=1}^{K} |\langle \mathscr{N}_{p}^{\ell}(y), e_{m+k+\beta} \rangle|^{2} \, {\rm d}y \\
&= \lim_{j \to \infty} \int_{B_1} \sum_{\ell=1}^{Q} \sum_{\beta=1}^{K} |\langle N_{j}^{\ell}(y), \eta_{\beta}(p+r_j {\bf ex}_{j}(y)) \rangle|^{2} \, {\rm d}y = 0,
\end{split}
\]
where the $\eta_{\beta}$'s are a local orthonormal frame of the normal bundle of $\M$ in $\R^d$ extending the $e_{m+k+\beta}$'s in a  neighborhood of $p$.

\textit{Step 4: harmonicity of the limit: proof of $(b)$.} We show now that $\mathscr{N}_{p}$ is locally $\Dir$-minimizing in $B_1$ and, moreover, that for every $0 < \rho < 1$ the following identity holds true:
\begin{equation} \label{conv_energy}
\Dir(\mathscr{N}_{p}, B_{\rho}) = \liminf_{j \to \infty} \Dir(N_{j}, B_{\rho}).
\end{equation}

In order to obtain the proof of the above claim, we need to exploit the minimizing property of the Jacobi $Q$-valued field $N$ in order to deduce some crucial information on the blow-up sequence $N_j$. Fix $j \in \mathbb{N}$, and let $u \colon B_{1} \subset T_{0}\Sigma_{j} \simeq \R^{m} \to \A_{Q}(\R^{d})$ be any $W^{1,2}$ map such that $u|_{\partial B_{1}} = N_{j}|_{\partial B_{1}}$. Then, the map $\tilde{u} \in W^{1,2}\left( \B_{r_j}(p) , \A_{Q}(\R^d)\right)$ defined by
\[
\begin{split}
\tilde{u}(x) &:= r_{j}^{- \frac{m}{2}} \| N \|_{L^{2}(\B_{r_j}(p))} \left( u \circ \psi_{j}^{-1} \right)^{\perp}(x) \\
&= r_{j}^{- \frac{m}{2}} \| N \|_{L^{2}(\B_{r_j}(p))} \sum_{\ell=1}^{Q} \left\llbracket \sum_{\beta = 1}^{k} \langle u^{\ell} \circ \psi_{j}^{-1}(x), \nu_{\beta}(x) \rangle \nu_{\beta}(x) \right\rrbracket
\end{split}
\]
is a section of $\N\Sigma$ in $\B_{r_{j}}(p)$ such that $\tilde{u}|_{\partial \B_{r_j}(p)} = N|_{\partial \B_{r_j}(p)}$. By minimality, it follows then that
\begin{equation} \label{BU:Dir-1}
\Jac(N, \B_{r_{j}}(p)) \leq \Jac(\tilde{u}, \B_{r_{j}}(p)).
\end{equation}
Standard computations show that \eqref{BU:Dir-1} is equivalent to the condition
\begin{equation} \label{BU:Dir0}
\mathscr{F}_{j}(N_{j}) \leq \mathscr{F}_{j}(u),
\end{equation}
where $\mathscr{F}_{j}(u)$ is the functional defined by
\begin{equation} \label{BU:Dir1}
\begin{split}
\mathscr{F}_{j}(u) :&= \int_{B_1} \sum_{\ell=1}^{Q} \sum_{i=1}^{m} \sum_{\alpha=1}^{k} \left| \langle D_{\varepsilon_i}u^{\ell}(y), \nu_{\alpha} \circ \psi_{j}(y) \rangle + r_{j} \langle u^{\ell}(y), (D_{\xi_i}\nu_{\alpha} - \nabla^{\perp}_{\xi_i}\nu_{\alpha}) \circ \psi_{j}(y) \right|^{2} \, \mathbf{J}\mathbf{ex}_{j}(y) \, {\rm d}y \\
&- r_{j}^{2} \int_{B_1} \sum_{\ell=1}^{Q} \left( \left| A \circ \psi_{j}(y) \cdot (u^{\ell})^{\perp_{\psi_{j}(y)}} \right|^{2} + \Ri \circ \psi_{j}((u^{\ell})^{\perp_{\psi_{j}(y)}},(u^{\ell})^{\perp_{\psi_{j}(y)}}) \ \right) \, \mathbf{J}\mathbf{ex}_{j}(y) \, {\rm d}y \\
&= \mathscr{F}_{j}^{(1)}(u) + \mathscr{F}_{j}^{(2)}(u)
\end{split}
\end{equation}
on the space of $u \in W^{1,2}\left( B_{1} , \A_{Q}(\R^d) \right)$ such that $u|_{\partial B_1} = N_{j}|_{\partial B_1}$. Note that the following notation has been adopted in formula \eqref{BU:1}: $(u^{\ell})^{\perp_{\psi_{j}(y)}}$ is the orthogonal projection of the vector $u^{\ell}(y)$ onto $T^{\perp}_{\psi_{j}(y)}\Sigma$, given by
\[
(u^{\ell})^{\perp_{\psi_{j}(y)}} = \sum_{\beta = 1}^{k} \langle u^{\ell}(y), \nu_{\beta}(\psi_{j}(y)) \rangle \nu_{\beta}(\psi_{j}(y)).
\]
Hence, one has
\[
\left| A \circ \psi_{j}(y) \cdot (u^{\ell})^{\perp_{\psi_{j}(y)}} \right|^{2} = \sum_{i,h=1}^{m} \left| \sum_{\beta=1}^{k} A_{ih}^{\beta}(\psi_{j}(y)) \langle u^{\ell}(y), \nu_{\beta}(\psi_{j}(y)) \rangle \right|^{2}, 
\]
with $A_{ih}^{\beta} := \langle A(\xi_i, \xi_h), \nu_{\beta} \rangle$, and
\[
\Ri \circ \psi_{j}((u^{\ell})^{\perp_{\psi_{j}(y)}},(u^{\ell})^{\perp_{\psi_{j}(y)}}) = \sum_{i=1}^{m} \sum_{\beta, \gamma=1}^{k} R_{i\beta\gamma}^{i}(\psi_{j}(x)) \langle u^{\ell}(y), \nu_{\beta}(\psi_{j}(y)) \rangle \langle u^{\ell}(y), \nu_{\gamma}(\psi_{j}(y)) \rangle,
\]
with $R_{i\beta\gamma}^{i} := \langle R(\xi_{i}, \nu_{\beta})\nu_{\gamma}, \xi_{i} \rangle$.
\\

Now, for every $0 < \rho < 1$, set 
\[
D_{\rho} := \liminf_{j \to \infty} \Dir(N_{j}, B_{\rho}) = \liminf_{j \to \infty} \int_{B_{\rho}} \sum_{i=1}^{m} |D_{e_i} N_{j}|^{2} \, {\rm d}y,
\]
and suppose by contradiction that either $\mathscr{N}_{p}$ is not $\Dir$-minimizing in $B_{\rho}$ or $\Dir(\mathscr{N}_{p}, B_{\rho}) < D_{\rho}$ \footnote{Observe that the inequality 
\[
\Dir(\mathscr{N}_{p}, B_{\rho}) \leq D_{\rho}
\]
is guaranteed for every $\rho$ because the Dirichlet functional is lower semi-continuous with respect to weak convergence in $W^{1,2}$.
} for some $\rho$. In any of the two situations, there exists a $\rho_{0} > 0$ such that for any $\rho \geq \rho_{0}$ there exists a multiple valued map $g \in W^{1,2}(B_{\rho}, \A_{Q}(\R^k))$ with
\begin{equation} \label{BU:Dir2}
g|_{\partial B_\rho} = \mathscr{N}_{p}|_{\partial B_\rho} \hspace{0.3cm} \mbox{ and } \hspace{0.3cm} \gamma_{\rho} := D_{\rho} - \Dir(g, B_{\rho}) > 0.
\end{equation}

A simple application of Fatou's lemma shows that for almost every $\rho \in \left( 0,1 \right)$ both the quantities $\liminf_{j} \Dir(N_{j}, \partial B_{\rho})$ and $\liminf_{j} \| N_{j} \|_{L^{2}(\partial B_{\rho})}^{2}$ are finite:
\[
\int_{0}^{1} \liminf_{j \to \infty} \Dir(N_{j}, \partial B_{\rho}) \, \mathrm{d}\rho \leq \liminf_{j \to \infty} \int_{0}^{1} \Dir(N_{j}, \partial B_{\rho}) \, \mathrm{d}\rho = \liminf_{j \to \infty} \Dir(N_{j}, B_1) \leq M < \infty,
\]
\[
\int_{0}^{1} \liminf_{j \to \infty} \| N_{j} \|_{L^{2}(\partial B_{\rho})}^{2} \, \mathrm{d}\rho \leq \liminf_{j \to \infty} \int_{0}^{1} \| N_{j} \|_{L^{2}(\partial B_{\rho})}^{2} \, \mathrm{d}\rho = \lim_{j \to \infty} \| N_{j} \|_{L^{2}(B_{\rho})}^{2} = \| \mathscr{N}_{p} \|_{L^{2}(B_{\rho})}^{2} \leq 1.
\]
Therefore, passing if necessary to a subsequence, we can fix a radius $\rho \geq \rho_0$ such that
\begin{equation} \label{BU:Dir3}
\Dir(\mathscr{N}_{p}, \partial B_{\rho}) \leq \lim_{j \to \infty} \Dir(N_{j}, \partial B_{\rho}) \leq M < \infty
\end{equation}
and 
\begin{equation} \label{BU:Dir3bis}
\| \mathscr{N}_{p} \|_{L^{2}(\partial B_{\rho})}^{2} \leq \lim_{j \to \infty} \| N_{j} \|_{L^{2}(\partial B_{\rho})}^{2} \leq 1.
\end{equation}
This allows us also to fix the corresponding map $g$ satisfying the conditions in \eqref{BU:Dir2}. The strategy to complete the proof is now the following: we will use the map $g$ to construct, for every $j$, a competitor $u_{j}$ for the functional $\mathscr{F}_{j}$, that is a map $u_{j} \in W^{1,2}(B_{1}, \A_{Q}(\R^d))$ with $u_{j}|_{\partial B_1} = N_{j}|_{\partial B_1}$. Then, we will show that if $j$ is chosen sufficiently large then $\mathscr{F}_{j}(u_j) < \mathscr{F}_{j}(N_j)$, thus contradicting \eqref{BU:Dir0} and, in turn, the minimality of $N$ in $\B_{r_j}(p)$.

The construction of the maps $u_{j}$ is analogous to the one presented in \cite[Proposition 3.20]{DLS11a}: we fix a number $0 < \delta < \frac{\rho}{2}$ to be suitably chosen later, and for every $j \in \mathbb{N}$ we define $u_{j}$ on $B_1$ as follows:
\[
u_{j}(y) :=
\begin{cases}
g\left( \frac{\rho y}{\rho - \delta} \right) &\mbox{ for } y \in B_{\rho - \delta},\\
h_{j}(y) & \mbox{ for } y \in B_{\rho} \setminus B_{\rho - \delta}, \\
N_{j}(y) &\mbox{ for } y \in B_{1} \setminus B_{\rho},
\end{cases}
\] 
where the maps $h_{j}$ interpolate between $g\left( \frac{\rho y}{\rho - \delta} \right) = \mathscr{N}_{p}\left( \frac{\rho y}{\rho - \delta} \right) \in W^{1,2}(\partial B_{\rho - \delta}, \A_Q)$ and $N_{j} \in W^{1,2}(\partial B_{\rho}, \A_Q)$. Observe that the existence of the $h_{j}$'s is guaranteed by Proposition \ref{Luckhaus} (also cf. \cite[Lemma 2.15]{DLS11a}). 

As anticipated, the goal is now to show that this map $u_{j}$ has less $\mathscr{F}_{j}$ energy than $N_{j}$ when $j$ is big enough (and thus $r_j$ is suitably close to $0$). We first note that $u_{j}$ differs from $N_{j}$ only on $B_{\rho}$, therefore our analysis will be carried on this smaller ball only. Then, fix a small number $\theta > 0$, and recall that, in the limit as $j \uparrow \infty$, the exponential maps ${\bf ex}_{j}$ converge uniformly to the identity map of the unit ball in $\R^m$, whereas the maps $\psi_{j}$ converge uniformly to the constant map identically equal to $p$. Hence, the first line in the definition of $\mathscr{F}_{j}(u_{j})$ can be estimated by
\begin{equation} \label{BU:Dir4} 
\mathscr{F}_{j}^{(1)}(u_{j})\big|_{B_{\rho}} \leq (1 + \theta) \Dir(u_{j}, B_{\rho}) + \theta \| u_{j} \|_{L^{2}(B_{\rho})}^{2}
\end{equation}
for all $j \geq j_{0}(\theta)$. On the other hand, the definition of $u_{j}$ together with the estimate \eqref{interp_est} imply that
\begin{equation} \label{BU:Dir5}
\begin{split}
\Dir(u_j, B_\rho) &\leq \Dir(u_{j}, B_{\rho - \delta}) + C \delta \big( \Dir(u_{j}, \partial B_{\rho - \delta}) + \Dir(N_{j}, \partial B_{\rho}) \big) + \frac{C}{\delta} \int_{\partial B_{\rho}} \G(u_j, N_j)^{2} \\
&\leq \Dir(g, B_{\rho}) + C \delta \Dir(\mathscr{N}_{p}, \partial B_{\rho}) + C \delta \Dir(N_{j}, \partial B_{\rho}) + \frac{C}{\delta} \int_{\partial B_{\rho}} \G(\mathscr{N}_p, N_j)^{2}
\end{split}
\end{equation}
whereas
\begin{equation} \label{BU:Dir5bis}
\begin{split}
\| u_{j} \|_{L^2(B_\rho)}^{2} &= \| u_{j} \|^{2}_{L^2(B_{\rho - \delta})} + \| u_{j} \|^{2}_{L^2(B_\rho \setminus B_{\rho - \delta})} \\
&\leq \| g \|_{L^2(B_\rho)}^{2} + \| h_j \|_{L^2(B_\rho \setminus B_{\rho - \delta})}^{2} \\
&\overset{\eqref{interp_L2}}{\leq} \| g \|_{L^2(B_\rho)}^{2} + C \delta \left( \| \mathscr{N}_{p} \|^{2}_{L^2(\partial B_\rho)} + \|N_j\|^{2}_{L^2(\partial B_\rho)} \right) \\
&\overset{\eqref{BU:Dir3bis}}{\leq} \| g \|_{L^{2}(B_{\rho})}^{2} + 3C\delta.
\end{split}
\end{equation}

Concerning the second term in the functional $\mathscr{F}_{j}$, it is easy to compute that
\begin{equation} \label{BU:Dir6}
\begin{split}
\mathscr{F}_{j}^{(2)}(u_{j})\big|_{B_{\rho}} &\leq \mathscr{F}_{j}^{(2)}(N_{j})\big|_{B_{\rho}} + C r_{j}^{2} \left( \int_{B_{\rho}} |N_{j}|^{2} \, {\rm d}y \right)^{\frac{1}{2}} \left( \int_{B_{\rho}} \G(u_{j}, N_{j})^{2} \, {\rm d}y \right)^{\frac{1}{2}} + C r_{j}^{2} \int_{B_{\rho}} \G(u_{j}, N_{j})^{2} \, {\rm d}y \\
&\leq \mathscr{F}_{j}^{(2)}(N_{j})\big|_{B_{\rho}} + C r_{j}^{2},
\end{split}
\end{equation}
because the $L^{2}$ norms of both the maps $u_{j}$ and $N_{j}$ are uniformly bounded in $j$. 
\\

We can finally close the argument. Choose $\delta$ such that $4 C \delta (M+1) \leq \gamma_\rho$, where $M$ and $\gamma_{\rho}$ are the constants in \eqref{BU:Dir3} and \eqref{BU:Dir2} respectively. Invoking \eqref{BU:Dir3} and using the uniform convergence of $N_j$ to $\mathscr{N}_p$, from \eqref{BU:Dir5} follows
\begin{equation} \label{BU:Dir8}
\begin{split}
\Dir(u_j, B_{\rho}) &\leq D_{\rho} - \gamma_{\rho} + C \delta M + C \delta (M+1) + \frac{C}{\delta} \int_{\partial B_{\rho}} \G(\mathscr{N}_p, N_j)^{2} \\
&\leq D_{\rho} - \frac{\gamma_\rho}{2} + \frac{C}{\delta} \int_{\partial B_{\rho}} \G(\mathscr{N}_p, N_j)^{2} \leq D_{\rho} - \frac{\gamma_\rho}{4}
\end{split}
\end{equation}
whenever $j$ is sufficiently large. A suitable choice of the parameter $\theta$ in \eqref{BU:Dir4} depending on $\gamma_\rho$, $D_{\rho}$ and $\| g \|_{L^{2}(B_{\rho})}$ allows to conclude that for $j$ big enough (depending on the same quantities):
\begin{equation} \label{BU:Dir9}
\begin{split}
\mathscr{F}_{j}^{(1)}(u_{j})\big|_{B_{\rho}} &\leq D_{\rho} - \frac{\gamma_{\rho}}{8} \\
&\leq \Dir(N_j, B_{\rho}) - \frac{\gamma_\rho}{16} \\
&\leq \mathscr{F}_{j}^{(1)}(N_{j})\big|_{B_{\rho}} - \frac{\gamma_\rho}{32}.
\end{split}
\end{equation}
Observe that in the last inequality we have used again that the manifolds $\Sigma_j$ are becoming more and more flat in the limit $j \uparrow \infty$, and also that the projection of the $N_j$'s on the orthogonal complement to $\R^k$ is vanishing in an $L^2$ sense in the same limit. Now, summing \eqref{BU:Dir6} and \eqref{BU:Dir9}, we conclude:
\begin{equation} \label{BU:Dir10}
\mathscr{F}_{r_j}(u_j) \leq \mathscr{F}_{r_j}(N_j) - \frac{\gamma_\rho}{32} + C r_j^2.
\end{equation}
The desired contradiction is immediately obtained by choosing $j$ so big that $C r_{j}^{2} \leq \frac{\gamma_\rho}{64}$.

\textit{Step 5: homogeneity of the limit: proof of $(c)$.} We conclude the proof of the theorem showing that the limit map $\mathscr{N}_{p}$ admits a homogeneous extension to the whole $\R^m$. In other words, the goal is to show that
\[
\mathscr{N}_{p}\left( \frac{\rho y}{|y|} \right) = \left( \frac{\rho}{|y|} \right)^{\mu} \mathscr{N}_{p}(y)
\]
for all $y \in B_{1} \setminus \{0\}$, for all $0 < \rho < 1$, and with $\mu = I_{0}(p)$.

 The strategy is to take advantage of \cite[Corollary 3.16]{DLS11a}: since $\mathscr{N}_{p}$ is $\Dir$-minimizing, in order to prove that it is homogeneous it suffices to show that its frequency function at the origin $y = 0$ is constant. Hence, we set for $0 < \rho < 1$:
\begin{equation} \label{lim_freq}
\mathscr{I}(\rho) := \frac{\rho \mathscr{D}(\rho)}{\mathscr{H}(\rho)},
\end{equation}
where
\begin{equation} \label{lim_D}
\mathscr{D}(\rho) := \Dir(\mathscr{N}_p, B_{\rho}) = \int_{B_{\rho}} |D \mathscr{N}_{p}|^{2} \, {\rm d}y = \int_{B_{\rho}} \sum_{\ell=1}^{Q} \sum_{i=1}^{m} \sum_{\alpha=1}^{k} |\langle D_{e_i}\mathscr{N}_{p}^{\ell}, e_{m+\alpha} \rangle|^{2} \, {\rm d}y,
\end{equation}
and
\begin{equation} \label{lim_H}
\mathscr{H}(\rho) := \int_{\partial B_{\rho}} |\mathscr{N}_{p}|^{2} \, {\rm d}y.
\end{equation}

We first observe that $\mathscr{I}(\rho)$ is well defined for all $\rho \in \left( 0, 1 \right)$. Indeed, if there is $\rho_{0}$ such that $\mathscr{H}(\rho_0) = 0$, then by minimality it must be $\mathscr{N}_{p} \equiv Q \llbracket 0 \rrbracket$ in $B_{\rho_0}$. On the other hand, the unique continuation property of $\Dir$-minimizers (cf. \cite[Lemma 7.1]{DLS13b}) would then imply that $\mathscr{N}_{p} \equiv Q \llbracket 0 \rrbracket$ in the whole $B_{1}$, which in turn contradicts the fact that $\| \mathscr{N}_{p} \|_{L^{2}(B_1)} = 1$. In other words, this shows that the origin is singular for $\mathscr{N}_{p}$.

Extract, if necessary, a subsequence such that the $\liminf$ in \eqref{conv_energy} can be replaced by a $\lim$, and compute:
\begin{equation}
\begin{split} \label{hom:1}
\mathscr{I}(\rho) &= \frac{\rho \int_{B_{\rho}} \sum_{\ell=1}^{Q} \sum_{i=1}^{m} \sum_{\alpha=1}^{k} |\langle D_{e_i}\mathscr{N}_{p}^{\ell}, e_{m+\alpha} \rangle|^{2} \, {\rm d}y}{\int_{\partial B_{\rho}} |\mathscr{N}_{p}|^{2} \, {\rm d}y} \\
&= \lim_{j \to \infty} \frac{\rho \int_{B_{\rho}} \sum_{\ell=1}^{Q} \sum_{i=1}^{m} \sum_{\alpha=1}^{k} |\langle D_{\varepsilon_{i}}N_{j}^{\ell}, \nu_{\alpha} \circ \psi_{j} \rangle|^{2} \, {\bf J}{\bf ex}_{j}(y) \, {\rm d}y}{\int_{\partial B_{\rho}} |N_{j}|^{2} \, {\bf J}{\bf ex}_{j}(y) \, {\rm d}y} \\
&= \lim_{j \to \infty} \frac{\rho r_{j} \Dir^{\N\Sigma}(N, \B_{\rho r_j}(p))}{\int_{\partial\B_{\rho r_j}(p)} |N|^{2} \, \dHa^{m-1}} \\
&= \lim_{j \to \infty} \frac{\rho r_{j} {\bf D}(\rho r_j)}{{\bf H}(\rho r_j)} = \lim_{j\to \infty} {\bf I}(\rho r_j) = I_{0},
\end{split}
\end{equation}
where we have used (modifications of) formulae \eqref{BU:-1}, \eqref{BU:2} and finally \eqref{am_impr:6bis}. 

As already anticipated, \cite[Corollary 3.16]{DLS11a} implies now that $\mathscr{N}_{p}$ is a $\mu$-homogeneous $Q$-valued function, with $\mu = I_{0}(p) > 0$. 
\end{proof}

\begin{remark}
Note that from the proof of Theorem \ref{blow-up:thm} it follows that the convergence of (a subsequence of) the $N_{p,r_{j}}$ to $\mathscr{N}_{p}$ is actually \emph{strong} in $W^{1,2}$ in any ball $B_{\rho} \subset B_{1}$ (cf. formula \eqref{conv_energy}). This stronger convergence has been in fact tacitly used in deriving \eqref{hom:1}.
\end{remark}

\section{The closing argument: proof of Theorem \ref{sing:thm}} \label{sec:sing}

\begin{proposition} \label{sing_Q:thm}
Let $N \in \Gamma_{Q}^{1,2}(\N\Omega)$ be $\Jac$-minimizing. Assume $\Omega \subset \Sigma^{m}$ is connected. Then:
\begin{itemize}
\item[$(i)$] either $N = Q \llbracket \zeta \rrbracket$ with $\zeta \colon \Omega \to \R^d$ a classical Jacobi field,
\item[$(ii)$] or the set ${\rm D}_{Q}(N)$ of multiplicity $Q$ points is a relatively closed proper subset of $\Omega$ consisting of isolated points if $m = 2$ and with $\dim_{\Ha}({\rm D}_{Q}(N)) \leq m - 2$ if $m \geq 3$.
\end{itemize}
\end{proposition}

\begin{proof}
Assume without loss of generality that $\bfeta \circ N = 0$, so that $p \in {\rm D}_{Q}(N)$ if and only if $N(p) = Q \llbracket 0 \rrbracket$. We first observe the following fact: the set ${\rm D}_{Q}(N)$ is relatively closed in $\Omega$. This can be rapidly seen writing ${\rm D}_{Q}(N) = \sigma^{-1}(\{Q\})$, where $\sigma \colon \Omega \to \mathbb{N}$ is the function given by
\begin{equation} \label{sing_Q:1}
\sigma(x) := {\rm card}(\spt(N(x))),
\end{equation}
and noticing that, since $N$ is continuous, $\sigma$ is lower semi-continuous. 

We will now treat the two cases $m=2$ and $m \geq 3$ separately.

\textit{Case 1: dimension $m=2$.} In this case, we claim that the points $p \in \sing_{Q}(N)$ are isolated in ${\rm D}_{Q}(N)$. Assume by contradiction that this is not the case, and let $p \in \sing_{Q}(N)$ be the limit of a sequence $\{ x_{j} \}_{j=1}^{\infty}$ of points in ${\rm D}_{Q}(N)$. Set $r_{j} := {\bf d}(x_{j}, p)$. Since $r_{j} \downarrow 0$, by Theorem \ref{blow-up:thm} the corresponding blow-up family $N_{p,r_{j}}$ converges uniformly, up to a subsequence, to a $\Dir$-minimizing, $\mu$-homogeneous tangent map $\mathscr{N}_{p} \colon B_{1} \subset \R^{2} \to \A_{Q}(\R^{k})$ with $\| \mathscr{N}_{p} \|_{L^{2}(B_1)} = 1$ and $\bfeta \circ \mathscr{N}_{p} \equiv 0$. Moreover, since each $x_{j} \in {\rm D}_{Q}(N)$, the points $y_{j} := \psi_{p,r_{j}}^{-1}(x_{j})$ are a sequence of multiplicity $Q$ points for the corresponding $N_{p,r_{j}}$ in $\Sf^{1} = \partial B_{1}$: from this, we conclude that there exists $w \in \Sf^{1}$ such that $\mathscr{N}_{p}(w) = Q \llbracket 0 \rrbracket$. Up to rotations, we can assume that $w = e_{1}$. Denote $z := \frac{1}{2} e_{1}$, and observe that, since $\mathscr{N}_{p}$ is homogeneous, necessarily $\mathscr{N}_{p}(z) = Q \llbracket 0 \rrbracket$. Consider now the blow-up of $\mathscr{N}_{p}$ at $z$: by \cite[Lemma 3.24]{DLS11a}, any tangent map $h$ to $\mathscr{N}_{p}$ at $z$ is a non-trivial $\beta$-homogeneous $\Dir$-minimizer, with $\beta$ equal to the frequency of $\mathscr{N}_{p}$ at $z$, and such that $h(x_{1}, x_{2}) = \hat{h}(x_{2})$, for some function $\hat{h} \colon \R \to \A_{Q}(\R^k)$ which is $\Dir$-minimizing on every interval. Moreover, since $\Dir(h, B_{1}) > 0$, it must also be $\Dir(\hat{h}, I) > 0$, where $I := \left[ -1, 1 \right]$. On the other hand, every $1$-dimensional $\Dir$-minimizer $\hat{h}$ is affine, that is it has the form $\hat{h}(x) = \sum_{i=1}^{Q} \llbracket L_{i}(x) \rrbracket$, where the $L_{i}$'s are affine functions such that either $L_{i} \equiv L_{j}$ or $L_{i}(x) \neq L_{j}(x)$ for every $x \in \R$, for any $i,j$. Now, since $\hat{h}(0) = Q \llbracket 0 \rrbracket$, we deduce that $\hat{h} = Q \llbracket L \rrbracket$; on the other hand, $\bfeta \circ h \equiv 0$, and thus necessarily $L = 0$. This contradicts $\Dir(\hat{h}, I) > 0$.

Hence, if $p \in {\rm D}_{Q}(N)$ then either $p$ is isolated or, in case $p$ is a regular multiplicity $Q$ point, there exists an open neighborhood $V$ of $p$ such that $V \subset {\rm D}_{Q}(N)$. From this we deduce that ${\rm reg}_{Q}(N)$ is both open and closed in $\Omega$. Since $\Omega$ is connected, then either ${\rm reg}_{Q}(N) = \Omega$, and $N \equiv Q \llbracket 0 \rrbracket$, or ${\rm reg}_{Q}(N) = \emptyset$, and ${\rm D}_{Q}(N) = \sing_{Q}(N)$ consists of isolated points. This completes the proof in the dimension $m = 2$ case.

\textit{Case 2: dimension $m \geq 3$.} In this case, the goal is to show that $\Ha^{s}({\rm D}_{Q}(N)) = 0$ for every $s > m-2$, unless $N \equiv Q \llbracket 0 \rrbracket$. Consider the set $\sing_{Q}(N)$. We first claim that $\Ha^{s}(\sing_{Q}(N)) = 0$ for every $s > m - 2$. Suppose by contradiction that this is not the case. Then, by \cite[Theorem 3.6 (2)]{Simon83}, there exist $s > m - 2$ and a subset $F \subset \sing_{Q}(N)$ with $\Ha^{s}(F) > 0$ such that every point $p \in F$ is a point of positive upper $s$-density for the measure $\Ha^{s}_{\infty}$, that is
\begin{equation} \label{sing_Q:21}
 \limsup_{r \to 0} \frac{\Ha^{s}_{\infty}(\sing_{Q}(N) \cap \B_{r}(p))}{r^s} > 0 \hspace{0.5cm} \mbox{ for every } p \in F.
\end{equation}
Here, the symbol $\Ha^{s}_{\infty}$ denotes as usual the $s$-dimensional Hausdorff pre-measure, defined by
\[
\Ha^{s}_{\infty}(A) := \inf\left\lbrace \sum_{h=1}^{\infty} \omega_{s} \left( \frac{{\rm diam}(E_h)}{2} \right)^{s} \, \colon \, A \subset \bigcup_{h=1}^{\infty} E_h \right\rbrace,
\] 
with $\omega_{s} := \frac{\pi^{\frac{s}{2}}}{\Gamma\left( \frac{s}{2} + 1 \right)}$, where $\Gamma(s)$ is the usual Gamma function. Among the properties of $\Ha^{s}_{\infty}$, it is worth recalling now that it is upper semi-continuous with respect to Hausdorff convergence of compact sets: in other words, if $K_j$ is a sequence of compact sets converging to $K$ in the sense of Hausdorff, then
\begin{equation} \label{sing_Q:22}
\limsup_{j \to \infty} \Ha^{s}_{\infty}(K_j) \leq \Ha^{s}_{\infty}(K).
\end{equation} 

Now, \eqref{sing_Q:21} together with Theorem \ref{blow-up:thm} imply the existence of a point $p \in \sing_{Q}(N)$ and a sequence of radii $r_{j} \downarrow 0$ such that the blow-up maps $N_{j} = N_{p,2r_j}$ converge uniformly to a $\Dir$-minimizing homogeneous $Q$-valued function $\mathscr{N}_{p} \colon B_1 \subset \R^m \to \A_{Q}(\R^k)$ with $\bfeta \circ \mathscr{N}_p \equiv 0$ and $\| \mathscr{N}_p \|_{L^2(B_1)} = 1$, and furthermore such that
\begin{equation} \label{sing_Q:23}
\limsup_{j \to \infty} \frac{\Ha^{s}_{\infty}(\sing_{Q}(N) \cap \B_{r_j}(p))}{r_{j}^s} > 0,
\end{equation}
or, equivalently,
\begin{equation} \label{sing_Q:24}
\limsup_{j \to \infty} \Ha^{s}_{\infty}(\sing_{Q}(N_j) \cap B_{\frac{1}{2}}) > 0,
\end{equation}
where $B_{\frac{1}{2}} \subset T_{0}\Sigma_{j} \simeq \R^m$. Set $K_{j} := \sing_{Q}(N_j) \cap \overline{B}_{\frac{1}{2}}$, and observe that, since $N_j$ converge to $\mathscr{N}_p$ uniformly, the compact sets $K_j$ converge in the sense of Hausdorff to a compact set $K \subset {\rm D}_{Q}(\mathscr{N}_p)$. From the semi-continuity property \eqref{sing_Q:22}, we can therefore deduce that:
\begin{equation} \label{sing_Q:25}
\begin{split}
\Ha^{s}({\rm D}_{Q}(\mathscr{N}_p)) &\geq \Ha^{s}_{\infty}({\rm D}_{Q}(\mathscr{N}_p)) \geq \Ha^{s}_{\infty}(K) \\
&\geq \limsup_{j \to \infty} \Ha^{s}_{\infty}(K_j) \geq \limsup_{j \to \infty} \Ha^{s}_{\infty}(\sing_Q(N_j) \cap B_{\frac{1}{2}}) > 0.
\end{split}
\end{equation}
Since $s > m-2$, \cite[Proposition 3.22]{DLS11a} implies that this can happen only if $\mathscr{N}_p \equiv Q \llbracket \zeta \rrbracket$, where $\zeta \colon B_1 \to \R^k$ is a harmonic function. Since $\bfeta \circ \mathscr{N}_p \equiv 0$, then it must be $\mathscr{N}_p \equiv Q \llbracket 0 \rrbracket$, which in turns contradicts the fact that $\| \mathscr{N}_p \|_{L^2(B_1)} = 1$.

We can therefore conclude that necessarily $\Ha^{s}(\sing_{Q}(N)) = 0$ for every $s > m-2$. Since $\sing_{Q}(N) = \partial {\rm D}_{Q}(N) \cap \Omega$, either ${\rm D}_{Q}(N) = \Omega$ and $N \equiv Q \llbracket 0 \rrbracket$, or ${\rm D}_{Q}(N) = \sing_{Q}(N)$. The proof is complete.
\end{proof}

\begin{remark}
As a corollary of Proposition \ref{sing_Q:thm}, one easily deduces the following: if $N$ is a $\Jac$-minimizing $Q$-valued vector field in the open and connected subset $\Omega \subset \Sigma$ which is not of the form $N = Q \llbracket \zeta \rrbracket$ for some classical Jacobi field $\zeta$, then ${\rm D}_{Q}(N) = \sing_{Q}(N)$, that is all multiplicity $Q$ points are singular. 
\end{remark}

We have now all the tools that are needed to prove Theorem \ref{sing:thm}.

\begin{proof}[Proof of Theorem \ref{sing:thm}]
Since our manifolds are always assumed to be second-countable spaces, $\Omega$ can have at most countably many connected components, and these connected components are open. Hence, there is no loss of generality in assuming that $\Omega$ itself is connected: in the general case, we would just work on each connected component separately.

The fact that $\sing(N)$ is a relatively closed set in $\Omega$ (whereas $\reg(N)$ is open) is an immediate consequence of Definition \ref{sing_set}. Let $\sigma$ be the function defined in \eqref{sing_Q:1}. If $x \in \Omega$ is a regular point, then it is clear that $\sigma$ is continuous at $x$. On the other hand, assume $x$ is a point of continuity for $\sigma$, and write $N(x) = \sum_{j=1}^{J} k_{j} \llbracket P^j \rrbracket$, where the $k_{j}$'s are integers such that $\sum_{j=1}^{J} k_j = Q$, each $P^j \in T_{x}^{\perp}\Sigma$ and $P^i \neq P^j$ if $i \neq j$. Since the target of $\sigma$ is discrete, the fact that $\sigma$ is continuous at $x$ implies that in fact $\sigma(z) = J$ for all $z$ in a neighborhood $U$ of $x$. Hence, since $N$ is continuous, there exists a neighborhood $x \in V \subset U$ such that the map $N|_{V}$ admits a continuous decomposition $N(z) = \sum_{j=1}^{J} k_j \llbracket N^{j}(z) \rrbracket$, where each map $N^{j} \colon V \to \N\Sigma$ is a classical Jacobi field. Therefore, $x \in \reg(N)$.

The above argument implies that $\sing(N)$ coincides with the set of points where $\sigma$ is discontinuous. The proof that its Hausdorff dimension cannot exceed $m-2$ will be obtained via induction on $Q$. If $Q=1$, there is nothing to prove, since single-valued $\Jac$-minimizers are classical Jacobi fields. Assume now that the theorem holds for every $Q^{*}$-valued $\Jac$-minimizer with $Q^{*} < Q$ and we prove it for $N$. If $N \equiv Q \llbracket \zeta \rrbracket$ with $\zeta$ a classical Jacobi field, then $\sing(N)$ is empty, and the theorem follows. Assume, therefore, this is not the case. By Proposition \ref{sing_Q:thm}, the set ${\rm D}_{Q}(N) = \sing_{Q}(N)$ is a closed subset of $\Omega$ which is at most countable if $m = 2$ and has Hausdorff dimension at most $m-2$ if $m \geq 3$. Consider now the open set $\Omega' := \Omega \setminus {\rm D}_{Q}(N)$. Since $N$ is continuous, we can find countably many open geodesic balls $\B_j$ such that $\Omega' = \bigcup_{j} \B_j$ and $N|_{\B_j}$ can be written as the superposition of two multiple-valued functions:
\begin{equation} \label{sing:1}
N|_{\B_j} = \llbracket N_{j, Q_1} \rrbracket + \llbracket N_{j, Q_2} \rrbracket \hspace{0.5cm} \mbox{ with } Q_1 < Q, Q_2 < Q
\end{equation}
and
\begin{equation} \label{sing:2}
\spt(N_{j, Q_1}(x)) \cap \spt(N_{j, Q_2}(x)) = \emptyset \hspace{0.5cm} \mbox{ for every } x \in \B_j.
\end{equation}
From this last condition, it follows that
\begin{equation} \label{sing:3}
\sing(N) \cap \B_j = \sing(N_{j, Q_1}) \cup \sing(N_{j, Q_2}).
\end{equation}
The maps $N_{j, Q_1}$ and $N_{j, Q_2}$ are both $\Jac$-minimizing, and thus, by inductive hypothesis, their singular set has Hausdorff dimension at most $m - 2$, and is at most countable if $m=2$. Finally:
\begin{equation} \label{sing:4}
\sing(N) = \sing_{Q}(N) \cup \bigcup_{j=1}^{\infty} \left( \sing(N_{j, Q_1}) \cup \sing(N_{j,Q_2}) \right)
\end{equation}
also has Hausdorff dimension at most $m-2$ and is at most countable if $m=2$.
\end{proof}

\section{Uniqueness of tangent maps in dimension $2$} \label{chap:improved_dim2}

This last section is devoted to the proof of Theorem \ref{uniqueness_tg_map}, which we restate for convenience as follows. Observe that, since the average $\bfeta \circ N$ of a Jacobi $Q$-field $N$ is a classical Jacobi field, there is no loss of generality in assuming that $\bfeta \circ N \equiv 0$.
\begin{theorem}[Uniqueness of the tangent map at collapsed singularities] \label{uniqueness_tg_map_statement}
Let $\Sigma \hookrightarrow \M$ be as in Assumption \ref{assumptions}, with $m = \dim \Sigma = 2$. Let $N \in \Gamma_{Q}^{1,2}(\N\Omega)$ be $\Jac$-minimizing in the open set $\Omega \subset \Sigma^2$ with $\bfeta \circ N \equiv 0$, and let $p \in \Omega$ be such that $N(p) = Q \llbracket 0 \rrbracket$ but $N$ does not vanish in a neighborhood of $p$. Then, there is a \emph{unique} tangent map $\mathscr{N}_p$ to $N$ at $p$ (that is, the blow-up family $N_{p,r}$ defined in \eqref{blow-ups} converges locally uniformly to $\mathscr{N}_p$).
\end{theorem}

Theorem \ref{uniqueness_tg_map_statement} has the following natural corollary.
\begin{corollary} \label{uniqueness_everywhere}
Let $\Omega$ be an open subset of the two-dimensional manifold $\Sigma^2 \hookrightarrow \M$ as in Assumption \ref{assumptions}, and let $N \in \Gamma_{Q}^{1,2}(\N\Omega)$ be $\Jac$-minimizing. Then, for every $p \in \Omega$ there exists a unique tangent map $\mathscr{N}_p$ to $N$ at $p$.
\end{corollary}

\begin{proof}
The proof is by induction on $Q \geq 1$. If $Q = 1$ then the result is trivial, since $N$ is a classical Jacobi field. Let us then assume that the claim holds true for every $Q' < Q$, and we prove that it holds true for $Q$ as well. Let $N \in \Gamma_{Q}^{1,2}(\N\Omega)$ be $\Jac$-minimizing, and let $p \in \Omega$. Without loss of generality, assume that $\bfeta \circ N \equiv 0$. If $\diam(N(p)) > 0$, then, since $N$ is continuous, there exists a neighborhood $U$ of $p$ in $\Omega$ such that $\left. N \right|_{U} = \llbracket N^{1} \rrbracket + \llbracket N^{2} \rrbracket$, where each $N^{i} \in \Gamma_{Q_i}^{1,2}(\N U)$ is $\Jac$-minimizing, $Q_{i} < Q$ for $i=1,2$ and $\spt(N^1(x)) \cap \spt(N^2(x)) = \emptyset$ for every $x \in U$. By induction hypothesis, $N^{1}$ and $N^{2}$ have unique tangent maps $\mathscr{N}^{1}_p$ and $\mathscr{N}^2_p$ at $p$ respectively. Hence, $\mathscr{N}_p := \llbracket \mathscr{N}^1_p \rrbracket + \llbracket \mathscr{N}^2_p \rrbracket$ is the unique tangent map to $N$ at $p$. 

If, on the other hand, $\diam(N(p)) = 0$, $N(p) = Q \llbracket 0 \rrbracket$ because $N$ has zero average. If $N \equiv Q \llbracket 0 \rrbracket$ in a neighborhood of $p$, then the unique tangent map to $N$ at $p$ is $\mathscr{N}_p \equiv Q \llbracket 0 \rrbracket$. If $N$ does not vanish identically in any neighborhood of $p$, then the tangent map $\mathscr{N}_p$ is unique by Theorem \ref{uniqueness_tg_map_statement}. In either case, this completes the proof of the corollary.
\end{proof}

It only remains to prove Theorem \ref{uniqueness_tg_map_statement}. The plan is the following: first, in $\S$ \ref{sec:decays} we show that under the assumptions of Theorem \ref{uniqueness_tg_map_statement} the frequency function ${\bf I}_{r} = {\bf I}_{N,p}(r)$ converges, as $r \downarrow 0$, to its limit $I_0 = I_0(p) > 0$ with rate $r^{\beta}$ for some $\beta > 0$ (cf. Proposition \ref{prop:decays} below). Then, we will use this key fact to deduce Theorem \ref{uniqueness_tg_map_statement} in $\S$ \ref{sec:unicita}. 

\subsection{Decay of the frequency function} \label{sec:decays}

The main result of this section is the following proposition. Recall the definitions of the energy function ${\bf D}(r)$, the height function ${\bf H}(r)$ and the frequency function ${\bf I}(r)$.
\begin{proposition} \label{prop:decays}
Let $N \in \Gamma_{Q}^{1,2}(\N\Omega)$ be $\Jac$-minimizing in $\Omega \subset \Sigma^2$, and let $p$ be such that $N(p) = Q\llbracket 0 \rrbracket$ but $N$ does not vanish in a neighborhood of $p$. Let $I_0 := I_0(p) > 0$ be the frequency of $N$ at $p$ (which exists and is strictly positive by Proposition \ref{am_impr_thm}). Then, there are $\beta, C, D_0, H_0 > 0$ such that for every $r$ sufficiently small one has
\begin{equation} \label{eq:decays}
\abs{{\bf I}(r) - I_0} + \Abs{\frac{{\bf H}(r)}{r^{2I_0+1}} - H_0} + \Abs{\frac{{\bf D}(r)}{r^{2I_0}} - D_0} \leq C r^{\beta}\,.
\end{equation} 
\end{proposition}

We will need a preliminary lemma.
\begin{lemma}
Let $N$ and $p$ be as in Proposition \ref{prop:decays}. For every $\mu > 0$ there exists $\beta_0 = \beta_0(\mu)$ and $C = C(\mu)$ such that for every $0 < \beta < \beta_0$ the inequality
\begin{equation} \label{stima_original_fct_final}
{\bf D}(r) \leq \frac{r}{2(2\mu + \beta)} {\bf D}'(r) + \frac{\mu(\mu + \beta)}{r(2\mu + \beta)} {\bf H}(r) + C_{\mu} r {\bf D}(r)
\end{equation}
holds true for every $r$ small enough.
\end{lemma}

\begin{proof}

Let $r_{0} < {\rm inj}(\Sigma)$ be a radius such that ${\bf I}(r) = {\bf I}_{N,p}(r)$ is well defined and bounded in $\B_{r}(p)$ for every $0 < r < \min\lbrace r_{0}, \dist(p, \partial \Omega)\rbrace$. Recall that for every $0 < r < \min\lbrace r_{0}, \dist(p,\partial\Omega)\rbrace$ the exponential map $\exp_{p}$ defines a diffeomorphism $\exp_{p} \colon \mathbb{D}_{r} \to \B_{r}(p)$, where $\mathbb{D}_{r}$ is the disk of radius $r$ in $\R^{2} \simeq \C$. Let $g := N \circ \exp_{p} \colon \mathbb{D}_{r} \to \A_{Q}(\R^d)$, and let $f \in W^{1,2}(\mathbb{D}_{r}, \A_{Q}(\R^d))$ be the ``harmonic extension'' of $\left. g \right|_{r \Sf^{1}}$ already considered in $\S$ \ref{sec:lemma_luckhaus}. In particular, let $\varphi(\theta) := g(re^{i\theta})$, and let $\varphi = \sum_{\ell=1}^{P} \llbracket \varphi_{\ell} \rrbracket$ be an irreducible decomposition of $\varphi$ in maps $\varphi_{\ell} \in W^{1,2}(\Sf^{1}, \A_{Q_\ell}(\R^d))$ such that for some $\gamma_{\ell} \in W^{1,2}(\Sf^{1}, \R^{d})$
\[
\varphi(\theta) = \sum_{\ell=1}^{P} \sum_{m=0}^{Q_\ell-1} \left\llbracket \gamma_{\ell}\left(\frac{\theta+2\pi m}{Q_\ell} \right) \right\rrbracket\,.
\]
Such an irreducible decomposition exists by \cite[Proposition 1.5]{DLS11a}. Then, if the Fourier expansions of the $\gamma_{\ell}$'s are given by
\[
\gamma_{\ell}(\theta) = \frac{a_{\ell,0}}{2} + \sum_{n=1}^{\infty} r^{n} \left( a_{\ell,n} \cos(n \theta) + b_{\ell,n} \sin(n\theta) \right)\,,
\]
we consider their harmonic extensions to $\mathbb{D}_{r}$, namely the functions defined by
\[
\zeta_{\ell}(\rho,\theta) := \frac{a_{\ell,0}}{2} + \sum_{n=1}^{\infty} \rho^{n} \left( a_{\ell,n} \cos(n \theta) + b_{\ell,n} \sin(n\theta) \right) \quad \mbox{for every } 0 < \rho \leq r\,,
\]
and finally we let
\[
f(\rho e^{i\theta}) := \sum_{\ell=1}^{P} \sum_{m=0}^{Q_\ell} \left\llbracket \zeta_{\ell}\left( \rho^{\sfrac{1}{Q_\ell}}, \frac{\theta + 2\pi m}{Q_\ell}\right) \right\rrbracket \quad \mbox{for } \rho e^{i\theta} \in \mathbb{D}_r\,.
\]

Recalling \cite[Lemma 3.12]{DLS11a}, one can explicitly compute the following quantities:
\begin{align} \label{fourier_harmonic1}
\int_{\mathbb{D}_{r}} \abs{Df}^{2} &= \sum_{\ell=1}^{P} \Dir(\zeta_{\ell}, \mathbb{D}_r) = \pi \sum_{\ell=1}^{P} \sum_{n=1}^{\infty} r^{2n} n \left( \abs{a_{\ell,n}}^{2} + \abs{b_{\ell,n}}^{2} \right)\,, \\ \label{fourier_harmonic2}
\int_{r \Sf^{1}} \abs{\partial_{\tau} f}^{2} &= \sum_{\ell=1}^P \Dir(\varphi_\ell, r\Sf^1) = \frac{1}{r} \sum_{\ell=1}^{P} \frac{1}{Q_\ell} \int_{0}^{2\pi} \abs{\gamma_{\ell}'(\alpha)}^{2} \, {\rm d}\alpha = \pi \sum_{\ell=1}^{P}\sum_{n=1}^{\infty} \frac{r^{2n -1}n^2}{Q_\ell} \left( \abs{a_{\ell,n}}^{2} + \abs{b_{\ell,n}}^{2} \right)\,, \\ \label{fourier_harmonic3}
\int_{r\Sf^{1}} \abs{f}^{2} &= r\sum_{\ell=1}^{P} Q_{\ell} \int_{0}^{2\pi} \abs{\gamma_\ell(\alpha)}^{2}\, {\rm d}\alpha = \pi \sum_{\ell=1}^{P} Q_{\ell} \left( \frac{r\abs{a_{\ell,0}}^2}{2} + \sum_{n=1}^{\infty} r^{2n+1} \left( \abs{a_{\ell,n}}^{2} + \abs{b_{\ell,n}}^{2} \right) \right)\,,
\end{align}
where $\partial_{\tau}$ denotes the tangential derivative along $r\Sf^1$.

Now, it is an elementary fact (cf. \cite[proof of Proposition 5.2]{DLS11a}) that for any $\mu > 0$ there exists $\beta_{0} = \beta_{0}(\mu) > 0$ such that for every $0 < \beta < \beta_{0}$ it holds
\begin{equation} \label{fatto_algebrico}
(2\mu + \beta) n \leq \frac{n^2}{Q_\ell} + \mu (\mu + \beta) Q_{\ell} \quad \mbox{for every $n \in \Na$ and for every $Q_\ell$}\,.
\end{equation}

Multiplying \eqref{fatto_algebrico} by $\pi r^{2n} \left( \abs{a_{\ell,n}}^2 + \abs{b_{\ell,n}}^{2} \right)$ and summing over $n$ and $\ell$, we obtain from \eqref{fourier_harmonic1}, \eqref{fourier_harmonic2}, and \eqref{fourier_harmonic3} that for every $\mu > 0$ there exists $\beta_0 > 0$ such that for every $0 < \beta < \beta_0$
\begin{equation} \label{stima_harmonic_competitor}
(2\mu + \beta) \int_{\mathbb{D}_r} \abs{Df}^{2} \leq r \int_{r \Sf^{1}} \abs{\partial_{\tau} g}^{2} + \frac{\mu (\mu + \beta)}{r} \int_{r\Sf^{1}} \abs{g}^{2}\,.
\end{equation}

Now, let $u := f \circ \exp_{p}^{-1} \colon \B_{r}(p) \to \A_{Q}(\R^{d})$, so that the orthogonal projection $u^{\perp}$ is a map in $\Gamma_{Q}^{1,2}(\N\B_{r}(p))$ with $\left. u^{\perp} \right|_{r\Sf^{1}} = \left. N \right|_{r \Sf^{1}}$. The minimality of $N$ then implies that
\[
\begin{split}
\Jac(N, \B_{r}(p)) &\leq \Jac(u^{\perp}, \B_{r}(p)) \\ &\leq \Dir(u^{\perp}, \B_{r}(p)) + C_0 \int_{\B_{r}(p)} \abs{u}^{2} \\ & \overset{\eqref{proj}}{\leq} (1+r) \Dir(u,\B_{r}(p)) + \frac{C}{r} \int_{\B_{r}(p)} \abs{u}^{2}\,,
\end{split}
\]
from which in turn follows
\[
{\bf D}(r) \leq (1 + r) \Dir(u, \B_{r}(p)) + \frac{C}{r} \int_{\B_{r}(p)} \abs{u}^{2} + C_0 \int_{\B_{r}(p)} \abs{N}^{2}\,.
\]

Using that the metric of $\Sigma$ in $\B_{r}(p)$ is almost euclidean when $r \to 0$, we conclude that for small $r$'s
\[
{\bf D}(r) \leq (1 + Cr) \left[ (1 + r) \Dir(f, \mathbb{D}_r) + \frac{C}{r} \int_{\mathbb{D}_r} \abs{f}^{2} \right] + C_{0} \int_{\B_{r}(p)} \abs{N}^{2}\,.
\]

Now, by definition of $f$ one has
\begin{equation} \label{comparison_L2_bdry}
\int_{\mathbb{D}_{r}} \abs{f}^{2} \leq C r \int_{r \Sf^{1}} \abs{g}^{2} \leq Cr (1 + Cr) \int_{\partial \B_{r}(p)} \abs{N}^{2}\,.
\end{equation}
Combining \eqref{comparison_L2_bdry} with \eqref{stima_harmonic_competitor}, we deduce that for every $\mu > 0$ there exists $\beta_{0} = \beta_{0}(\mu)$ such that for every $0 < \beta < \beta_{0}$ 
\begin{equation} \label{stima_original_fct1}
{\bf D}(r) \leq (1+Cr) \left[ \frac{r}{2\mu + \beta} \int_{\partial \B_{r}(p)} \abs{\nabla_{\tau}^{\perp}N}^{2} + \frac{\mu(\mu + \beta)}{r(2\mu + \beta)} {\bf H}(r) \right] + C_{\mu} {\bf H}(r) + C_{0} \int_{\B_{r}(p)} \abs{N}^{2}\,. 
\end{equation}

Next, observe that the inner variation formula \eqref{IV:est} together with the Poincar\'e inequality \eqref{Poinc_type:eq} imply that in dimension $m=2$
\[
\Abs{\int_{\partial B_{r}(p)} \abs{\nabla_{\tau}^{\perp} N}^2 - \frac{{\bf D}'(r)}{2}} \leq C r {\bf D}(r)\,,
\]
and thus, since, again by the Poincar\'e inequality
\[
\int_{\B_{r}(p)} \abs{N}^{2} \leq C r^{2} {\bf D}(r)\,,
\]
equation \eqref{stima_original_fct1} reads
\begin{equation} \label{stima_original_fct2}
{\bf D}(r) \leq (1+Cr) \left[ \frac{r}{2(2\mu + \beta)} {\bf D}'(r) + \frac{\mu(\mu + \beta)}{r(2\mu + \beta)} {\bf H}(r) \right] + C_{\mu} {\bf H}(r) + C_{\mu} r^{2} {\bf D}(r)\,.
\end{equation}

Finally, divide by $1+Cr$ and use that ${\bf H}(r) \leq \frac{2}{I_0} r {\bf D}(r)$ for small $r$'s to finally conclude the validity of \eqref{stima_original_fct_final}.
\end{proof}

\begin{proof}[Proof of Proposition \ref{prop:decays}]
Let $N$ and $p$ be as in the statement, and fix a suitably small radius $r_0 > 0$. In particular, take $r_0 < \min\lbrace {\rm inj}(\Sigma), \dist(p,\partial\Omega) \rbrace$ such that the conclusions of Proposition \ref{am_impr_thm} hold. Recall from the aforementioned proposition that there exists $\lambda > 0$ such that the function $r \in \left( 0, r_0 \right) \mapsto e^{\lambda r} {\bf I}(r)$ is monotone non-decreasing. As an immediate corollary we deduce that when $r$ is small enough
\begin{equation} \label{frequency_lower_bound}
{\bf I}(r) - I_0 \geq -Cr\,.
\end{equation}

The goal now is to get the upper bound. In order to do this, first we exploit the variation estimates deduced in Lemma \ref{FVE} to compute again the derivative
\[
\begin{split}
{\bf I}'(r) &= \frac{{\bf D}(r)}{{\bf H}(r)} + \frac{r {\bf D}'(r)}{{\bf H}(r)} - \frac{r {\bf D}(r) {\bf H}'(r)}{{\bf H}(r)^2} \\
&= \frac{r {\bf D}'(r)}{{\bf H}(r)} - 2 \frac{r {\bf D}(r) {\bf E}(r)}{{\bf H}(r)^2} - \frac{r {\bf D}(r)}{{\bf H}(r)^2} \mathcal{E}_{\eqref{H'_est}}(r) \\
&= \frac{r {\bf D}'(r)}{{\bf H}(r)} - 2 \frac{r {\bf D}(r)^2}{{\bf H}(r)^2} - \frac{r {\bf D}(r)}{{\bf H}(r)^2} \left( \mathcal{E}_{\eqref{H'_est}}(r) + 2\mathcal{E}_{\eqref{OV:est}}(r) \right)\,,
\end{split}
\]
where 
\begin{align} \label{errore_tipo1}
&\abs{\mathcal{E}_{\eqref{H'_est}}(r)} \overset{\eqref{H'_est}}{\leq} C_{0} r {\bf H}(r)\,, \\ \label{errore_tipo2}
&\abs{\mathcal{E}_{\eqref{OV:est}}(r)} = \abs{{\bf D}(r) - {\bf E}(r)} \overset{\eqref{OV:est},\eqref{Poinc_type:eq}}{\leq} C_{0} r^{2} {\bf D}(r)\,.
\end{align} 

Now, apply the estimate \eqref{stima_original_fct_final} from the previous lemma with $\mu = I_0$ to deduce that for every $0 < \beta < \beta_0(I_0)$ one has
\[
\frac{r {\bf D}'(r)}{{\bf H}(r)} - 2 \frac{r {\bf D}(r)^2}{{\bf H}(r)^2} \geq \frac{2}{r} \left( I_0 + \beta - {\bf I}(r) \right) \left( {\bf I}(r) - I_0 \right) - 2C_{I_0}(2I_0 + \beta) {\bf I}(r)\,, 
\]
so that, recalling that ${\bf I}(r) \leq C$, we can finally estimate
\begin{equation} \label{bound_below_I_derivative}
{\bf I}'(r) \geq \frac{2}{r} \left( I_0 + \beta - {\bf I}(r) \right) \left( {\bf I}(r) - I_0 \right) - C {\bf I}(r)\,.
\end{equation}

Hence, if we fix $0 < \beta < \beta_0(I_0)$ we easily conclude 
\begin{equation}
\begin{split}
\frac{d}{dr} \left[ \frac{{\bf I}(r) - I_0}{r^{\beta}} \right] &= \frac{{\bf I}'(r)}{r^\beta} - \beta \frac{{\bf I}(r) - I_0}{r^{\beta +1}} \\
&\geq \frac{1}{r^{\beta + 1}} \left( 2I_0 + \beta - 2 {\bf I}(r) \right) \left( {\bf I}(r) - I_0 \right) - \frac{C}{r^{\beta}} \geq - \frac{C}{r^\beta}\,,
\end{split}
\end{equation}
for all radii $0 < r < r_0(\beta)$.

Integrating in $\left[ r, r_0 \right]$ we conclude
\begin{equation}
\frac{{\bf I}(r_0) - I_0}{r_0^\beta} - \frac{{\bf I}(r) - I_0}{r^\beta} \geq - Cr_{0}^{1-\beta}\,,
\end{equation}
that is
\begin{equation} 
{\bf I}(r) - I_0 \leq C r^{\beta}\,.
\end{equation}
This concludes the proof of
\begin{equation} \label{decay_frequency}
\abs{{\bf I}(r) - I_0} \leq C r^\beta
\end{equation}

To get the other estimates, compute
\[
\begin{split}
\frac{d}{dr}\left[ \log\left( \frac{{\bf H}(r)}{r^{2I_0 + 1}} \right) \right] &= \frac{{\bf H}'(r)}{{\bf H}(r)} - \frac{2I_0+1}{r} = \frac{2{\bf E}(r)}{{\bf H}(r)} + \frac{1}{{\bf H}(r)} \mathcal{E}_{\eqref{H'_est}}(r) - \frac{2I_0}{r} \\ &= \frac{2}{r} \left( {\bf I}(r) - I_0 \right) + \frac{1}{{\bf H}(r)} \left( \mathcal{E}_{\eqref{H'_est}}(r) + 2\mathcal{E}_{\eqref{OV:est}}(r) \right)\,,
\end{split}
\]
with $\mathcal{E}_{\eqref{H'_est}}(r)$ and $\mathcal{E}_{\eqref{OV:est}}(r)$ satisfying the same bounds as in \eqref{errore_tipo1}, \eqref{errore_tipo2}. Using that ${\bf I}(r) \leq C$, this allows to conclude that
\begin{equation}
\frac{2}{r} \left( {\bf I}(r) - I_0 \right) - Cr \leq \frac{d}{dr}\left[ \log\left( \frac{{\bf H}(r)}{r^{2I_0 + 1}} \right) \right] \leq \frac{2}{r} \left( {\bf I}(r) - I_0 \right) + Cr\,.
\end{equation}

After applying \eqref{decay_frequency}, integrating on $0 < s < r < r_0$ and taking exponentials, we therefore obtain the estimate
\begin{equation} \label{decay_height}
e^{-C_{\beta} \left( r^{\beta} - s^{\beta} \right)} \leq \frac{{\bf H}(r)}{r^{2I_0+1}} \frac{s^{2I_0 + 1}}{{\bf H}(s)} \leq e^{C_{\beta} \left( r^{\beta} - s^{\beta} \right)}\,.
\end{equation}

In particular, \eqref{decay_height} implies that the map 
\[
r \in \left( 0, r_0 \right) \mapsto \frac{{\bf H}(r) e^{-C_{\beta}r^\beta}}{r^{2I_0+1}}
\]
is monotone non-increasing. In turn, from this immediately follows the existence of the limit
\[
H_{0} := \lim_{r\to0} \frac{{\bf H}(r)}{r^{2I_0+1}}\,.
\]
The rate of convergence
\[
\Abs{\frac{{\bf H}(r)}{r^{2I_0+1}} - H_0} \leq C r^{\beta} \quad \mbox{for $r$ small enough} 
\]
is also a standard consequence of \eqref{decay_height}. 

Finally, we set $D_{0} := I_0 \cdot H_0$ and immediately obtain
\begin{equation} \label{decay_dirichlet}
\Abs{\frac{{\bf D}(r)}{r^{2I_0}} - D_0} = \Abs{{\bf I}(r) \frac{{\bf H}(r)}{r^{2I_0+1}} - I_0 \cdot H_0}  \leq \Abs{{\bf I}(r) - I_0} \frac{{\bf H}(r)}{r^{2I_0+1}} + I_0 \Abs{\frac{{\bf H}(r)}{r^{2I_0+1}} - H_0} \overset{\eqref{decay_frequency}, \eqref{decay_height}}{\leq} C r^\beta\,.
\end{equation}

\end{proof}

\subsection{Uniqueness of the tangent map at collapsed singularities} \label{sec:unicita}

We are now ready to prove Theorem \ref{uniqueness_tg_map_statement}.

\begin{proof}[Proof of Theorem \ref{uniqueness_tg_map_statement}]

Let $N$ and $p$ be as in the statement, and recall the definition of the blow-up maps $N_{r} = N_{p,r}$ given in \eqref{blow-ups} (together with the definitions of the maps ${\bf ex}_{r}$ and $\psi_r = \psi_{p,r}$ used in there). We first remark that by the Poincar\'e inequality \eqref{Poinc_type:eq} and the reverse Poincar\'e inequality \eqref{reverse_Poinc:eq} any convergence result for the maps $N_{r}$ as $r \downarrow 0$ is equivalent to the same result obtained for the maps
\[
\tilde{N}_{r}(y) := \frac{r^{\frac{m-2}{2}} N(\psi_{r}(y))}{\sqrt{{\bf D}(r)}}\,.
\]

Let us assume without loss of generality that $D_0 = 1$. Then, in dimension $m=2$ the decay estimate \eqref{decay_dirichlet} implies that for $r \downarrow 0$
\begin{equation}
\tilde{N}_{r}(z) = r^{-I_0} N(\psi_{r}(z)) \left( 1 + o(r^{\sfrac{\beta}{2}}) \right)\,,
\end{equation}
and therefore in order to show the existence of a uniform limit for the maps $\tilde{N}_{r}$ in $\mathbb{D}_1$ it suffices to show the existence of a uniform limit for the maps $\tilde{u}_{r}(z) := r^{-I_0} N(\psi_{r}(z))$. Furthermore, if we write $z = \rho e^{i\theta} \in \mathbb{D}_1$ we see immediately that
\[
\tilde{u}_{r}(\rho e^{i\theta}) = r^{-I_0} N(\psi_r(\rho e^{i\theta})) = \rho^{I_0} (\rho r)^{-I_0} N(\psi_{\rho r}(e^{i\theta})) = \rho^{I_0} \tilde{u}_{\rho r}(e^{i\theta})\,,
\]
and thus our goal will be achieved if we show uniform convergence of the maps $\left.\tilde{u}_{r}\right|_{\Sf^1}$. For the sake of notational simplicity we will then remove the tilde, call $w = e^{i\theta}$ the variable on $\Sf^1$ and consider the one-parameter family of maps $u_{r} \colon \Sf^{1} \to \A_{Q}(\R^d)$ given by
\[
u_{r}(w) = r^{-I_0} N(\psi_{r}(w))\,. 
\]

We then fix $\frac{r}{2} \leq s \leq r$ and compute
\begin{equation} \label{L2_cauchy_est1}
\begin{split}
\int_{\Sf^1} \G(u_r, u_s)^{2}\, \dHa^1 &= \int_{\Sf^1} \G\left( \frac{N(\psi_r(w))}{r^{I_0}}, \frac{N(\psi_s(w))}{s^{I_0}} \right)^2\, \dHa^1(w) \\ &\leq \int_{\Sf^1} \sum_{\ell=1}^{Q} \left( \int_{s}^{r} \Abs{\frac{d}{dt} \left( \frac{N^{\ell}(\psi_t(w))}{t^{I_0}} \right)} \, {\rm d}t   \right)^{2} \, \dHa^1(w) \\
&\leq (r-s) \int_{\Sf^1} \int_{s}^{r} \sum_{\ell=1}^{Q} \Abs{\frac{d}{dt} \left( \frac{N^{\ell}(\psi_t(w))}{t^{I_0}} \right)}^{2} \, {\rm d}t \, \dHa^{1}(w)\,.
\end{split}
\end{equation}
Note that in the above computation we have used \cite[Proposition 1.2]{DLS11a} and the fact that the map $t \in  \left( s,r \right) \mapsto \frac{N(\psi_t(w))}{t^{I_0}}$ is in $W^{1,2}$ for a.e. $w \in \Sf^1$. 

Now, we have
\[
\frac{d}{dt} \left( \frac{N^{\ell}(\psi_t(w))}{t^{I_0}} \right) = \frac{DN^{\ell}(\psi_t(w)) \cdot {\rm d} \left.\exp_{p}\right|_{tw}(w)}{t^{I_0}} - I_{0} \frac{N^{\ell}(\psi_t(w))}{t^{I_0+1}}\,,
\]
and thus
\[
\begin{split}
\Abs{\frac{d}{dt} \left( \frac{N^{\ell}(\psi_t(w))}{t^{I_0}} \right)}^{2} \leq &\frac{\abs{\nabla_{\hat{r}}^{\perp} N^{\ell}(\psi_t(w))}^2}{t^{2I_0}} + I_{0}^2 \frac{\abs{N^\ell(\psi_t(w))}^2}{t^{2I_0+2}} - 2I_0 \frac{\langle \nabla_{\hat{r}}^{\perp} N^\ell(\psi_t(w)), N^\ell(\psi_t(w)) }{t^{2I_0+1}}\\ &+ {\rm Err}\,,
\end{split}
\]
where
\[
{\rm Err} \leq C t^{1-2I_0} \abs{\nabla_{\hat{r}}^{\perp}N^\ell(\psi_t(w))}^{2} + C t^{-2I_0} \abs{N^\ell(\psi_t(w))}^{2}\,
\]
for small $t$.

Inserting in \eqref{L2_cauchy_est1} and changing variable $x = \psi_t(w)$ we easily obtain from the variation estimates in Lemma \ref{FVE}:
\begin{equation} \label{stima_L2_diadica}
\begin{split}
\int_{\Sf^1} \G(u_r,u_s)^2 \dHa^1 \leq &(r-s)(1+Cr) \int_{s}^{r} \frac{{\bf G}(t)}{t^{2I_0+1}} + I_0^2 \frac{{\bf H}(t)}{t^{2I_0+3}} - 2I_0 \frac{{\bf E}(t)}{t^{2I_0+2}} \, {\rm d}t \\ &+ C (r-s) \int_{s}^{r} \frac{{\bf G}(t)}{t^{2I_0}} + \frac{{\bf H}(t)}{t^{2I_0+1}} \, {\rm d}t \\
= & \underbrace{(r-s)(1+Cr) \int_{s}^{r} \frac{{\bf D}'(t)}{2t^{2I_0+1}} + I_0^2 \frac{{\bf H}(t)}{t^{2I_0+3}} - 2I_0 \frac{{\bf D}(t)}{t^{2I_0+2}} \, {\rm d}t}_{=:A} \\ &+\underbrace{C (r-s) \int_{s}^{r} \frac{{\bf G}(t)}{t^{2I_0}} + \frac{{\bf H}(t)}{t^{2I_0+1}} \, {\rm d}t}_{=:E_1} \\ &+ \underbrace{(r-s)(1+Cr) \int_{s}^{r} \frac{\mathcal{E}_{\eqref{IV:est}}(t)}{t^{2I_0+1}} + \frac{\mathcal{E}_{\eqref{OV:est}}(t)}{t^{2I_0+2}}\,{\rm d}t}_{=:E_2}\,.
\end{split}
\end{equation}

Now, we have
\begin{equation}
\begin{split}
A &= (r-s)(1+Cr) \int_{s}^{r} \frac{1}{2t} \left( \frac{{\bf D}(t)}{t^{2I_0}} \right)' + I_0^2 \frac{{\bf H}(t)}{t^{2I_0+3}} - I_0 \frac{{\bf D}(t)}{t^{2I_0+2}} \, {\rm d}t \\
&= (r-s)(1+Cr) \int_{s}^{r} \frac{1}{2t} \left( \frac{{\bf D}(t)}{t^{2I_0}} \right)' + I_0 \frac{{\bf H}(t)}{t^{2I_0+3}} \left( I_0 - {\bf I}(t) \right) \, {\rm d}t \,,
\end{split}
\end{equation}
so that, for $s = \frac{r}{2}$
\begin{equation}\label{termine_principale}
A \leq C \Abs{\frac{{\bf D}(r)}{r^{2I_0}} - \frac{{\bf D}(\sfrac{r}{2})}{(\sfrac{r}{2})^{2I_0}}} + C \int_{\sfrac{r}{2}}^{r} \frac{I_0 - {\bf I}(t)}{t} \, {\rm d}t \overset{\eqref{decay_dirichlet}, \eqref{decay_frequency}}{\leq} C r^{\beta}\,.
\end{equation}

For what concerns the error terms, we can use the variation estimates \eqref{OV:est} and \eqref{IV:est} together with the Poincar\'e inequality \eqref{Poinc_type:eq} to control
\begin{equation}\label{primo_errore}
\abs{E_2} \leq C \frac{r}{2} \int_{\sfrac{r}{2}}^{r} \frac{{\bf D}(t)}{t^{2I_0}} \, {\rm d}t \overset{\eqref{decay_dirichlet}}{\leq} C r^2\,,
\end{equation} 
and
\begin{equation}\label{secondo_errore}
\begin{split}
\abs{E_1} &\leq C \frac{r}{2} \int_{\sfrac{r}{2}}^{r} \frac{{\bf D}'(t)}{t^{2I_0}} + \frac{\abs{\mathcal{E}_{\eqref{IV:est}}(t)}}{t^{2I_0}} + \frac{{\bf H}(t)}{t^{2I_0+1}} \, {\rm d}t \\
&\leq C r^{1-2I_0} \abs{{\bf D}(r) - {\bf D}(\sfrac{r}{2})} + C r^{2} \int_{\sfrac{r}{2}}^{r} \frac{{\bf D}(t)}{t^{2I_0}} \, {\rm d}t + C r \int_{\sfrac{r}{2}}^{r} \frac{{\bf H}(t)}{t^{2I_0+1}} \, {\rm d}t \\
&\overset{\eqref{decay_dirichlet}, \eqref{decay_height}}{\leq} C r^{1+\beta}\,. 
\end{split}
\end{equation}

Plugging \eqref{termine_principale}, \eqref{primo_errore} and \eqref{secondo_errore} in \eqref{stima_L2_diadica} we conclude that
\begin{equation}\label{unico_profilo_finale}
\int_{\Sf^1} \G(u_{r}, u_{\frac{r}{2}})^{2} \, \dHa^1 \leq C r^{\beta}\,.
\end{equation}

With an elementary dyadic argument analogous to \cite[proof of Theorem 5.3]{DLS11a}, we conclude that the family $u_r$ is $L^{2}$-Cauchy. Since the $u_{r}$'s are equi-H\"older (cf. \eqref{unif_Holder}), this suffice to conclude uniform convergence to a unique limit.

\end{proof}

\bibliographystyle{aomalpha}
\bibliography{Jacobi}

\providecommand{\bysame}{\leavevmode\hbox to3em{\hrulefill}\thinspace}
\providecommand{\noopsort}[1]{}
\providecommand{\mr}[1]{\href{http://www.ams.org/mathscinet-getitem?mr=#1}{MR~#1}}
\providecommand{\zbl}[1]{\href{http://www.zentralblatt-math.org/zmath/en/search/?q=an:#1}{Zbl~#1}}
\providecommand{\jfm}[1]{\href{http://www.emis.de/cgi-bin/JFM-item?#1}{JFM~#1}}
\providecommand{\arxiv}[1]{\href{http://www.arxiv.org/abs/#1}{arXiv~#1}}
\providecommand{\doi}[1]{\url{https://doi.org/#1}}
\providecommand{\MR}{\relax\ifhmode\unskip\space\fi MR }
\providecommand{\MRhref}[2]{%
  \href{http://www.ams.org/mathscinet-getitem?mr=#1}{#2}
}
\providecommand{\href}[2]{#2}
\begin{thebibliography}{BDGG69}

\bibitem[ope86]{openGMT}
Some open problems in geometric measure theory and its applications suggested
  by participants of the 1984 {AMS} summer institute,  in \emph{Geometric
  measure theory and the calculus of variations ({A}rcata, {C}alif., 1984)}
  (\bgroup\scshape{}J.~E. Brothers\egroup{}, ed.), \emph{Proc. Sympos. Pure
  Math.} \textbf{44}, Amer. Math. Soc., Providence, RI, 1986, pp.~441--464.
  \mr{840292}.  \doi{10.1090/pspum/044/840292}.

\bibitem[All72]{Allard72}
\bgroup\scshape{}W.~K. Allard\egroup{}, On the first variation of a varifold,
  \emph{Ann. of Math. (2)} \textbf{95} (1972), 417--491. \mr{0307015 (46
  \#6136)}.

\bibitem[AA81]{AA81}
\bgroup\scshape{}W.~K. Allard\egroup{} and \bgroup\scshape{}F.~J. Almgren,
  Jr.\egroup{}, On the radial behavior of minimal surfaces and the uniqueness
  of their tangent cones,  \emph{Ann. of Math. (2)} \textbf{113} no.~2 (1981),
  215--265. \mr{607893}.  \doi{10.2307/2006984}.

\bibitem[Alm00]{Almgren00}
\bgroup\scshape{}F.~J. Almgren, Jr.\egroup{}, \emph{Almgren's big regularity
  paper}, \emph{World Scientific Monograph Series in Mathematics} \textbf{1},
  World Scientific Publishing Co., Inc., River Edge, NJ, 2000, $Q$-valued
  functions minimizing Dirichlet's integral and the regularity of
  area-minimizing rectifiable currents up to codimension 2, With a preface by
  Jean E. Taylor and Vladimir Scheffer. \mr{1777737 (2003d:49001)}.

\bibitem[Amb90]{Ambrosio90}
\bgroup\scshape{}L.~Ambrosio\egroup{}, Metric space valued functions of bounded
  variation,  \emph{Ann. Scuola Norm. Sup. Pisa Cl. Sci. (4)} \textbf{17} no.~3
  (1990), 439--478. \mr{1079985 (92d:26022)}.  Available at
  \url{http://www.numdam.org/item?id=ASNSP_1990_4_17_3_439_0}.

\bibitem[BDGG69]{BDGG}
\bgroup\scshape{}E.~Bombieri\egroup{}, \bgroup\scshape{}E.~De~Giorgi\egroup{},
  and \bgroup\scshape{}E.~Giusti\egroup{}, Minimal cones and the {B}ernstein
  problem,  \emph{Invent. Math.} \textbf{7} (1969), 243--268. \mr{0250205}.
  \doi{10.1007/BF01404309}.

\bibitem[dC92]{doC92}
\bgroup\scshape{}M.~P. do~Carmo\egroup{}, \emph{Riemannian geometry},
  \emph{Mathematics: Theory \& Applications}, Birkh{\"a}user Boston, Inc.,
  Boston, MA, 1992, Translated from the second Portuguese edition by Francis
  Flaherty. \mr{1138207 (92i:53001)}.  \doi{10.1007/978-1-4757-2201-7}.

\bibitem[DL13]{DeLellis2013}
\bgroup\scshape{}C.~De~Lellis\egroup{}, \emph{{Errata to "Q-valued functions
  revisited"}}, Available at
  \url{http://www.math.uzh.ch/fileadmin/user/delellis/publikation/Errata_memo.pdf},
  October 2013.

\bibitem[DL16a]{DL16}
\bgroup\scshape{}C.~De~Lellis\egroup{}, The regularity of minimal surfaces in
  higher codimension,  in \emph{Current developments in mathematics 2014}, Int.
  Press, Somerville, MA, 2016, pp.~153--229. \mr{3468252}.

\bibitem[DL16b]{DeLellis2015}
\bgroup\scshape{}C.~De~Lellis\egroup{}, The size of the singular set of
  area-minimizing currents,  in \emph{Surveys in differential geometry 2016.
  {A}dvances in geometry and mathematical physics}, \emph{Surv. Differ. Geom.}
  \textbf{21}, Int. Press, Somerville, MA, 2016, pp.~1--83. \mr{3525093}.

\bibitem[DLFS11]{DLFS11}
\bgroup\scshape{}C.~De~Lellis\egroup{}, \bgroup\scshape{}M.~Focardi\egroup{},
  and \bgroup\scshape{}E.~N. Spadaro\egroup{}, Lower semicontinuous functionals
  for {A}lmgren's multiple valued functions,  \emph{Ann. Acad. Sci. Fenn.
  Math.} \textbf{36} no.~2 (2011), 393--410. \mr{2757522 (2012h:49019)}.
  \doi{10.5186/aasfm.2011.3626}.

\bibitem[DMSV16]{DLMSV16}
\bgroup\scshape{}C.~{De Lellis}\egroup{},
  \bgroup\scshape{}A.~{Marchese}\egroup{},
  \bgroup\scshape{}E.~{Spadaro}\egroup{}, and
  \bgroup\scshape{}D.~{Valtorta}\egroup{}, {Rectifiability and upper Minkowski
  bounds for singularities of harmonic Q-valued maps},  \emph{ArXiv:1612.01813}
  (2016). Available at \url{https://arxiv.org/abs/1612.01813}.

\bibitem[DLS14]{DLS14}
\bgroup\scshape{}C.~De~Lellis\egroup{} and
  \bgroup\scshape{}E.~Spadaro\egroup{}, Regularity of area minimizing currents
  {I}: gradient {$L^p$} estimates,  \emph{Geom. Funct. Anal.} \textbf{24} no.~6
  (2014), 1831--1884. \mr{3283929}.  \doi{10.1007/s00039-014-0306-3}.

\bibitem[DS15]{DLS13a}
\bgroup\scshape{}C.~{De Lellis}\egroup{} and
  \bgroup\scshape{}E.~{Spadaro}\egroup{}, {Multiple valued functions and
  integral currents},  \emph{Ann. Sc. Norm. Super. Pisa Cl. Sci. (5)}
  \textbf{XIV} (2015), 1239--1269.

\bibitem[DLS16a]{DLS13b}
\bgroup\scshape{}C.~De~Lellis\egroup{} and
  \bgroup\scshape{}E.~Spadaro\egroup{}, Regularity of area minimizing currents
  {II}: center manifold,  \emph{Ann. of Math. (2)} \textbf{183} no.~2 (2016),
  499--575. \mr{3450482}.  \doi{10.4007/annals.2016.183.2.2}.

\bibitem[DLS16b]{DLS13c}
\bgroup\scshape{}C.~De~Lellis\egroup{} and
  \bgroup\scshape{}E.~Spadaro\egroup{}, Regularity of area minimizing currents
  {III}: blow-up,  \emph{Ann. of Math. (2)} \textbf{183} no.~2 (2016),
  577--617. \mr{3450483}.  \doi{10.4007/annals.2016.183.2.3}.

\bibitem[DSS15a]{DLSS15a}
\bgroup\scshape{}C.~{De Lellis}\egroup{},
  \bgroup\scshape{}E.~{Spadaro}\egroup{}, and
  \bgroup\scshape{}L.~{Spolaor}\egroup{}, {Regularity theory for
  $2$-dimensional almost minimal currents I: Lipschitz approximation},
  \emph{ArXiv:1508.05507} (2015). Available at
  \url{https://arxiv.org/abs/1508.05507}.

\bibitem[DSS15b]{DLSS15b}
\bgroup\scshape{}C.~{De Lellis}\egroup{},
  \bgroup\scshape{}E.~{Spadaro}\egroup{}, and
  \bgroup\scshape{}L.~{Spolaor}\egroup{}, {Regularity theory for
  $2$-dimensional almost minimal currents II: branched center manifold},
  \emph{ArXiv:1508.05509} (2015). Available at
  \url{https://arxiv.org/abs/1508.05509}.

\bibitem[DSS15c]{DLSS15c}
\bgroup\scshape{}C.~{De Lellis}\egroup{},
  \bgroup\scshape{}E.~{Spadaro}\egroup{}, and
  \bgroup\scshape{}L.~{Spolaor}\egroup{}, {Regularity theory for
  $2$-dimensional almost minimal currents III: blowup},
  \emph{ArXiv:1508.05510} (2015). Available at
  \url{https://arxiv.org/abs/1508.05510}.

\bibitem[DLS11]{DLS11a}
\bgroup\scshape{}C.~De~Lellis\egroup{} and \bgroup\scshape{}E.~N.
  Spadaro\egroup{}, {$Q$}-valued functions revisited,  \emph{Mem. Amer. Math.
  Soc.} \textbf{211} no.~991 (2011), vi+79. \mr{2663735 (2012k:49112)}.
  \doi{10.1090/S0065-9266-10-00607-1}.

\bibitem[Fed65]{Fed65}
\bgroup\scshape{}H.~Federer\egroup{}, Some theorems on integral currents,
  \emph{Trans. Amer. Math. Soc.} \textbf{117} (1965), 43--67. \mr{0168727}.

\bibitem[Fed69]{Federer69}
\bgroup\scshape{}H.~Federer\egroup{}, \emph{Geometric measure theory},
  \emph{Die Grundlehren der mathematischen Wissenschaften, Band 153},
  Springer-Verlag New York Inc., New York, 1969. \mr{0257325 (41 \#1976)}.

\bibitem[GMS98]{GMS98}
\bgroup\scshape{}M.~Giaquinta\egroup{}, \bgroup\scshape{}G.~Modica\egroup{},
  and \bgroup\scshape{}J.~r. Sou\v{c}ek\egroup{}, \emph{Cartesian currents in
  the calculus of variations. {I}}, \emph{Ergebnisse der Mathematik und ihrer
  Grenzgebiete. 3. Folge. A Series of Modern Surveys in Mathematics [Results in
  Mathematics and Related Areas. 3rd Series. A Series of Modern Surveys in
  Mathematics]} \textbf{37}, Springer-Verlag, Berlin, 1998, Cartesian currents.
  \mr{1645086}.  \doi{10.1007/978-3-662-06218-0}.

\bibitem[Han05]{Hang05}
\bgroup\scshape{}F.~Hang\egroup{}, On the weak limits of smooth maps for the
  {D}irichlet energy between manifolds,  \emph{Comm. Anal. Geom.} \textbf{13}
  no.~5 (2005), 929--938. \mr{2216146 (2007b:58017)}.  Available at
  \url{http://projecteuclid.org/euclid.cag/1144438301}.

\bibitem[Hir16a]{Hir14}
\bgroup\scshape{}J.~Hirsch\egroup{}, Boundary regularity of {D}irichlet
  minimizing {$Q$}-valued functions,  \emph{Ann. Sc. Norm. Super. Pisa Cl. Sci.
  (5)} \textbf{16} no.~4 (2016), 1353--1407. \mr{3616337}.

\bibitem[Hir16b]{Hir16}
\bgroup\scshape{}J.~Hirsch\egroup{}, Partial {H}{\"o}lder continuity for
  {$Q$}-valued energy minimizing maps,  \emph{Comm. Partial Differential
  Equations} \textbf{41} no.~9 (2016), 1347--1378. \mr{3551461}.
  \doi{10.1080/03605302.2016.1204313}.

\bibitem[HSV17]{HSV}
\bgroup\scshape{}J.~{Hirsch}\egroup{}, \bgroup\scshape{}S.~{Stuvard}\egroup{},
  and \bgroup\scshape{}D.~{Valtorta}\egroup{}, Rectifiability of the singular
  set of multiple valued energy minimizing harmonic maps,
  \emph{ArXiv:1708.02116} (2017). Available at
  \url{https://arxiv.org/abs/1708.02116}.

\bibitem[KP08]{KP08}
\bgroup\scshape{}S.~G. Krantz\egroup{} and \bgroup\scshape{}H.~R.
  Parks\egroup{}, \emph{Geometric integration theory}, \emph{Cornerstones},
  Birkh{\"a}user Boston, Inc., Boston, MA, 2008. \mr{2427002}.
  \doi{10.1007/978-0-8176-4679-0}.

\bibitem[KW13]{KW13}
\bgroup\scshape{}B.~{Krummel}\egroup{} and
  \bgroup\scshape{}N.~{Wickramasekera}\egroup{}, {Fine properties of branch
  point singularities: Two-valued harmonic functions},  \emph{ArXiv:1311.0923}
  (2013). Available at \url{https://arxiv.org/abs/1311.0923}.

\bibitem[Lee97]{Lee97}
\bgroup\scshape{}J.~M. Lee\egroup{}, \emph{Riemannian manifolds},
  \emph{Graduate Texts in Mathematics} \textbf{176}, Springer-Verlag, New York,
  1997, An introduction to curvature. \mr{1468735 (98d:53001)}.
  \doi{10.1007/b98852}.

\bibitem[Luc88]{Luckhaus88}
\bgroup\scshape{}S.~Luckhaus\egroup{}, Partial {H}{\"o}lder continuity for
  minima of certain energies among maps into a {R}iemannian manifold,
  \emph{Indiana Univ. Math. J.} \textbf{37} no.~2 (1988), 349--367. \mr{963506
  (89m:58043)}.  \doi{10.1512/iumj.1988.37.37017}.

\bibitem[Mor82]{Morgan82}
\bgroup\scshape{}F.~Morgan\egroup{}, On the singular structure of
  two-dimensional area minimizing surfaces in {${\bf R}^{n}$},  \emph{Math.
  Ann.} \textbf{261} no.~1 (1982), 101--110. \mr{675210}.
  \doi{10.1007/BF01456413}.

\bibitem[Mos05]{Moser05}
\bgroup\scshape{}R.~Moser\egroup{}, \emph{Partial regularity for harmonic maps
  and related problems}, World Scientific Publishing Co. Pte. Ltd., Hackensack,
  NJ, 2005. \mr{2155901}.  \doi{10.1142/9789812701312}.

\bibitem[Res97]{Res1}
\bgroup\scshape{}Y.~G. Reshetnyak\egroup{}, Sobolev classes of functions with
  values in a metric space,  \emph{Sibirsk. Mat. Zh.} \textbf{38} no.~3 (1997),
  657--675, iii--iv. \mr{1457485 (98h:46031)}.  \doi{10.1007/BF02683844}.

\bibitem[Res04]{Res2}
\bgroup\scshape{}Y.~G. Reshetnyak\egroup{}, Sobolev classes of functions with
  values in a metric space. {II},  \emph{Sibirsk. Mat. Zh.} \textbf{45} no.~4
  (2004), 855--870. \mr{2091651 (2005e:46055)}.
  \doi{10.1023/B:SIMJ.0000035834.03736.b6}.

\bibitem[Res06]{Res3}
\bgroup\scshape{}Y.~G. Reshetnyak\egroup{}, On the theory of {S}obolev classes
  of functions with values in a metric space,  \emph{Sibirsk. Mat. Zh.}
  \textbf{47} no.~1 (2006), 146--168. \mr{2215302 (2007e:46027)}.
  \doi{10.1007/s11202-006-0013-x}.

\bibitem[Sim83a]{Sim83}
\bgroup\scshape{}L.~Simon\egroup{}, Asymptotics for a class of nonlinear
  evolution equations, with applications to geometric problems,  \emph{Ann. of
  Math. (2)} \textbf{118} no.~3 (1983), 525--571. \mr{727703}.
  \doi{10.2307/2006981}.

\bibitem[Sim83b]{Simon83}
\bgroup\scshape{}L.~Simon\egroup{}, \emph{Lectures on geometric measure
  theory}, \emph{Proceedings of the Centre for Mathematical Analysis,
  Australian National University} \textbf{3}, Australian National University,
  Centre for Mathematical Analysis, Canberra, 1983. \mr{756417 (87a:49001)}.

\bibitem[Sim94]{Sim94}
\bgroup\scshape{}L.~Simon\egroup{}, Uniqueness of some cylindrical tangent
  cones,  \emph{Comm. Anal. Geom.} \textbf{2} no.~1 (1994), 1--33.
  \mr{1312675}.  \doi{10.4310/CAG.1994.v2.n1.a1}.

\bibitem[Sim96]{Simon96}
\bgroup\scshape{}L.~Simon\egroup{}, \emph{Theorems on regularity and
  singularity of energy minimizing maps}, \emph{Lectures in Mathematics ETH
  Z{\"u}rich}, Birkh{\"a}user Verlag, Basel, 1996, Based on lecture notes by
  Norbert Hungerb{{\"u}}hler. \mr{1399562 (98c:58042)}.
  \doi{10.1007/978-3-0348-9193-6}.

\bibitem[Sim68]{Simons68}
\bgroup\scshape{}J.~Simons\egroup{}, Minimal varieties in riemannian manifolds,
   \emph{Ann. of Math. (2)} \textbf{88} (1968), 62--105. \mr{0233295 (38
  \#1617)}.

\bibitem[{Stu}17]{SS17a}
\bgroup\scshape{}S.~{Stuvard}\egroup{}, {Multiple valued sections of vector
  bundles: the reparametrization theorem for $Q$-valued functions revisited},
  \emph{ArXiv:1705.00054} (2017). Available at
  \url{https://arxiv.org/abs/1705.00054}.

\bibitem[Whi83]{BW83}
\bgroup\scshape{}B.~White\egroup{}, Tangent cones to two-dimensional
  area-minimizing integral currents are unique,  \emph{Duke Math. J.}
  \textbf{50} no.~1 (1983), 143--160. \mr{700134}.
  \doi{10.1215/S0012-7094-83-05005-6}.

\bibitem[Whi88]{White88}
\bgroup\scshape{}B.~White\egroup{}, Homotopy classes in {S}obolev spaces and
  the existence of energy minimizing maps,  \emph{Acta Math.} \textbf{160}
  no.~1-2 (1988), 1--17. \mr{926523 (89a:58031)}.  \doi{10.1007/BF02392271}.

\end{thebibliography}

\Addresses

\end{document}